\documentclass[11pt]{amsart}

\usepackage{mathrsfs}
\usepackage{graphicx}
\usepackage{amssymb}
\usepackage{amsmath,amscd}
\usepackage{amsthm}
\usepackage{url,verbatim}
\usepackage{enumitem}
\usepackage{xcolor}
\usepackage[colorlinks,citecolor=blue,urlcolor=blue]{hyperref}

\makeatletter
\newcommand\xleftrightarrow[2][]{%
  \ext@arrow 9999{\longleftrightarrowfill@}{#1}{#2}}
\newcommand\longleftrightarrowfill@{%
  \arrowfill@\leftarrow\relbar\rightarrow}
\makeatother

\theoremstyle{plain}

\newtheorem{theorem}{Theorem}
\newtheorem{definition}[theorem]{Definition}
\newtheorem{lemma}[theorem]{Lemma}
\newtheorem{proposition}[theorem]{Proposition}
\newtheorem{corollary}[theorem]{Corollary}

\newtheorem{assumption}[theorem]{Assumption}
\newtheorem{remark}[theorem]{Remark}

\newcommand\EE{{\mathbb E}}

\newcommand\RR{{\mathbb R}}
\newcommand\ZZ{{\mathbb Z}}
\newcommand\NN{{\mathbb N}}
\newcommand\QQ{{\mathbb Q}}

\newcommand\YY{{\mathbb {Y}}}

\renewcommand\ell{l}

\newcommand\CC{\mathbb{C}}

\newcommand\bA{\mathbf{A}}
\newcommand\bB{\mathbf{B}}
\newcommand\MP{\mathbb{MP}}

\newcounter{mycount}

\numberwithin{equation}{section}
\numberwithin{theorem}{section}
\numberwithin{figure}{section}

\title[Doubly Free-Boundary Macdonald Processes]
{Doubly Free-Boundary Macdonald Processes:
Reflection Identities and Jack Asymptotics}
\thanks{Corresponding author: Zhongyang Li.}

\author[Z. Li]{Zhongyang Li}
\address{Department of Mathematics,
University of Connecticut,
Storrs, Connecticut 06269-3009, USA}
\email{zhongyang.li@uconn.edu}
\urladdr{\url{https://mathzhongyangli.wordpress.com}}

\author[K. Shi]{Kaili Shi}
\address{Department of Mathematics and Statistics,
Boston University,
Boston, Massachusetts, USA}
\email{shkl@bu.edu}


\begin{document}

\begin{abstract}
We introduce a doubly free-boundary Macdonald process on rail-yard
interlacings and develop a reflection calculus for its observables.
Boundary Cauchy--Littlewood identities, combined with Negu\c t
operators, yield exact multipoint contour formulas for arbitrary
\(L/R\) words.  Under the Jack scaling
\[
        q=t^\alpha,\qquad t=e^{-n\beta\epsilon},
\]
and piecewise-periodic data, these formulas imply a
Laplace-transform law of large numbers and a weak slope-measure
limit shape at \(L\)-type columns.

For arbitrary piecewise-periodic \(L/R\) backgrounds and finitely
many \(L\)-type marked columns, under the stated contour, branch,
and normal-convergence hypotheses, the centered height-Laplace
observables converge jointly to a Gaussian vector.  Its covariance
exhibits a boundary--deformation separation: the Jack parameter
and the microscopic rail-yard data enter through the one-point
spectral factors and the normalization, whereas the two-point
interaction is the logarithmic derivative of an annular prime
function generated by the two boundary reflections.  Thus the
deformation changes the spectral map while preserving the annular
image geometry of the Schur specialization.

For \(\beta=1\), in the all-\(L\) sector and under explicit signed
zero--pole and root-localization hypotheses, we characterize regular
liquid and frozen points through the nonreal-root structure of the
characteristic equation and show that nondegenerate regular interfaces lie on the real
double-root locus. The half-space Macdonald-process formulas are recovered in the
continuous degeneration \(v\downarrow0\), which forces the right
boundary partition to be empty.
\end{abstract}

\maketitle

\setcounter{tocdepth}{2}
\tableofcontents

\section{Introduction}

Random tiling and dimer models provide a basic class of exactly solvable probability
measures in which deterministic limit shapes coexist with Gaussian height fluctuations.  In
many Schur and determinantal models, the analysis is ultimately controlled by correlation
kernels and their steepest-descent asymptotics.  Macdonald and Jack deformations usually
remove this determinantal structure.  This paper asks what remains of the limit-shape and
Gaussian-fluctuation theory when a Jack deformation is combined with two free boundaries in
rail-yard geometry.

The answer is that the two free boundaries affect both orders of the asymptotic theory.
At first order they enter the Laplace-transform law of large numbers through reflected
one-point factors and hence modify the deterministic limit shape.  More importantly,
they also change the image structure of the second-order field.  For the height
Laplace-test observables considered here, the limiting covariance is an annular
reflected-image kernel.  This kernel is produced by repeated reflection of Cauchy factors
between the two free boundaries; it is absent in the empty-boundary Schur degeneration and
survives in the Jack scaling limit.

\subsection{The boundary mechanism}

Rail-yard graphs form a flexible class of bipartite graphs whose perfect matchings are
encoded by interlacing sequences of partitions.  In Schur degenerations this encoding leads
to Schur processes and free-boundary Schur processes.  The model studied in this paper is
the corresponding Macdonald-deformed measure on partition sequences, represented
geometrically as doubly free-boundary dimer states on rail-yard graphs.  Except in
Schur-type degenerations, this should not be confused with a local product dimer Gibbs
measure: the rail-yard geometry supplies the interlacing dynamics, while the weights are
Macdonald weights with boundary specializations.

The essential new feature is algebraic reflection at the free boundaries.  Ordinary
full-space Macdonald processes are controlled by skew Cauchy commutations, which interchange
two adjacent bulk skew factors and produce a Cauchy factor.  Free boundaries require an
additional step: skew Littlewood identities sum over the boundary partition and turn an
incoming skew factor into a reflected outgoing one, multiplied by a boundary factor.  Iterating
this reflection between the two boundaries produces the infinite one-point products and the
two-boundary covariance correction.  

This places the paper at the interface of two established theories, while the
asymptotic phenomenon studied here is specific to the doubly free-boundary geometry.
Ahn's work on Macdonald plane partitions~\cite{Ah18} treats the full-space
jack-process side: skew plane partitions, equivalently lozenge tilings or
honeycomb dimers, are analyzed by difference-operator asymptotics and yield
deterministic limit shapes and Gaussian free field fluctuations as \(q,t\to1\).

The boundary side is closer in spirit to the half-space Macdonald-process theory of
Barraquand--Borodin--Corwin~\cite{BBC20}.  In that theory, the ordinary Cauchy
summation identity of full-space Macdonald processes is replaced by Littlewood
summation identities, and new exact observables are needed.  The one-free-boundary
degeneration of the present model belongs to this half-space framework.  Here we
analyze the Jack scaling of the doubly free-boundary formulas themselves.  The two
free boundaries generate repeated reflections, producing infinite one-point factors and
a two-boundary covariance correction which survives in the scaling limit.  Thus the main result is not merely a Macdonald-process analogue of
full-space plane-partition asymptotics: it shows that the annular
reflected-image geometry generated by two free boundaries persists
under the non-determinantal Jack deformation.

A useful point of comparison is the Schur-weighted doubly
free-boundary rail-yard model studied in~\cite{zl23}.  In that setting
the measure is the local edge-product dimer measure, and the repeated
boundary reflections already produce infinite reflected products and
an annular covariance for height-Laplace observables.  The present work
concerns a genuinely Macdonald-deformed, non-determinantal measure.  We
develop the corresponding two-boundary Cauchy identities and
multipoint Negu\c t contour formulas, and analyze their Jack limit under
\[
        q=t^\alpha,\qquad t=e^{-n\beta\epsilon}.
\]
The main structural conclusion is a deformation-stability phenomenon:
the Jack parameter and the microscopic \(L/R\) data modify the
one-point spectral factors and the normalization, whereas the
two-boundary image geometry of the limiting covariance remains
annular.  Thus the present results provide a Macdonald reflection
theory for the two-free-boundary geometry, rather than a formal
specialization of the Schur formulas.

\subsection{Exact formulas}

The first contribution is algebraic.  We prove boundary Cauchy identities for Macdonald
polynomials, including the even and conjugate-even boundary specializations used by the free
boundaries.  These identities are combined with Negu\c t operators to obtain exact contour
formulas for exponential height-Laplace observables.  In the present non-determinantal Jack
setting, these contour formulas play the role that correlation kernels play in Schur models.
They are also the point at which the two-boundary mechanism becomes visible: after
normalization, the integrands contain both infinite one-point reflected products and a
non-trivial two-point free-boundary correction.

The formulas are stated for the general \(L/R\) rail-yard process.  They contain the empty-
and one-free-boundary cases as limits.  The half-space formulas arise by taking the
appropriate one-boundary limit of the doubly free-boundary partition and contour identities;
in the contour formulas this limit is taken together with the associated reflected products,
pole separation, and branch choices.

\subsection{Law of large numbers and weak limit shape}

The second contribution is asymptotic.  We work in the exact Jack scaling
\[
        q=t^\alpha,\qquad t=e^{-n\beta\epsilon},
\]
with piecewise periodic rail-yard data.  Under the analytic admissibility and branch
hypotheses stated below, the one-point contour formula gives a Laplace-transform law of
large numbers at \(L\)-type marked columns.  The limiting transform is naturally written in
terms of
\[
        S_\chi(w):=\mathcal G_\chi(w)\prod_{r\ge1}\mathcal F_{u,v,r}(w),
\]
where all non-integer powers are taken with the compatible branch system specified in the
main assumptions.

The deterministic limit shape is formulated in the natural Macdonald coordinate.  We set
\[
        x=e^{-n\beta\kappa},
\]
so that the integer moments of the pushed-forward slope measure are exactly the moments
provided by the Negu\c t contour formula.  With this choice of variable, the transform asymptotics become moment asymptotics for the
pushed-forward slope measures.  An exponential moment bound gives compact moment
determinacy, so these moment asymptotics determine a unique weak limit.  This yields the
deterministic limit shape in the weak slope-measure sense, without requiring analytic
continuation of the Laplace exponent.

\subsection{Annular Gaussian fluctuations}

The main probabilistic result is the fluctuation theorem.  For arbitrary piecewise-periodic \(L/R\) backgrounds, the centered
height-Laplace observables at finitely many \(L\)-type marked columns,
written in the charge-centered \((q,t)\)-column coordinate, converge
in finite-dimensional distributions to a Gaussian vector under the
stated contour, branch, and normal-convergence hypotheses.  The limiting
covariance is not the pullback of the standard upper-half-plane image covariance.  Instead,
the two free boundaries produce an annular logarithmic-derivative kernel.

In the notation used below, the \(L\)-chart covariance kernel is
\[
        \mathsf K_{LL}(z,w)
        =
        \partial_z\partial_w
        \log\frac{\Theta_{\mathfrak q}(z/w)}{\Theta_{\mathfrak q}(u^2zw)},
        \qquad \mathfrak q=(uv)^2.
\]
Equivalently, in nested contour coordinates this is the bulk double-pole term together with
the annular image term and the two reflected-boundary image terms.  The free-boundary
correction is not a lower-order error: it is produced by mixed reflections between the two
boundaries and survives as \(\epsilon\to0\).  Thus the second-order effect of the two free
boundaries is genuinely macroscopic.

The term ``annular Gaussian free field'' is used here in this precise Laplace-test sense.  We
prove convergence of the \(L\)-chart height-Laplace observables and identify their covariance
with the annular prime-function logarithmic-derivative kernel.   

\subsection{Frozen boundary}

The frozen-boundary analysis is an application of the limit-shape theorem and is kept
separate from the Gaussian result.  For the root-count classification we restrict to \(\beta=1\).  In this case
the characteristic equation is
\[
   S_\chi(w)=e^{-n\kappa}.
\]
Under Assumptions~\ref{ap64}--\ref{ap65} and the
nonexceptionality conditions of Theorem~\ref{p412}, regular liquid
points correspond to a unique nonreal conjugate pair, while regular
frozen points correspond to the all-real-root regime.
 Consequently nondegenerate regular interfaces lie on the real double-root locus
\[
\begin{cases}
\mathcal G_\chi(w)\displaystyle\prod_{r\ge1}\mathcal F_{u,v,r}(w)
        = e^{-n\kappa},\\[2mm]
\displaystyle
\frac{d}{dw}\log\!\left(\mathcal G_\chi(w)\prod_{r\ge1}\mathcal F_{u,v,r}(w)\right)=0.
\end{cases}
\]

The conditional hypotheses are kept explicit.
Assumptions~\ref{ap64}--\ref{ap65} provide a concrete separated
all-\(L\) regime for the finite signed zero--pole order, while
Lemmas~\ref{lem:local-radial-noescape} and
\ref{lem:local-polefree-noescape} supply the local pole-free
no-escape input under endpoint nonexceptionality.

In contrast with the mixed-\(L/R\) Schur analysis
of~\cite{zl23}, the regular root classification proved here is
currently restricted to the all-\(L\) sector.  For mixed \(L/R\) Jack
data, the \(R\)-type spectral factors carry the generally non-integral
exponent \(\alpha\), so the finite truncations are branch-dependent
rather than rational and the Schur polynomial root-count argument does
not apply directly.

\subsection{Organization}

Section~\ref{sect:mr} defines the Macdonald-deformed doubly free-boundary rail-yard measure,
introduces the height observables, and states the main theorems.  Section~\ref{sect:pf}
proves the partition-function formula by boundary reflection and commutation identities for
Macdonald polynomials.  Section~\ref{sect:lthf} derives the contour formula for the expectation
of the height-Laplace transform using Negu\c t operators.  Section~\ref{sect:as} performs the
Jack asymptotic analysis of the one-point and multi-point formulas; the infinite reflected
products produce the functions \(\mathcal F_{u,v,r}\), and the two-boundary part produces the
covariance correction.

Section~\ref{sect:fb} proves the Laplace-transform law of large numbers, upgrades it to the
weak slope-measure limit shape in the variable \(x=e^{-n\beta\kappa}\), and derives the
conditional regular frozen-boundary equations.  It also gives explicit separated all-\(L\) sufficient conditions
for the finite signed zero--pole order and proves the local
root-localization input used in the regular root classification. Section~\ref{sect:gff} identifies the fluctuation covariance
with the annular reflected-image kernel and proves the annular Laplace-test Gaussian theorem.
Section~\ref{sect:hsmp} records the one-boundary degeneration corresponding to the
half-space Macdonald-process geometry.  Appendix~\ref{sc:dmp} collects the Macdonald
polynomial identities and technical inputs used in the proof.

\section{Main Results}\label{sect:mr}

In this section, we review rail-yard graphs and dimer coverings, and state the main results proved in this paper.

\subsection{Weighted rail-yard graphs}

The rail-yard graph was introduced in \cite{bbccr} as a broad class of graphs on which random perfect matchings give rise to Schur processes. It is believed to be the most general known family of graphs for which perfect matchings naturally correspond to measures governed by Schur-process-type dynamics. We now recall the definition.

Let $l,r\in\ZZ$ with $l\le r$, and set
\[
[l..r]:=[l,r]\cap\ZZ.
\]
For a positive integer $m$, we also write
\[
[m]:=\{1,2,\ldots,m\}.
\]
Consider two sequences indexed by integers in $[l..r]$:
\begin{itemize}
    \item the $LR$-sequence $a=(a_l,a_{l+1},\ldots,a_r)\in\{L,R\}^{[l..r]}$;
    \item the sign sequence $b=(b_l,b_{l+1},\ldots,b_r)\in\{+,-\}^{[l..r]}$.
\end{itemize}
The \emph{rail-yard graph} $RYG(l,r,a,b)$ associated with $(l,r,a,b)$ is the bipartite graph with vertex set
\[
[2l-1..2r+1]\times\left(\ZZ+\frac12\right).
\]
A vertex is called \emph{even} (resp.\ \emph{odd}) if its abscissa is an even (resp.\ odd) integer. Each even vertex $(2m,y)$, $m\in[l..r]$, is incident to three edges: two horizontal edges joining it to $(2m-1,y)$ and $(2m+1,y)$, and one diagonal edge joining it to
\begin{itemize}
    \item $(2m-1,y+1)$ if $(a_m,b_m)=(L,+)$;
    \item $(2m-1,y-1)$ if $(a_m,b_m)=(L,-)$;
    \item $(2m+1,y+1)$ if $(a_m,b_m)=(R,+)$;
    \item $(2m+1,y-1)$ if $(a_m,b_m)=(R,-)$.
\end{itemize}
See Figure~\ref{fig:rye} for an example.

\begin{figure}
\includegraphics[width=.8\textwidth]{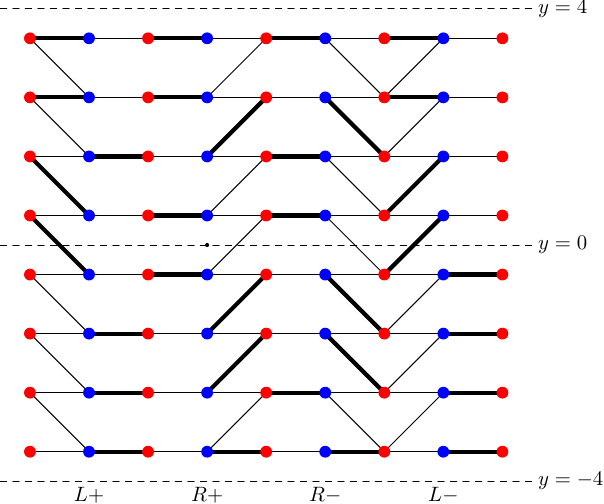}
\caption{A rail-yard graph with $LR$-sequence $a=\{L,R,R,L\}$ and sign sequence $b=\{+,+,-,-\}$. Odd vertices are represented by red points, and even vertices are represented by blue points. Dark edges represent a dimer covering. Assume that above the horizontal line $y=4$, only horizontal edges with an odd vertex on the left are present in the dimer configuration; and below the horizontal line $y=-4$, only horizontal edges with an even vertex on the left are present. The corresponding sequence of partitions (from left to right) is $\emptyset\prec(2,0,\ldots)\prec'(3,1,1,\ldots)\succ'(2,0,\ldots)\succ\emptyset$.}
\label{fig:rye}
\end{figure}

The left boundary (resp.\ right boundary) of $RYG(l,r,a,b)$ consists of all odd vertices with abscissa $2l-1$ (resp.\ $2r+1$). Vertices that do not lie on the boundary are called \emph{inner} vertices. A face of $RYG(l,r,a,b)$ is called an \emph{inner face} if all vertices on its boundary are inner vertices.

We assign edge weights to $RYG(l,r,a,b)$ as follows:
\begin{itemize}
    \item every horizontal edge has weight $1$;
    \item each diagonal edge adjacent to an even vertex with abscissa $2i$ has weight $x_i$.
\end{itemize}
These local weights are parameters entering the Macdonald-deformed measure defined in
Section~\ref{sect:pf}; except in Schur-type degenerations, the probability measure used
below is not merely the product of local occupied-edge weights.

\subsection{Dimer coverings}

A \emph{perfect matching}, or \emph{dimer covering}, of a graph is a subset of edges such that each vertex is incident to exactly one edge; see \cite{RK09,VG21}. For rail-yard graphs with free boundary, we use the following variant.

\begin{definition}[Free-boundary dimer covering]\label{df21}
A \emph{free-boundary dimer covering} of \(RYG(l,r,a,b)\) is a subset \(M\) of edges such that
\begin{enumerate}
    \item each inner vertex is incident to exactly one selected edge;
    \item each left-boundary or right-boundary vertex is incident to at most one selected edge;
    \item only finitely many selected edges are diagonal;
    \item there exists \(N>0\) such that the following two reference phases hold:
    \begin{enumerate}
        \item if \(y>N\), then for every even vertex \((2i,y)\), \(i\in[l..r]\), the selected edge incident to \((2i,y)\) is the horizontal edge joining \((2i-1,y)\) to \((2i,y)\);
        \item if \(y<-N\), then for every even vertex \((2i,y)\), \(i\in[l..r]\), the selected edge incident to \((2i,y)\) is the horizontal edge joining \((2i,y)\) to \((2i+1,y)\).
    \end{enumerate}
\end{enumerate}
Equivalently, sufficiently high in the graph every present edge is the horizontal edge lying to the right of its odd endpoint, while sufficiently low in the graph every present edge is the horizontal edge lying to the left of its odd endpoint.
\end{definition}

\begin{definition}[Partitions, conjugates, and parity]\label{def:partition}
A \emph{partition} is a non-increasing sequence
\[
        \lambda=(\lambda_i)_{i\ge1}
\]
of nonnegative integers that is eventually zero.  Let \(\mathbb Y\) denote the
set of all partitions.  The size of \(\lambda\) is
\[
        |\lambda|=\sum_{i\ge1}\lambda_i .
\]

The \emph{conjugate partition} \(\lambda'\) is the partition whose Young
diagram is obtained from that of \(\lambda\) by reflection across the main
diagonal.  Equivalently,
\[
        \lambda'_i
        :=
        \bigl|\{j\ge1:\lambda_j\ge i\}\bigr|,
        \qquad i\ge1 .
\]

A partition \(\lambda\) is called \emph{even} if all of its parts are even,
that is,
\[
        \lambda_i\in 2\mathbb Z_{\ge0}
        \qquad \text{for all } i\ge1 .
\]
It is called \emph{conjugate-even} if its conjugate partition is even, namely
\[
        \lambda' \text{ is even}.
\]
Equivalently, every column length of the Young diagram of \(\lambda\) is even.

For two partitions \(\lambda,\mu\), we write \(\lambda\succ\mu\), or
\(\mu\prec\lambda\), if
\[
        \lambda_1\ge \mu_1\ge \lambda_2\ge \mu_2\ge \lambda_3\ge\cdots .
\]
Equivalently, \(\lambda/\mu\) is a horizontal strip.  We write
\[
        \lambda\prec'\mu
\]
if
\[
        \lambda'\prec\mu' .
\]
Equivalently, \(\mu/\lambda\) is a vertical strip.
\end{definition}

\begin{definition}[Particle--hole configuration and column partitions]
\label{def:pure-dimer}
Let \(M\) be a dimer covering of \(RYG(l,r,a,b)\) in the sense of
Definition~\ref{df21}.  We associate a particle--hole configuration to the odd vertices as
follows.  If an odd vertex is incident to a selected edge on its right, we put a hole at that
vertex; if it is incident to a selected edge on its left, we put a particle there.  At the two
free boundaries we use the exterior-edge convention: an unmatched left-boundary odd vertex
is regarded as matched to a virtual exterior edge on its left, and is therefore a particle;
an unmatched right-boundary odd vertex is regarded as matched to a virtual exterior edge on
its right, and is therefore a hole.

For \(m\in[l..r+1]\), this gives a Maya diagram on the odd column \(x=2m-1\).  By the
tail condition in Definition~\ref{df21}, this Maya diagram has only holes sufficiently far
above and only particles sufficiently far below.  It therefore has finite charge and determines
a partition
\[
        \lambda^{(M,m)}
        =
        \bigl(\lambda^{(M,m)}_1,\lambda^{(M,m)}_2,\ldots\bigr)
        \in\mathbb Y
\]
as follows.  List the particles on the column \(x=2m-1\) from top to bottom.  For \(i\ge1\),
let \(\lambda^{(M,m)}_i\) be the number of holes strictly below the \(i\)-th highest particle.
The tail condition implies that this sequence is weakly decreasing and eventually zero, so it
is an ordinary Young partition.  We call \(\lambda^{(M,m)}\) the partition associated with the
odd column \(x=2m-1\).

In particular, the left and right boundary columns determine the boundary partitions
\[
        \lambda^{(M,l)},\qquad \lambda^{(M,r+1)}.
\]
When no confusion can arise, we write \(M\) also for the covering together with this induced
particle--hole configuration and its associated column partitions.
\end{definition}

See Figure~\ref{fig:rye} for an example.

For a free-boundary dimer covering $M$ of $RYG(l,r,a,b)$, we define the associated height function $h_M$ on faces of $RYG(l,r,a,b)$ as follows. Definition~\ref{df21}(4) implies that there exists $N>0$ such that when $y<-N$, only horizontal edges with an even vertex on the left are present. Fix a face $f_0$ whose midpoint lies on the horizontal line $y=-N$, and set
\[
\overline h_M(f_0)=0.
\]
For any two adjacent faces $f_1$ and $f_2$ sharing at least one edge, define the increment of $\overline h_M$ by the following rules:
\begin{itemize}
    \item If moving from $f_1$ to $f_2$ crosses a horizontal edge with an odd vertex on the left, then
    \[
    \overline h_M(f_2)-\overline h_M(f_1)=
    \begin{cases}
    1,&\text{if the edge is present in }M,\\
    -1,&\text{if the edge is absent from }M;
    \end{cases}
    \]
    \item If moving from $f_1$ to $f_2$ crosses a diagonal edge with an odd vertex on the left, then
    \[
    \overline h_M(f_2)-\overline h_M(f_1)=
    \begin{cases}
    2,&\text{if the edge is present in }M,\\
    0,&\text{if the edge is absent from }M.
    \end{cases}
    \]
\end{itemize}
These rules specify the increment for oriented dual crossings for which the odd endpoint of the crossed edge lies on the left.  If the same edge is crossed in the opposite direction, or equivalently if the odd endpoint lies on the right for the chosen orientation, the increment is defined to be the negative of the increment assigned to the reversed crossing.  With this antisymmetry convention the preliminary height function is well defined up to the choice of the initial face.  Equivalently, for the vertical lines used below, one may take the explicit formulas \eqref{hm1}--\eqref{hm2} as the definition of the normalized height.

Let $\overline h_0$ be the preliminary height function associated with the dimer configuration in which
\begin{itemize}
    \item no diagonal edge is present; and
    \item every present edge is horizontal with an even vertex on the left.
\end{itemize}
The height function associated with $M$ is then defined by
\begin{equation}\label{dhm}
    h_M=\overline h_M-\overline h_0.
\end{equation}

Let $m\in[l..r]$. The vertical line $x=2m-\frac12$ crosses only edges whose left endpoint is odd. Therefore, for each point $\left(2m-\frac12,y\right)$ lying in a face of $RYG(l,r,a,b)$, we have
\begin{equation}\label{hm1}
    h_M\!\left(2m-\frac12,y\right)=2\left[N_{h,M}^-\!\left(2m-\frac12,y\right)+N_{d,M}^-\!\left(2m-\frac12,y\right)\right],
\end{equation}
where $N_{h,M}^-\left(2m-\frac12,y\right)$ is the number of present horizontal edges in $M$ crossed by $x=2m-\frac12$ below height $y$, and $N_{d,M}^-\left(2m-\frac12,y\right)$ is the number of present diagonal edges in $M$ crossed by the same vertical line below height $y$. Both quantities are finite for each finite $y$.

Similarly, the vertical line $x=2m+\frac12$ crosses only edges whose left endpoint is even. Hence, for each point $\left(2m+\frac12,y\right)$ lying in a face of $RYG(l,r,a,b)$,
\begin{equation}\label{hm2}
    h_M\!\left(2m+\frac12,y\right)=2\left[J_{h,M}^-\!\left(2m+\frac12,y\right)-N_{d,M}^-\!\left(2m+\frac12,y\right)\right],
\end{equation}
where $J_{h,M}^-\left(2m+\frac12,y\right)$ is the number of absent horizontal edges in $M$ crossed by $x=2m+\frac12$ below height $y$, and $N_{d,M}^-\left(2m+\frac12,y\right)$ is the number of present diagonal edges crossed by the same vertical line below height $y$. Again, both quantities are finite for each finite $y$.

\begin{lemma}[Height normalization and finiteness]\label{lem:height-normalization}
For every free-boundary dimer covering $M$, the vertical-line formulas \eqref{hm1}--\eqref{hm2} define the same normalized height as the face-increment construction above.  In particular the normalized height is path-independent on faces.  Moreover, for every fixed column $m$ and every $k>0$, the height-Laplace integrals
\[
\int_{-\infty}^{\infty}
\left|h_M\!\left(2m\pm\frac12,y\right)\right|e^{-ky}\,dy
\]
are finite.  The same assertion holds for the column-wise $(q,t)$-coordinate system of Definition~\ref{def:qt-column-embedding} after replacing $y$ by that coordinate.
\end{lemma}

\begin{proof}
By Definition \ref{df21}(4), below a sufficiently low horizontal line the normalized height is zero.  Above that line, along a fixed vertical line the height changes only when the line crosses a horizontal edge, or an occupied diagonal edge.  There are finitely many occupied diagonal edges, and the horizontal contribution grows at most linearly in the vertical coordinate. It follows that $\int_{-\infty}^{\infty}
\left|h_M\!\left(2m\pm\frac12,y\right)\right|e^{-ky}\,dy$ is finite.
\end{proof}

We use the notation
\[
 i_{\equiv_n}=
 \begin{cases}
 i\bmod n, & \text{if } i\bmod n\neq 0,\\
 n, & \text{otherwise.}
 \end{cases}
\]

\subsection{The Macdonald-deformed doubly free-boundary measure}
\label{subsec:macdonald-dfb-measure}

We now define the probability measure used in the main theorems.  The factorized formula
for its normalizing constant will be proved in Section~\ref{sect:pf}.

Let \(q,t\in(0,1)\), let \(u,v\in(0,1)\), and let
\(c_l,c_r\in\{el,oa,deel,eoa\}\) be the left and right boundary types.  For a partition
\(\lambda\), and for a box \(s=(i,j)\in\lambda\), let
\[
        a_\lambda(s):=\lambda_i-j,
        \qquad
        l_\lambda(s):=\lambda'_j-i
\]
be its arm and leg lengths, respectively; equivalently, \(a_\lambda(s)\) is the number of
boxes of \(\lambda\) strictly to the right of \(s\), and \(l_\lambda(s)\) is the number of boxes
strictly below \(s\).

Set
\[
 b_\lambda(s;q,t)
 =
 \frac{1-q^{a_\lambda(s)}t^{l_\lambda(s)+1}}
      {1-q^{a_\lambda(s)+1}t^{l_\lambda(s)}} .
\]
Define
\[
 b_\lambda(q,t):=\prod_{s\in\lambda}b_\lambda(s;q,t),
\]
and
\[
 b_{\lambda}^{el}
 :=
 \prod_{\substack{s\in\lambda\\ l_\lambda(s)\ {\rm even}}}
 b_\lambda(s;q,t),
 \qquad
 b_{\lambda}^{oa}
 :=
 \prod_{\substack{s\in\lambda\\ a_\lambda(s)\ {\rm odd}}}
 b_\lambda(s;q,t).
\]
We also set
\begin{equation}\label{eq:bbar-def}
 \overline b_\lambda^{el}:=\frac{b_\lambda(q,t)}{b_\lambda^{el}},
 \qquad
 \overline b_\lambda^{oa}:=\frac{b_\lambda(q,t)}{b_\lambda^{oa}}.
\end{equation}
For the even and conjugate-even boundary types we use the same \(b\)-weights:
\[
 b_\lambda^{deel}:=b_\lambda^{el},
 \qquad
 b_\lambda^{eoa}:=b_\lambda^{oa},
\]
and
\[
 \overline b_\lambda^{deel}:=\overline b_\lambda^{el},
 \qquad
 \overline b_\lambda^{eoa}:=\overline b_\lambda^{oa}.
\]

Let
\begin{equation}\label{eq:boundary-indicator}
\mathbf 1_c(\lambda):=
\begin{cases}
1, & c\in\{el,oa\},\\
\mathbf 1_{\{\lambda\ {\rm conjugate\text{-}even}\}}, & c=deel,\\
\mathbf 1_{\{\lambda\ {\rm even}\}}, & c=eoa.
\end{cases}
\end{equation}


For a single-variable specialization $[x]$ and its dual specialization $[x]'$, the skew Macdonald symmetric functions are given by (see the remarks on p.~346 of \cite{MG95})
\begin{align}
 P_{\lambda/\mu}([x];q,t)&=\delta_{\mu\prec\lambda}\,\psi_{\lambda/\mu}(q,t)x^{|\lambda|-|\mu|},
 &
 Q_{\lambda/\mu}([x];q,t)&=\delta_{\mu\prec\lambda}\,\varphi_{\lambda/\mu}(q,t)x^{|\lambda|-|\mu|},
 \label{pqs1}
\end{align}
and
\begin{align}
 P_{\lambda/\mu}([x]';q,t)&=\delta_{\mu'\prec\lambda'}\,\varphi'_{\lambda/\mu}(q,t)x^{|\lambda|-|\mu|},
 &
 Q_{\lambda/\mu}([x]';q,t)&=\delta_{\mu'\prec\lambda'}\,\psi'_{\lambda/\mu}(q,t)x^{|\lambda|-|\mu|}.
 \label{pqs2}
\end{align}
Here $\psi$, $\varphi$, $\psi'$, and $\varphi'$ are given on p.~341 of \cite{MG95}:
\begin{align*}
\varphi_{\lambda/\mu}(q,t)
 &=\prod_{s\in C_{\lambda/\mu}}\frac{b_\lambda(s;q,t)}{b_\mu(s;q,t)},\\
\psi_{\lambda/\mu}(q,t)
 &=\prod_{s\in R_{\lambda/\mu}\setminus C_{\lambda/\mu}}\frac{b_\mu(s;q,t)}{b_\lambda(s;q,t)},\\
\varphi'_{\lambda/\mu}(q,t)
 &=\prod_{s\in R_{\lambda/\mu}}\frac{b_\mu(s;q,t)}{b_\lambda(s;q,t)}
 =\prod_{s\in C_{\lambda'/\mu'}}\frac{b_{\lambda'}(s;t,q)}{b_{\mu'}(s;t,q)}
 =\varphi_{\lambda'/\mu'}(t,q),\\
\psi'_{\lambda/\mu}(q,t)
 &=\prod_{s\in C_{\lambda/\mu}\setminus R_{\lambda/\mu}}\frac{b_\lambda(s;q,t)}{b_\mu(s;q,t)}
 =\psi_{\lambda'/\mu'}(t,q),
\end{align*}
where $C_{\lambda/\mu}$ (respectively, $R_{\lambda/\mu}$) denotes the union of columns (respectively, rows) intersecting $\lambda/\mu$.
The coefficients \(\psi,\varphi,\psi',\varphi'\) are the standard Macdonald
branching coefficients; they are products of \(b_\nu(s;q,t)\)-ratios and are
positive when \(0<q,t<1\).

Let
\[
\pmb\mu=(\mu^{(l)},\mu^{(l+1)},\ldots,\mu^{(r+1)})
\]
be a sequence of partitions satisfying the local interlacing relations determined by
\(RYG(l,r,a,b)\)(See Lemma \ref{lem:dimer-partition-correspondence}):
\begin{align}
\begin{array}{ll}
(a_i,b_i)=(L,+):& \mu^{(i)}\prec\mu^{(i+1)},\\
(a_i,b_i)=(L,-):& \mu^{(i)}\succ\mu^{(i+1)},\\
(a_i,b_i)=(R,+):& \mu^{(i)}\prec'\mu^{(i+1)},\\
(a_i,b_i)=(R,-):& \mu^{(i)}\succ'\mu^{(i+1)}.
\end{array}\label{cdd}
\end{align}
Here \(\prec\) denotes horizontal-strip interlacing and \(\prec'\) denotes the
corresponding vertical-strip interlacing after conjugation.

For such a sequence, set
\begin{align}
A(\pmb\mu)
&:=
\prod_{\substack{i\in[l..r]\\ (a_i,b_i)=(L,+)}}
 P_{\mu^{(i+1)}/\mu^{(i)}}([x_i];q,t)
\prod_{\substack{i\in[l..r]\\ (a_i,b_i)=(L,-)}}
 Q_{\mu^{(i)}/\mu^{(i+1)}}([x_i];q,t)
\notag\\
&\qquad\times
\prod_{\substack{i\in[l..r]\\ (a_i,b_i)=(R,+)}}
 Q_{[\mu^{(i+1)}]'/[\mu^{(i)}]'}([x_i];t,q)
\prod_{\substack{i\in[l..r]\\ (a_i,b_i)=(R,-)}}
 P_{[\mu^{(i)}]'/[\mu^{(i+1)}]'}([x_i];t,q).
\label{ppt}
\end{align}
Here $P,Q$ are (skew) Macdonald polynomials; see Appendix \ref{sc:dmp}, (\ref{pqs1}), (\ref{pqs2}).

The unnormalized Macdonald weight of \(\pmb\mu\) is
\[
W_{c_l,c_r}(\pmb\mu)
:=
\mathbf 1_{c_l}(\mu^{(l)})
\mathbf 1_{c_r}(\mu^{(r+1)})
\,
u^{|\mu^{(l)}|}v^{|\mu^{(r+1)}|}
\frac{b_{\mu^{(r+1)}}^{c_r}}
     {\overline b_{\mu^{(l)}}^{c_l}}
A(\pmb\mu).
\]
The normalizing constant is
\begin{equation}\label{def-Zclr}
Z_{c_l,c_r}
:=
\sum_{\pmb\mu}
W_{c_l,c_r}(\pmb\mu),
\end{equation}
where the sum is over all partition sequences satisfying the above local interlacing
relations.  Whenever
\[
        0<Z_{c_l,c_r}<\infty,
\]
we define the doubly free-boundary Macdonald rail-yard measure by
\begin{equation}\label{dpm}
\mathbb P_{c_l,c_r}(\pmb\mu)
:=
\frac{W_{c_l,c_r}(\pmb\mu)}{Z_{c_l,c_r}}.
\end{equation}
Via Lemma~\ref{lem:dimer-partition-correspondence}, this is also the probability measure
on the corresponding canonical free-boundary dimer states.

\begin{assumption}[Piecewise periodic Jack scaling]\label{ap5}
Let
\[
\{RYG(l^{(\epsilon)},r^{(\epsilon)},a^{(\epsilon)},b^{(\epsilon)})\}_{\epsilon>0}
\]
be a sequence of rail-yard graphs, and let \(x_i^{(\epsilon)}\) denote the weight of a
diagonal edge adjacent to an even vertex with abscissa \(2i\).
\begin{enumerate}
    \item \textbf{Piecewise periodicity of the graph.}
    For a positive integer \(n\) and real numbers
    \(V_0<V_1<\cdots<V_m\), we say that the sequence is \(n\)-periodic with transition
    points \(V_0,\ldots,V_m\) as \(\epsilon\to0\) if the following hold.
    \begin{enumerate}
        \item For each \(\epsilon>0\), there exist integer multiples of \(n\)
        \[
        l^{(\epsilon)}=v_0^{(\epsilon)}<v_1^{(\epsilon)}<\cdots<v_m^{(\epsilon)}
        =r^{(\epsilon)}
        \]
        such that
        \[
        \lim_{\epsilon\to0}\epsilon v_p^{(\epsilon)}=V_p,
        \qquad p\in\{0\}\cup[m].
        \]
        \item The sequence \(a^{(\epsilon)}\) is \(n\)-periodic on
        \([l^{(\epsilon)}..r^{(\epsilon)}]\) and does not depend on \(\epsilon\).
        More precisely, there exist \(a_1,\ldots,a_n\in\{L,R\}\) such that
        \[
        a_i^{(\epsilon)}=a_{i_{\equiv_n}}.
        \]
        \item For each \(p\in[m]\), the sequence \(b^{(\epsilon)}\) is \(n\)-periodic
        on \((v_{p-1}^{(\epsilon)},v_p^{(\epsilon)})\) and does not depend on
        \(\epsilon\), although it may depend on \(p\).  More precisely, there exist
        \(b_{p,1},\ldots,b_{p,n}\in\{+,-\}\) such that for
        \(i\in(v_{p-1}^{(\epsilon)},v_p^{(\epsilon)})\),
        \[
        b_i^{(\epsilon)}=b_{p,i_{\equiv_n}}.
        \]
    \end{enumerate}

    \item \textbf{Periodicity of weights.}
    The weights \(x_i^{(\epsilon)}\) are periodic \(q\)-volume weights:
    \[
    x_i^{(\epsilon)}=
    \begin{cases}
    e^{-\epsilon(i-i_{\equiv_n})}\tau_k,& b_i^{(\epsilon)}=+,\\
    e^{\epsilon(i-i_{\equiv_n})}\tau_k^{-1},& b_i^{(\epsilon)}=-,
    \end{cases}
    \]
    where \(k=i_{\equiv_n}\) and \(\tau_1,\ldots,\tau_n>0\) are independent of
    \(\epsilon\).

   \item \textbf{Exact Jack scaling and marked columns.}
We work in the exact Jack specialization
\begin{equation}\label{jw}
        q=t^\alpha,\qquad \alpha>0,
\end{equation}
and use the exact mesh parametrization
\begin{equation}\label{dtbeta}
        t=e^{-n\beta\epsilon},
\end{equation}
where \(\beta>0\) is independent of \(\epsilon\).

Let \(s\) be a positive integer.  For each \(d\in[s]\), assume that there exists a
sequence \(i_d^{(\epsilon)}\in[l^{(\epsilon)}..r^{(\epsilon)}]\) such that
\begin{equation}\label{dci}
        \lim_{\epsilon\to0}\epsilon i_d^{(\epsilon)}=\chi_d,
\end{equation}
with \(\chi_1\le\cdots\le\chi_s\), and such that \(i_d^{(\epsilon)}\bmod n\) is
independent of \(\epsilon\).  Equivalently, there exist
\(i_1^*,\ldots,i_s^*\in[n]\) for which
\[
        \bigl(i_d^{(\epsilon)}\bigr)_{\equiv_n}=i_d^*.
\]
When \(s=1\), we suppress the index and write
\[
        \lim_{\epsilon\to0}\epsilon i^{(\epsilon)}=\chi,
        \qquad
        \bigl(i^{(\epsilon)}\bigr)_{\equiv_n}=i^*.
\]
\end{enumerate}
\end{assumption}

For later use, we set
\[
l^{(0)}:=V_0,\qquad r^{(0)}:=V_m.
\]

\medskip

\noindent\textbf{Slope-measure convention.}
Let \(F:\mathbb R\to\mathbb R\) be a right-continuous nondecreasing function.
Its \emph{slope measure} in the \(\kappa\)-direction is the distributional
derivative \(dF\), equivalently the positive Borel measure characterized by
\[
        dF((a,b])=F(b)-F(a),\qquad a<b.
\]
Thus, if \(F\) is differentiable, then \(dF=F'(\kappa)\,d\kappa\); if \(F\)
has a jump of size \(J\) at \(\kappa_0\), then \(dF\) has an atom of mass
\(J\) at \(\kappa_0\).

In the theorem below,
\[
        H_\epsilon(\kappa)
        =
        \epsilon h_{M^{(\epsilon)}}^{(q,t)}
        \left(i^{(\epsilon)},\frac{\kappa}{\epsilon}\right)
\]
is the rescaled height profile along a marked \(L\)-type column.  By the
vertical-line formula \eqref{hm1}, this height is twice the number of occupied
horizontal and diagonal edges below the observation height.  Hence
\(H_\epsilon\) is nondecreasing in \(\kappa\).  The lower reference phase fixes
the additive constant by making the normalized height equal to zero for all
sufficiently negative \(\kappa\).

Assume the one-point admissibility condition of
Assumption~\ref{ass:contour-branch-admissibility} at the marked column $\chi$. set
\begin{align}
        S_\chi(w)
        :=
        \mathcal G_\chi(w)\prod_{r\ge1}\mathcal F_{u,v,r}(w),\label{dsc}
\end{align}
where \(\mathcal G_\chi\) and \(\mathcal F_{u,v,r}\) are defined in
\eqref{dgc} and \eqref{dfuvk}.  On the admissible one-point contour
\(\mathcal C\) of Proposition~\ref{p57}, define
\begin{align}
        T_\chi(w)
        :=
        \exp\!\left\{\beta\,\mathrm{Log}_{\chi,\mathcal C}S_\chi(w)\right\},\label{dtc}
\end{align}
with the logarithm fixed by
Assumption~\ref{ass:contour-branch-admissibility}.

\begin{theorem}(Laplace law of large numbers and weak slope-measure limit shape.)
\label{l61}
\label{thm:limit-shape-beta}
\label{thm:weak-limit-shape-beta}
Fix boundary types \(c_l,c_r\in\{el,oa,deel,eoa\}\) and parameters
\(u,v\in(0,1)\).  For each \(\epsilon>0\), let \(M^{(\epsilon)}\) be sampled
from the doubly free-boundary Macdonald measure \(\mathbb P_{c_l,c_r}\) in
\eqref{dpm}.  Suppose Assumption~\ref{ap5} holds, and suppose that the
one-point admissibility condition of
Assumption~\ref{ass:contour-branch-admissibility} holds at the marked column.

Let \(i^{(\epsilon)}\) be a marked \(L\)-type column satisfying \eqref{dci},
with fixed residue class modulo \(n\), and write
\[
        \epsilon i^{(\epsilon)}\to\chi
        \in (V_0,V_m)\setminus\{V_1,\ldots,V_{m-1}\}.
\]
In the charge-centered \((q,t)\)-coordinate of
Definition~\ref{def:qt-column-embedding}, set
\[
        H_\epsilon(\kappa)
        :=
        \epsilon h^{(q,t)}_{M^{(\epsilon)}}
        \left(i^{(\epsilon)},\frac{\kappa}{\epsilon}\right).
\]

Then, for every \(k\in\mathbb Z_{>0}\),
\begin{equation}\label{lph}
        \int_{\mathbb R}e^{-n\beta k\kappa}H_\epsilon(\kappa)\,d\kappa
        \xrightarrow{\mathbb P}
        \frac{1}{n^2\alpha k^2\beta^2\pi\mathbf i}
        \oint_{\mathcal C}T_\chi(w)^k\,\frac{dw}{w}.
\end{equation}

Moreover, let \(\nu_\chi^{(\epsilon)}\) be the pushed-forward slope measure
defined by
\begin{equation}\label{dnu-eps}
        \nu_\chi^{(\epsilon)}(B)
        :=
        \int_{\{\kappa:\,e^{-n\beta\kappa}\in B\}}
        n\beta e^{-n\beta\kappa}\,dH_\epsilon(\kappa),
        \qquad B\subset[0,\infty).
\end{equation}
Then \(\nu_\chi^{(\epsilon)}\) converges weakly in probability to a
deterministic compactly supported finite measure \(\nu_\chi\), whose moments are
\begin{equation}\label{nu-moments}
        \int_0^\infty x^{k-1}\nu_\chi(dx)
        =
        \frac{1}{\alpha k\pi\mathbf i}
        \oint_{\mathcal C}T_\chi(w)^k\,\frac{dw}{w},
        \qquad k\in\mathbb Z_{>0}.
\end{equation}
The associated macroscopic height profile is
\begin{equation}\label{def-H-from-nu}
        \mathcal H(\chi,\kappa)
        :=
        \frac{1}{n\beta}
        \int_{(e^{-n\beta\kappa},\infty)}
        \frac{1}{x}\,\nu_\chi(dx).
\end{equation}
\end{theorem}

\begin{definition}[Regular liquid and frozen points]\label{df29}
Let \(\mathcal H(\chi,\kappa)\) be the deterministic limit shape obtained from
Theorem~\ref{l61}.  A point \((\chi,\kappa)\) is called a \emph{regular slope point} if
\[
        \chi\in (V_0,V_m)\setminus\{V_1,\ldots,V_{m-1}\}
\]
and the vertical slope
\[
        \partial_\kappa\mathcal H(\chi,\kappa)
\]
is well defined at this point.

At a regular slope point, we call \((\chi,\kappa)\) \emph{liquid} if
\[
        \partial_\kappa\mathcal H(\chi,\kappa)
        \in \left(0,\frac{2}{\alpha}\right),
\]
and \emph{frozen} if
\[
        \partial_\kappa\mathcal H(\chi,\kappa)
        \in \left\{0,\frac{2}{\alpha}\right\}.
\]
The \emph{frozen boundary} is the interface between liquid and frozen regular
slope points.
\end{definition}

\begin{theorem}
[Regular liquid and frozen points in the separated all-\(L\) regime]
\label{p412}

Fix boundary types
\[
        c_l,c_r\in\{el,oa,deel,eoa\}
\]
and parameters \(u,v\in(0,1)\).
Suppose Assumption~\ref{ap5} holds and that the relevant one-point
contour and spectral-branch admissibility conditions of
Assumption~\ref{ass:contour-branch-admissibility} hold.
Use the limit shape \(\mathcal H\) supplied by
Theorem~\ref{l61}.

Assume further that
\[
        \beta=1
\]
and that Assumptions~\ref{ap64}--\ref{ap65} hold.

Let
\[
        S_\chi(w)
        =
        \mathcal G_\chi(w)
        \prod_{r\ge1}\mathcal F_{u,v,r}(w)
\]
be defined with the compatible logarithm fixed by
Assumption~\ref{ass:contour-branch-admissibility}.
For \(K\ge1\), set
\[
        S_{\chi,K}(w)
        :=
        \mathcal G_\chi(w)
        \prod_{r=1}^{K}\mathcal F_{u,v,r}(w),
        \qquad
        C_{\chi,K}
        :=
        \lim_{w\to\infty}S_{\chi,K}(w),
\]
where the latter limit is taken in the reduced rational expression.

Let \((\chi,\kappa)\) be a regular slope point and put
\[
        x:=e^{-n\kappa}.
\]
Assume that the following nonexceptionality conditions hold.

\begin{enumerate}[label=\textup{(\roman*)}]
\item
The Stieltjes boundary-value formula \eqref{dsm2-natural} holds at
\(x\), and the root labels and logarithms in
Lemma~\ref{lem:root-labelled-stieltjes-jack}
admit continuation to \(x+\mathbf i0\) through the upper half-plane.

\item
\[
        S_\chi(0)\neq x.
\]

\item
There exists \(K_*\ge1\) such that
\[
        S_{\chi,K}(0)\neq x,
        \qquad
        C_{\chi,K}\neq x,
        \qquad
        K\ge K_*.
\]
\end{enumerate}

Then the characteristic equation
\[
        S_\chi(w)=x
\]
has at most one nonreal conjugate pair, counted with multiplicity.

Moreover,
\[
(\chi,\kappa)\text{ is liquid}
\quad\Longleftrightarrow\quad
S_\chi(w)=x
\text{ has a unique nonreal conjugate pair}.
\]
Writing its upper-half-plane member as \(w_+\), one has
\[
        \frac{\partial\mathcal H(\chi,\kappa)}{\partial\kappa}
        =
        \frac{2}{\alpha}
        -
        \frac{2\arg w_+}{\pi\alpha},
        \qquad
        \arg w_+\in(0,\pi).
\]
Equivalently,
\[
(\chi,\kappa)\text{ is frozen}
\quad\Longleftrightarrow\quad
S_\chi(w)=x
\text{ has only real roots}.
\]

Let \(\mathscr N\) denote the set of regular slope points satisfying
conditions \textup{(i)}--\textup{(iii)} above.  At every point of
\(\operatorname{int}\mathscr N\) where the characteristic equation
has no multiple real root, the number of upper-half-plane roots is
locally constant.  Consequently, the portion of the regular frozen
boundary lying in \(\operatorname{int}\mathscr N\) is contained in
the real multiple-root locus
\begin{equation}
\label{fb}
\begin{cases}
\displaystyle
S_\chi(w)=e^{-n\kappa},\\[4pt]
\displaystyle
\frac{d}{dw}
\mathrm{Log}_{\chi,\mathcal C_\chi}S_\chi(w)=0,
\end{cases}
\qquad w\in\mathbb R.
\end{equation}
At a nondegenerate fold point the multiple root has multiplicity
two, so \eqref{fb} is the usual double-root system.
\end{theorem}

The next theorem describes finite-dimensional fluctuations of the
unrescaled height-Laplace transforms in the charge-centered
\((q,t)\)-column coordinate.  For arbitrary piecewise-periodic
\(L/R\) backgrounds, with all marked columns of \(L\)-type and
under the stated multipoint contour, branch, and normal-convergence
hypotheses, the limiting vector is Gaussian.

\begin{theorem}[Annular Laplace-test Gaussian fluctuations]\label{t77}
Fix boundary types $c_l,c_r\in\{el,oa,deel,eoa\}$ and parameters $u,v\in(0,1)$.  Let \newline $\{RYG(l^{(\epsilon)},r^{(\epsilon)},a^{(\epsilon)},b^{(\epsilon)})\}_{\epsilon>0}$ be a sequence of rail-yard graphs satisfying Assumption~\ref{ap5}. Assume that all marked columns are of \(L\)-type:
\[
        a_{i_d^{(\epsilon)}}^{(\epsilon)}=L,
        \qquad d\in[s],
\]
for all sufficiently small \(\epsilon\).
Assume further that the multipoint contour, branch, and
normal-convergence hypotheses used in Theorem~\ref{t58} hold.  For each $\epsilon>0$, let $M^{(\epsilon)}$ be the random canonical free-boundary dimer state corresponding, via Lemma~\ref{lem:dimer-partition-correspondence}, to a partition sequence sampled from \eqref{dpm}.

Fix $s\in\ZZ_{>0}$ and positive integers $k_1,\ldots,k_s$.  For each $d\in[s]$, choose a sequence $i_d^{(\epsilon)}\in[l^{(\epsilon)}..r^{(\epsilon)}]$ such that
\[
 \epsilon i_d^{(\epsilon)}\to\chi_d\in(V_0,V_m)\setminus\{V_1,\ldots,V_{m-1}\},
 \qquad (i_d^{(\epsilon)})_{\equiv_n}\ \text{is independent of }\epsilon.
\]
In the definition below the coordinate-system height is the charge-centered height
$h^{(q,t),\circ}_{M^{(\epsilon)}}(i_d^{(\epsilon)},\cdot)$ from Definition~\ref{def:charge-centered-qt-coordinate}.  Set
\[
X_{k_d}^{(\epsilon)}(i_d^{(\epsilon)})
:=\int_{-\infty}^{\infty}
\left(
 h^{(q,t),\circ}_{M^{(\epsilon)}}\left(i_d^{(\epsilon)},\frac{\eta}{\epsilon}\right)
 -\mathbb E h^{(q,t),\circ}_{M^{(\epsilon)}}\left(i_d^{(\epsilon)},\frac{\eta}{\epsilon}\right)
\right)e^{-n\beta k_d\eta}\,d\eta .
\]
Then
\[
\left(X_{k_1}^{(\epsilon)}(i_1^{(\epsilon)}),\ldots,X_{k_s}^{(\epsilon)}(i_s^{(\epsilon)})\right)
\Longrightarrow
\left(\mathcal X_{k_1}(\chi_1),\ldots,\mathcal X_{k_s}(\chi_s)\right)
\]
in distribution, where the limit is a centered Gaussian vector.  Its covariance is
\begin{align}
&\mathrm{Cov}\bigl[\mathcal X_{k_d}(\chi_d),\mathcal X_{k_h}(\chi_h)\bigr]\notag\\
&\quad=
\frac{4}{\alpha n^2\beta^2 k_dk_h(2\pi\mathbf i)^2}
\oint_{\mathcal C_d}\oint_{\mathcal C_h}
\Phi_{\chi_d,k_d}(z)\Phi_{\chi_h,k_h}(w)
\left[\frac{zw}{(z-w)^2}+\mathscr B_{u,v}(z,w)\right]
\frac{dz}{z}\frac{dw}{w}.
\label{main-corrected-fluc-cov}
\end{align}
Here
\[
\Phi_{\chi,k}(z)
:=\bigl[\mathcal G_{\chi}(z)\bigr]^{k\beta}
\prod_{r\ge1}\bigl[\mathcal F_{u,v,r}(z)\bigr]^{k\beta},
\]
$\mathscr B_{u,v}$ is defined in~\eqref{dBuv}, and the contours are as in Theorem~\ref{t58}; in particular they are chosen so that all four series defining $\mathscr B_{u,v}(z,w)$ converge normally on $\mathcal C_d\times\mathcal C_h$.  If $\chi_d=\chi_h$, the two contours are taken to be disjoint nested contours around the same pole set.  By Definition~\ref{def:annular-fb-gff} and Proposition~\ref{prop:annular-gff-kernel}, the bracketed kernel in \eqref{main-corrected-fluc-cov} is the logarithmic-derivative annular free-boundary image covariance kernel on the Laplace-test class.
\end{theorem}

\begin{remark}
The phrase ``annular free-boundary image covariance'' is used in the precise image-kernel sense of Definition~\ref{def:annular-fb-gff}.  The terms involving \(u^2zw\) and \(v^2/(zw)\) are the two free-boundary images.  Positivity of the covariance matrix for the test observables follows from the proof in Section~\ref{sect:gff}: for every finite collection of observables, the displayed matrix is the limit of covariance matrices of the prelimit centered random variables.
\end{remark}

\section{Partition Function}\label{sect:pf}

In this section, we compute the partition function of perfect matchings on doubly free-boundary rail-yard graphs with Macdonald weights. For general positive edge weights, the corresponding partition function was computed in \cite{zl23}, inspired by the partition-function computation for doubly free-boundary steep tilings in \cite{BCC17}. In the Macdonald-weighted setting, however, several additional combinatorial identities for Macdonald polynomials are needed; these will be established below.

\subsection{Dimer Coverings and Interlacing Partitions}

The free-boundary dimer-covering definition of Definition~\ref{df21}, together with the particle--hole convention of Definition~\ref{def:pure-dimer}, is formulated at the level of edge subsets and boundary Maya states.  With this convention every free-boundary dimer covering determines the boundary partitions $\lambda^{(M,l)}$ and $\lambda^{(M,r+1)}$, and conversely these boundary partitions determine the states of all unmatched boundary vertices.

\begin{lemma}[Rail-yard dimer--partition correspondence]
\label{lem:dimer-partition-correspondence}
Under the particle--hole convention of Definition~\ref{def:pure-dimer}, a
free-boundary dimer covering \(M\) of \(RYG(l,r,a,b)\) determines a sequence
of partitions
\[
        \bigl(\lambda^{(M,l)},\lambda^{(M,l+1)},\ldots,
        \lambda^{(M,r+1)}\bigr)\in\mathbb Y^{r-l+2}.
\]
This correspondence is bijective between free-boundary dimer coverings and
partition sequences
\[
        (\mu^{(l)},\mu^{(l+1)},\ldots,\mu^{(r+1)})
\]
satisfying, for every \(i\in[l..r]\), the local interlacing rule (\ref{cdd}).

If boundary labels \(c_l,c_r\in\{el,oa,deel,eoa\}\) are imposed, the same
bijection is restricted by the endpoint conditions
\[
        \mathbf 1_{c_l}(\mu^{(l)})=
        \mathbf 1_{c_r}(\mu^{(r+1)})=1,
\]
where \(\mathbf 1_c\) is the boundary indicator defined in
Section~\ref{subsec:macdonald-dfb-measure}.  Equivalently, \(deel\) imposes
that the conjugate endpoint partition is even, \(eoa\) imposes that the
endpoint partition itself is even, and \(el,oa\) impose no parity restriction.
\end{lemma}

\begin{proof}
The local rail-yard correspondence is the transfer-matrix correspondence of
\cite[Equations~(12)--(14), (19)--(20), Proposition~8]{bbccr}.  In that
statement, an admissible rail-yard covering with prescribed left and right
Maya boundary states is encoded by a matrix element
\[
        \langle l|
        \Gamma_{a_l b_l}(x_l)\Gamma_{a_{l+1}b_{l+1}}(x_{l+1})
        \cdots
        \Gamma_{a_r b_r}(x_r)
        |r\rangle .
\]
The four operators \(\Gamma_{L+},\Gamma_{L-},\Gamma_{R+},\Gamma_{R-}\)
act on partitions exactly by the four horizontal- or vertical-strip
interlacing relations displayed above.

Our free-boundary convention only changes how unmatched boundary odd vertices
are completed into Maya diagrams.  By Definition~\ref{def:pure-dimer}, unmatched
left-boundary odd vertices are treated as particles and unmatched right-boundary
odd vertices as holes; hence the two boundary columns also determine ordinary
Maya diagrams and therefore endpoint partitions.  Thus the same column-by-column
bijection applies with arbitrary endpoint partitions
\(\mu^{(l)}\) and \(\mu^{(r+1)}\).

The conditions attached to \(deel\) and \(eoa\) are not additional local matching
rules.  They are endpoint support restrictions coming from the Macdonald
free-boundary weights and are encoded by the indicators
\(\mathbf 1_{c_l}\) and \(\mathbf 1_{c_r}\).
\end{proof}

We define the $(q,t)$-charge on column $2m-1$ by
\begin{align}
 c^{(M,m,q,t)}
 &=\bigl[\text{number of particles on column }(2m-1)\text{ in the upper half-plane}\bigr] \notag\\
 &\qquad-\frac{\log q}{\log t}\bigl[\text{number of holes on column }(2m-1)\text{ in the lower half-plane}\bigr].
 \label{dcg}
\end{align}

For partitions $\lambda$ and $\mu$, we write $\lambda\supseteq\mu$ if $\lambda_i\geq\mu_i$ for all $i\in\NN$. Let $c_l,c_r\in\{el,oa,deel,eoa\}$.

Consider a sequence of partitions
\begin{align*}
 \pmb{\mu}=(\mu^{(l)},\mu^{(l+1)},\ldots,\mu^{(r+1)})\in\YY^{r-l+2}
\end{align*}
satisfying:
\begin{enumerate}[label=(\alph*)]
\item the local interlacing rule (\ref{cdd});
\item If $c_l=deel$ (respectively, $c_r=deel$), then $[\mu^{(l)}]'$ (respectively, $[\mu^{(r+1)}]'$) is even;
\item If $c_l=eoa$ (respectively, $c_r=eoa$), then $\mu^{(l)}$ (respectively, $\mu^{(r+1)}$) is even.
\end{enumerate}
Set
\begin{equation}\label{dZbdry}
\mathfrak Z_{c_l,c_r}(z):=\sum_{\lambda\in\YY}
\mathbf 1_{c_l}(\lambda)\mathbf 1_{c_r}(\lambda)
\frac{b_\lambda^{c_r}}{\overline b_\lambda^{c_l}}z^{|\lambda|}.
\end{equation}

We now derive a factorized formula for $Z_{c_l,c_r}$; defined as in (\ref{def-Zclr}).  The formula is explicit up to the terminal free-boundary normalization $\mathfrak Z_{c_l,c_r}(uv)$ defined in \eqref{dZbdry}, which is carried along as a separate boundary factor.  We begin with several standard identities for Macdonald polynomials.

\begin{lemma}
Let
\begin{align*}
 (a;q)_\infty:=\prod_{r=0}^{\infty}(1-aq^r).
\end{align*}
Let $\mathbf{x}=(x_1,x_2,\ldots)$ and $\mathbf{y}=(y_1,y_2,\ldots)$. The following identities appear on p.~349, 4(ii) and 4(iv), of \cite{MG95}:
\begin{align}
 \sum_{\lambda}b_{\lambda}^{el}P_{\lambda}(\mathbf{x})
 &=\prod_i\frac{(tx_i;q)_\infty}{(x_i;q)_\infty}
 \prod_{i<j}\frac{(tx_ix_j;q)_\infty}{(x_ix_j;q)_\infty}
 :=\Theta_{el}(\mathbf{x}),
 \label{ps1}\\
 \sum_{\lambda}b_{\lambda}^{oa}P_{\lambda}(\mathbf{x})
 &=\prod_i\frac{(qtx_i^2;q^2)_\infty}{(1-x_i)(q^2x_i^2;q^2)_\infty}
 \prod_{i<j}\frac{(tx_ix_j;q)_\infty}{(x_ix_j;q)_\infty}
 :=\Theta_{oa}(\mathbf{x}).
 \label{ps2}
\end{align}
The following identities appear on p.~349, 4(i) and 4(iii), of \cite{MG95}:
\begin{align}
 \sum_{\lambda:\,\lambda'\ \mathrm{even}}b_{\lambda}^{el}P_{\lambda}(\mathbf{x})
 &=\prod_{i<j}\frac{(tx_ix_j;q)_\infty}{(x_ix_j;q)_\infty}
 :=\Theta_{de,el}(\mathbf{x}),
 \label{ps3}\\
 \sum_{\lambda:\,\lambda\ \mathrm{even}}b_{\lambda}^{oa}P_{\lambda}(\mathbf{x})
 &=\prod_i\frac{(qtx_i^2;q^2)_\infty}{(x_i^2;q^2)_\infty}
 \prod_{i<j}\frac{(tx_ix_j;q)_\infty}{(x_ix_j;q)_\infty}
 :=\Theta_{e,oa}(\mathbf{x}).
 \label{ps4}
\end{align}
Define
\begin{align}
 \Pi(\mathbf{x},\mathbf{y};q,t):=\prod_{i,j}\frac{(tx_iy_j;q)_\infty}{(x_iy_j;q)_\infty}.
 \label{dpxy}
\end{align}
The Cauchy identity on p.~324, (4.13), of \cite{MG95} is
\begin{align}
 \sum_{\lambda}P_{\lambda}(\mathbf{x};q,t)Q_{\lambda}(\mathbf{y};q,t)=\Pi(\mathbf{x},\mathbf{y};q,t).
 \label{pqpi}
\end{align}
The identities on p.~323, (7.1$'$), and p.~324, (7.5), of \cite{MG95} are
\begin{align}
 P_{\mu}P_{\nu}=\sum_{\lambda}f_{\mu\nu}^{\lambda}P_{\lambda},
 \label{pmf}\\
 Q_{\lambda/\mu}=\sum_{\nu}f_{\mu\nu}^{\lambda}Q_{\nu}.
 \label{qlm}
\end{align}
Finally,
\begin{align}
 Q_{\lambda/\mu}=\frac{b_{\lambda}(q,t)}{b_{\mu}(q,t)}P_{\lambda/\mu}.
 \label{qrp}
\end{align}
\end{lemma}

\begin{lemma}\label{le13}
\begin{align}
 \sum_{\lambda}b_{\lambda}^{el}P_{\lambda/\eta}(\mathbf{x};q,t)
 &=\Theta_{el}(\mathbf{x})\sum_{\mu}b_{\mu}^{el}Q_{\eta/\mu}(\mathbf{x};q,t),
 \label{s1}\\
 \sum_{\lambda}b_{\lambda}^{oa}P_{\lambda/\eta}(\mathbf{x};q,t)
 &=\Theta_{oa}(\mathbf{x})\sum_{\mu}b_{\mu}^{oa}Q_{\eta/\mu}(\mathbf{x};q,t).
 \label{s2}
\end{align}
\end{lemma}

\begin{proof}
We use $*$ to denote either $el$ or $oa$, and let $\mathbf{y}=(y_1,y_2,\ldots)$. By \eqref{ps1}, \eqref{ps2}, and \eqref{pqpi},
\begin{align*}
 \Theta_*(\mathbf{x},\mathbf{y})
 &:=\sum_{\phi}b_{\phi}^{*}P_{\phi}(\mathbf{x},\mathbf{y})
 =\Theta_*(\mathbf{x})\Theta_*(\mathbf{y})\Pi(\mathbf{x},\mathbf{y};q,t)\\
 &=\Theta_*(\mathbf{x})
 \left[\sum_{\mu}b_{\mu}^{*}P_{\mu}(\mathbf{y};q,t)\right]
 \left[\sum_{\phi}P_{\phi}(\mathbf{x};q,t)Q_{\phi}(\mathbf{y};q,t)\right].
\end{align*}
Using \eqref{pmf} and \eqref{qlm}, we obtain
\begin{align}
 \Theta_*(\mathbf{x},\mathbf{y})
 &=\Theta_*(\mathbf{x})\sum_{\mu,\phi}b_{\mu}^{*}Q_{\phi}(\mathbf{x};q,t)P_{\mu}(\mathbf{y};q,t)P_{\phi}(\mathbf{y};q,t) \notag\\
 &=\Theta_*(\mathbf{x})\sum_{\mu,\phi}b_{\mu}^{*}Q_{\phi}(\mathbf{x};q,t)
 \sum_{\lambda}f_{\mu\phi}^{\lambda}P_{\lambda}(\mathbf{y};q,t) \notag\\
 &=\Theta_*(\mathbf{x})\sum_{\mu,\lambda}b_{\mu}^{*}P_{\lambda}(\mathbf{y};q,t)
 \sum_{\phi}f_{\mu\phi}^{\lambda}Q_{\phi}(\mathbf{x};q,t) \notag\\
 &=\Theta_*(\mathbf{x})\sum_{\mu,\lambda}b_{\mu}^{*}P_{\lambda}(\mathbf{y};q,t)Q_{\lambda/\mu}(\mathbf{x};q,t).
 \label{tt1}
\end{align}
On the other hand,
\begin{align}
 \Theta_*(\mathbf{x},\mathbf{y})
 =\sum_{\phi}b_{\phi}^{*}P_{\phi}(\mathbf{x},\mathbf{y})
 =\sum_{\phi,\lambda}b_{\phi}^{*}P_{\phi/\lambda}(\mathbf{x};q,t)P_{\lambda}(\mathbf{y};q,t).
 \label{tt2}
\end{align}
Comparing the coefficients of $P_{\lambda}(\mathbf{y};q,t)$ in \eqref{tt1} and \eqref{tt2} yields
\begin{align*}
 \sum_{\phi}b_{\phi}^{*}P_{\phi/\lambda}(\mathbf{x};q,t)
 =\Theta_*(\mathbf{x})\sum_{\mu}b_{\mu}^{*}Q_{\lambda/\mu}(\mathbf{x};q,t),
\end{align*}
which is exactly the claim.
\end{proof}

The following identities were proved in Proposition~2.1 of \cite{BBCW18}.

\begin{lemma}\label{lem25}
Define
\begin{align*}
 \mathcal{E}_{\lambda}(\mathbf{x}):=\sum_{\mu'\ \mathrm{even}}b_{\mu}^{el}Q_{\lambda/\mu}(\mathbf{x}).
\end{align*}
Then
\begin{align*}
 \sum_{\nu'\ \mathrm{even}}b_{\nu}^{el}P_{\nu/\lambda}(\mathbf{x})
 =\Theta_{de,el}(\mathbf{x})\sum_{\mu'\ \mathrm{even}}b_{\mu}^{el}Q_{\lambda/\mu}(\mathbf{x})
 =\Theta_{de,el}(\mathbf{x})\mathcal{E}_{\lambda}(\mathbf{x}).
\end{align*}
\end{lemma}

\begin{lemma}\label{lem26}
\begin{align*}
 \sum_{\nu\ \mathrm{even}}b_{\nu}^{oa}P_{\nu/\lambda}(\mathbf{x})
 =\Theta_{e,oa}(\mathbf{x})\sum_{\mu\ \mathrm{even}}b_{\mu}^{oa}Q_{\lambda/\mu}(\mathbf{x}).
\end{align*}
\end{lemma}

\begin{proof}
The proof is the same as that of Lemma~\ref{le13}, except that one sums only over even partitions at the relevant steps.
\end{proof}

\begin{corollary}\label{cl14}
Let $x\in\CC$, let $v\in\CC$ with $|v|<1$, and let $*\in\{el,oa\}$. Then
\begin{align}
 \sum_{\lambda}b_{\lambda}^{*}P_{\lambda/\eta}([x];q,t)v^{|\lambda|}
 &=\Theta_{*}(vx)\sum_{\mu}b_{\mu}^{*}Q_{\eta/\mu}([v^2x];q,t)v^{|\mu|},
 \label{271}\\
 \sum_{\lambda\ \mathrm{even}}b_{\lambda}^{oa}P_{\lambda/\eta}([x];q,t)v^{|\lambda|}
 &=\Theta_{e,oa}(vx)\sum_{\mu\ \mathrm{even}}b_{\mu}^{oa}Q_{\eta/\mu}([v^2x];q,t)v^{|\mu|},
 \label{272}\\
 \sum_{\lambda'\ \mathrm{even}}b_{\lambda}^{el}P_{\lambda/\eta}([x];q,t)v^{|\lambda|}
 &=\Theta_{de,el}(vx)\sum_{\mu'\ \mathrm{even}}b_{\mu}^{el}Q_{\eta/\mu}([v^2x];q,t)v^{|\mu|}.
 \label{273}
\end{align}
\end{corollary}

\begin{proof}
By \eqref{s1} and \eqref{s2},
\begin{align*}
 \sum_{\lambda}b_{\lambda}^{*}P_{\lambda/\eta}([x];q,t)v^{|\lambda|}
 &=v^{|\eta|}\sum_{\lambda}b_{\lambda}^{*}P_{\lambda/\eta}([vx];q,t)\\
 &=v^{|\eta|}\Theta_*(vx)\sum_{\mu}b_{\mu}^{*}Q_{\eta/\mu}([vx];q,t)\\
 &=\Theta_*(vx)\sum_{\mu}b_{\mu}^{*}Q_{\eta/\mu}([v^2x];q,t)v^{|\mu|}.
\end{align*}
This proves \eqref{271}. The identities \eqref{272} and \eqref{273} follow similarly from Lemmas~\ref{lem26} and \ref{lem25}, respectively.
\end{proof}

\begin{lemma}\label{le15}
Let $*\in\{el,oa\}$. Then
\begin{align}
 \sum_{\phi}\frac{1}{\overline{b}_{\phi}^{*}}Q_{\phi/\lambda}(\mathbf{x};q,t)
 &=\Theta_*(\mathbf{x};q,t)\sum_{\mu}\frac{1}{\overline{b}_{\mu}^{*}}P_{\lambda/\mu}(\mathbf{x};q,t),
 \label{281}\\
 \sum_{\phi\ \mathrm{even}}\frac{1}{\overline{b}_{\phi}^{oa}}Q_{\phi/\lambda}(\mathbf{x};q,t)
 &=\Theta_{e,oa}(\mathbf{x};q,t)\sum_{\mu\ \mathrm{even}}\frac{1}{\overline{b}_{\mu}^{oa}}P_{\lambda/\mu}(\mathbf{x};q,t),
 \label{282}\\
 \sum_{\phi'\ \mathrm{even}}\frac{1}{\overline{b}_{\phi}^{el}}Q_{\phi/\lambda}(\mathbf{x};q,t)
 &=\Theta_{de,el}(\mathbf{x};q,t)\sum_{\mu'\ \mathrm{even}}\frac{1}{\overline{b}_{\mu}^{el}}P_{\lambda/\mu}(\mathbf{x};q,t).
 \label{283}
\end{align}
\end{lemma}

\begin{proof}
We prove only \eqref{281}; the proofs of \eqref{282} and \eqref{283} are analogous. By \eqref{qrp},
\begin{align*}
 \sum_{\phi}\frac{1}{\overline{b}_{\phi}^{*}}Q_{\phi/\lambda}(\mathbf{x};q,t)
 =\sum_{\phi}\frac{1}{\overline{b}_{\phi}^{*}}\frac{b_{\phi}(q,t)}{b_{\lambda}(q,t)}P_{\phi/\lambda}(\mathbf{x};q,t).
\end{align*}
Using \eqref{eq:bbar-def}, this becomes
\begin{align*}
 \sum_{\phi}\frac{1}{\overline{b}_{\phi}^{*}}Q_{\phi/\lambda}(\mathbf{x};q,t)
 =\frac{1}{b_{\lambda}(q,t)}\sum_{\phi}b_{\phi}^{*}P_{\phi/\lambda}(\mathbf{x};q,t).
\end{align*}
Applying Lemma~\ref{le13}, we obtain
\begin{align*}
 \sum_{\phi}\frac{1}{\overline{b}_{\phi}^{*}}Q_{\phi/\lambda}(\mathbf{x};q,t)
 =\frac{1}{b_{\lambda}(q,t)}\Theta_*(\mathbf{x})\sum_{\mu}b_{\mu}^{*}Q_{\lambda/\mu}(\mathbf{x};q,t),
\end{align*}
and a second use of \eqref{qrp} yields \eqref{281}.
\end{proof}

\begin{corollary}\label{c16}
Let $x\in\CC$, let $v\in\CC$ with $|v|<1$, and let $*\in\{el,oa\}$. Then
\begin{align*}
 \sum_{\lambda}\frac{1}{\overline{b}_{\lambda}^{*}}Q_{\lambda/\eta}([x];q,t)v^{|\lambda|}
 &=\Theta_{*}(vx)\sum_{\mu}\frac{1}{\overline{b}_{\mu}^{*}}P_{\eta/\mu}([v^2x];q,t)v^{|\mu|},\\
 \sum_{\lambda\ \mathrm{even}}\frac{1}{\overline{b}_{\lambda}^{oa}}Q_{\lambda/\eta}([x];q,t)v^{|\lambda|}
 &=\Theta_{e,oa}(vx)\sum_{\mu\ \mathrm{even}}\frac{1}{\overline{b}_{\mu}^{oa}}P_{\eta/\mu}([v^2x];q,t)v^{|\mu|},\\
 \sum_{\lambda'\ \mathrm{even}}\frac{1}{\overline{b}_{\lambda}^{el}}Q_{\lambda/\eta}([x];q,t)v^{|\lambda|}
 &=\Theta_{de,el}(vx)\sum_{\mu'\ \mathrm{even}}\frac{1}{\overline{b}_{\mu}^{el}}P_{\eta/\mu}([v^2x];q,t)v^{|\mu|}.
\end{align*}
\end{corollary}

\begin{proof}
The proof repeats the argument of Corollary~\ref{cl14} with Lemma~\ref{le15} replacing Lemma~\ref{le13}.  The only change is that the conjugate-even boundary specialization is used at each reflection step; the Cauchy factors and terminal parity indicator are otherwise unchanged.
\end{proof}

\begin{lemma}\label{le17}
Let $\lambda,\mu\in\YY$. Then
\begin{align*}
 \sum_{\phi}P_{\phi/\lambda}(x;q,t)Q_{\phi/\mu}(y;q,t)
 &=\Pi(x,y;q,t)\sum_{\sigma}P_{\mu/\sigma}(x;q,t)Q_{\lambda/\sigma}(y;q,t),\\
 \sum_{\phi}Q_{\phi'/\lambda'}(x;t,q)Q_{\phi/\mu}(y;q,t)
 &=\Pi_0(x,y)\sum_{\sigma}Q_{\mu'/\sigma'}(x;t,q)Q_{\lambda/\sigma}(y;q,t),\\
 \sum_{\phi}P_{\phi/\lambda}(x;q,t)P_{\phi'/\mu'}(y;t,q)
 &=\Pi_0(x,y)\sum_{\sigma}P_{\mu/\sigma}(x;q,t)P_{\lambda'/\sigma'}(y;t,q),
\end{align*}
where
\begin{align}
 \Pi_0(x,y)=\prod_{i,j}(1+x_iy_j).
 \label{dp0xy}
\end{align}
\end{lemma}

\begin{proof}
See p.~352 of \cite{MG95}.
\end{proof}

\begin{lemma}\label{le18}
Let $\lambda,\mu\in\YY$. Then
\begin{align}
 \sum_{\phi}P_{\phi/\lambda}(x;q,t)P_{\mu/\phi}(y;q,t)
 &=\sum_{\nu}P_{\nu/\lambda}(y;q,t)P_{\mu/\nu}(x;q,t),
 \label{ppi}\\
 \sum_{\phi}P_{\phi/\lambda}(x;q,t)Q_{\mu'/\phi'}(y;t,q)
 &=\sum_{\nu}Q_{\nu'/\lambda'}(y;t,q)P_{\mu/\nu}(x;q,t).
 \label{pqi}
\end{align}
\end{lemma}

\begin{proof}
Since the skew Macdonald polynomial $P_{\mu/\lambda}(x,y;q,t)$ is symmetric in $x$ and $y$,
\begin{align*}
 P_{\mu/\lambda}(x,y;q,t)
 =\sum_{\phi}P_{\phi/\lambda}(x;q,t)P_{\mu/\phi}(y;q,t)
 =\sum_{\nu}P_{\nu/\lambda}(y;q,t)P_{\mu/\nu}(x;q,t),
\end{align*}
which proves \eqref{ppi}. Then \eqref{pqi} follows by applying $\omega_{q,t}$ to the variable $y$.
\end{proof}

For a symmetric function $f\in\Lambda_X$, write $f(\rho):=\rho(f)$. A specialization $\rho$ is determined by its generating function
\begin{align*}
 H(\rho;z;q,t):=\exp\left(\sum_{n\geq 1}\frac{1-t^n}{1-q^n}\frac{p_n(\rho)z^n}{n}\right).
\end{align*}
If $\rho_1,\rho_2$ are specializations and $c\in\CC$, define new specializations $\rho_1\cup\rho_2$, $c*\rho$, and $c\rho$ by
\begin{align*}
 H(\rho_1\cup\rho_2;z)&:=H(\rho_1;z)H(\rho_2;z),\\
 H(c*\rho;z)&:=H(\rho;z)^c,\\
 H(c\rho;z)&:=H(\rho;cz).
\end{align*}
Equivalently, for all $n\geq 1$,
\begin{align}
 p_n(\rho_1\cup\rho_2)&=p_n(\rho_1)+p_n(\rho_2),
 &
 p_n(c*\rho)&=c\,p_n(\rho),
 &
 p_n(c\rho)&=c^n p_n(\rho).
 \label{rhn}
\end{align}
Define also
\begin{align}
 H(\rho_1;\rho_2;q,t)
 &:=\sum_{\lambda\in\YY}P_{\lambda}(\rho_1)Q_{\lambda}(\rho_2)
 =\exp\left(\sum_{n\geq 1}\frac{1-t^n}{1-q^n}\frac{p_n(\rho_1)p_n(\rho_2)}{n}\right).
 \label{dh1}
\end{align}
In particular, if $\rho_2$ is the specialization of a single variable $z$, then $H(\rho_1;\rho_2;q,t)=H(\rho_1;z;q,t)$. Moreover,
\begin{align*}
 H(\rho;\rho_1\cup\rho_2;q,t)=H(\rho;\rho_1;q,t)H(\rho;\rho_2;q,t).
\end{align*}

The following formulas for $\Theta_{el}$ and $\Theta_{oa}$ will be useful.

\begin{lemma}\label{le110}
Assume $q,t,x_i\in(0,1)$. Then
\begin{align}
 \Theta_{el}(\rho;q,t)
 &=\exp\left[\sum_{n=1}^{\infty}\frac{1-t^n}{1-q^n}\left(\frac{\rho(p_{2n-1})}{2n-1}+\frac{[\rho(p_n)]^2}{2n}\right)\right],
 \label{sel}\\
 \Theta_{oa}(\rho;q,t)
 &=\exp\left[\sum_{n=1}^{\infty}\frac{q^n(q^n-t^n)\rho(p_{2n})}{n(1-q^{2n})}+\frac{\rho(p_n)}{n}+\frac{1-t^n}{1-q^n}\frac{[\rho(p_n)]^2-\rho(p_{2n})}{2n}\right],
 \label{soa}\\
 \Theta_{de,el}(\rho;q,t)
 &=\exp\left[\sum_{n=1}^{\infty}\frac{1-t^n}{1-q^n}\left(-\frac{\rho(p_{2n})}{2n}+\frac{[\rho(p_n)]^2}{2n}\right)\right],
 \label{sdel}\\
 \Theta_{e,oa}(\rho;q,t)
 &=\exp\left[\sum_{n=1}^{\infty}\frac{(1-t^nq^n)\rho(p_{2n})}{n(1-q^{2n})}+\frac{1-t^n}{1-q^n}\frac{[\rho(p_n)]^2-\rho(p_{2n})}{2n}\right].
 \label{seoa}
\end{align}
\end{lemma}

\begin{proof}
Starting from \eqref{ps1}, we compute
\begin{align*}
 \Theta_{el}(\rho;q,t)
 &=\exp\Biggl[\sum_i\sum_{r\ge0}\bigl(\log(1-tx_iq^r)-\log(1-x_iq^r)\bigr)\\
 &\qquad\qquad+\sum_{i<j}\sum_{r\ge0}\bigl(\log(1-tx_ix_jq^r)-\log(1-x_ix_jq^r)\bigr)\Biggr]\\
 &=\exp\left[\sum_{n=1}^{\infty}\frac{1-t^n}{1-q^n}\left(\sum_i\frac{x_i^n}{n}+\sum_{i<j}\frac{(x_ix_j)^n}{n}\right)\right],
\end{align*}
which simplifies to \eqref{sel}. The proofs of \eqref{soa}, \eqref{sdel}, and \eqref{seoa} are obtained in the same way from \eqref{ps2}, \eqref{ps3}, and \eqref{ps4}, respectively.
\end{proof}

Lemma~\ref{le110} has the following immediate consequence.

\begin{corollary}
With the notation in \eqref{rhn}, for $*\in\{el,oa,deel,eoa\}$,
\begin{align*}
 \Theta_*(\rho_1\cup\rho_2;q,t)=\Theta_*(\rho_1;q,t)\Theta_*(\rho_2;q,t)H(\rho_1;\rho_2;q,t).
\end{align*}
\end{corollary}

Applying Lemma~\ref{le110} together with the identity
\begin{align*}
 \omega_{q,t}(p_n)=(-1)^{n-1}\frac{1-q^n}{1-t^n}p_n,
\end{align*}
we obtain the following corollary.

\begin{corollary}
For any specialization $\rho$,
\begin{small}
\begin{align}
 \Theta_{el}(\omega_{q,t}\rho;q,t)
 &=\exp\left[\sum_{n=1}^{\infty}\frac{1-t^n}{1-q^n}\left(\frac{1-q^{2n-1}}{1-t^{2n-1}}\frac{\rho(p_{2n-1})}{2n-1}+\left(\frac{1-q^n}{1-t^n}\right)^2\frac{[\rho(p_n)]^2}{2n}\right)\right],
 \label{selo}\\
 \Theta_{oa}(\omega_{q,t}\rho;q,t)
 &=\exp\left[-\sum_{n=1}^{\infty}\frac{q^n(q^n-t^n)\rho(p_{2n})}{n(1-t^{2n})}+\sum_{n=1}^{\infty}\frac{(-1)^{n+1}(1-q^n)}{1-t^n}\frac{\rho(p_n)}{n}\right. \notag\\
 &\qquad\left.+\sum_{n=1}^{\infty}\frac{1-q^n}{1-t^n}\frac{[\rho(p_n)]^2}{2n}+\sum_{n=1}^{\infty}\frac{1+q^n}{1+t^n}\frac{\rho(p_{2n})}{2n}\right],
 \label{soao}\\
 \Theta_{de,el}(\omega_{q,t}\rho;q,t)
 &=\exp\left[\sum_{n=1}^{\infty}\frac{1+q^n}{1+t^n}\frac{\rho(p_{2n})}{2n}+\sum_{n=1}^{\infty}\frac{1-q^n}{1-t^n}\frac{[\rho(p_n)]^2}{2n}\right],
 \label{sdeelo}\\
 \Theta_{e,oa}(\omega_{q,t}\rho;q,t)
 &=\exp\left[-\sum_{n=1}^{\infty}\frac{(1-t^nq^n)\rho(p_{2n})}{n(1-t^{2n})}+\sum_{n=1}^{\infty}\frac{1+q^n}{1+t^n}\frac{\rho(p_{2n})}{2n}+\sum_{n=1}^{\infty}\frac{1-q^n}{1-t^n}\frac{[\rho(p_n)]^2}{2n}\right].
 \label{seoao}
\end{align}
\end{small}
\end{corollary}

\begin{proposition}[Factorized partition function with terminal boundary normalization]\label{p23}
Assume $u,v\in(0,1)$. Then the partition function is given by
\begin{small}
\begin{align}
 Z_{c_l,c_r}
 &=\prod_{\substack{l\le i<j\le r\\ b_i=+,\,b_j=-}}z_{i,j}
\left[\prod_{n\ge1}\prod_{\substack{i\in[l..r]\\ (a_i,b_i)=(L,+)}}\Theta_{c_r}(u^{2n-2}v^{2n-1}x_i)\Theta_{c_l}(u^{2n-1}v^{2n}x_i)\right.
 \notag\\
 &\qquad\times\prod_{\substack{i\in[l..r]\\ (a_i,b_i)=(R,+)}}\omega_{q,t}\bigl[\Theta_{c_r}(u^{2n-2}v^{2n-1}x_i)\Theta_{c_l}(u^{2n-1}v^{2n}x_i)\bigr]
 \notag\\
 &\qquad\times\prod_{\substack{i\in[l..r]\\ (a_i,b_i)=(L,-)}}\Theta_{c_l}(u^{2n-1}v^{2n-2}x_i)\Theta_{c_r}(u^{2n}v^{2n-1}x_i)
 \notag\\
&\qquad\left.\times\prod_{\substack{i\in[l..r]\\ (a_i,b_i)=(R,-)}}\omega_{q,t}\bigl[\Theta_{c_l}(u^{2n-1}v^{2n-2}x_i)\Theta_{c_r}(u^{2n}v^{2n-1}x_i)\bigr]\right]
\mathfrak Z_{c_l,c_r}(uv)
 \notag\\
&\quad\times\left[\prod_{n\ge1}\prod_{\substack{l\le i,j\le r\\ b_i=+,\,b_j=-,\,a_i\neq a_j}}(1+u^{2n}v^{2n}x_ix_j)\right.\\
&\left. \prod_{\substack{l\le i,j\le r\\ b_i=+,\,b_j=-,\,a_i=a_j}}\Pi(u^{2n}v^{2n}x_i,u^{2n-2}v^{2n-2}x_j;q,t)\Pi(u^{2n}v^{2n}x_i,u^{2n}v^{2n}x_j;q,t)\right]
 \notag\\
&\quad\times\left[\prod_{n\ge1}\prod_{\substack{i<j\in[l..r]\\ b_i=b_j=+,\,a_i\neq a_j}}(1+u^{2n-2}v^{2n}x_ix_j)
 \prod_{\substack{i<j\in[l..r]\\ b_i=b_j=-,\,a_i\neq a_j}}(1+u^{2n}v^{2n-2}x_ix_j)\right]
 \notag\\
&\quad\times\left[\prod_{n\ge1}\prod_{\substack{i<j\in[l..r]\\ b_i=b_j=+,\,a_i=a_j}}\Pi(u^{2n-2}v^{2n}x_i,u^{2n-2}v^{2n-2}x_j;q,t)\Pi(u^{2n}v^{2n}x_i,u^{2n-2}v^{2n}x_j;q,t)
\right.
 \notag\\
&\qquad\left.\times\prod_{\substack{i<j\in[l..r]\\ b_i=b_j=-,\,a_i=a_j}}\Pi(u^{2n}v^{2n-2}x_i,u^{2n-2}v^{2n-2}x_j;q,t)\Pi(u^{2n}v^{2n}x_i,u^{2n}v^{2n-2}x_j;q,t)\right].
 \label{ffp}
\end{align}
\end{small}
where
\begin{align}
 z_{i,j}=\begin{cases}
 1+x_ix_j,& a_i\neq a_j,\\
 \Pi(x_i,x_j;q,t),& a_i=a_j.
 \end{cases}
 \label{dzij}
\end{align}
Moreover, the partition function for dimer coverings on a rail-yard graph $RYG(l,r,\underline{a},\underline{b})$ with free boundary condition on the left and empty partition on the right is
\begin{align}
 Z_{c_l,c_r,f,\emptyset}
 &=\prod_{\substack{l\le i<j\le r\\ b_i=+,\,b_j=-}}z_{i,j}
 \left[\prod_{\substack{i\in[l..r]\\ (a_i,b_i)=(L,-)}}\Theta_{c_l}(ux_i)\right]
 \left[\prod_{\substack{i\in[l..r]\\ (a_i,b_i)=(R,-)}}\omega_{q,t}\bigl[\Theta_{c_l}(ux_i)\bigr]\right]
 \notag\\
 &\qquad\times\prod_{\substack{i<j\in[l..r]\\ b_i=b_j=-,\,a_i\neq a_j}}(1+u^2x_ix_j)
 \prod_{\substack{i<j\in[l..r]\\ b_i=b_j=-,\,a_i=a_j}}\Pi(u^2x_i,x_j;q,t).
 \label{fep}
\end{align}
\end{proposition}

\begin{proof}
Consider a sequence of partitions
\begin{align*}
 (\mu^{(l)},\mu^{(l+1)},\ldots,\mu^{(r+1)})\in\YY^{r-l+2}.
\end{align*}
Recall that the interlacing relations between consecutive partitions are prescribed by conditions (a)--(d) above. These interlacing relations satisfy nontrivial commutation relations (Lemma~\ref{le17}), trivial commutation relations (Lemma~\ref{le18}), and reflection relations (Corollary~\ref{c16}).

The strategy is to commute the interlacing relations, or equivalently the $+$ and $-$ signs, until the partition function factors into an explicit product.

\textit{Step 1.} Commute the interlacing relations so that every relation in $\{\prec,\prec'\}$ lies to the right of every relation in $\{\succ,\succ'\}$. By Lemma~\ref{le17}, these commutations contribute the factor
\begin{align}
 \prod_{\substack{l\le i<j\le r\\ b_i=+,\,b_j=-}}z_{i,j},
 \qquad
 z_{i,j}=\begin{cases}
 1+x_ix_j,& a_i\neq a_j,\\
 \Pi(x_i,x_j;q,t),& a_i=a_j.
 \end{cases}
 \label{step1factor}
\end{align}

\textit{Step 2.} For each $b_i=-$ (from left to right), move it to the left boundary, reflect it there, move it to the right boundary, reflect it again, and finally return it to its original position.

By Corollary~\ref{c16}, reflection at the left boundary contributes
\begin{align*}
 \prod_{\substack{i\in[l..r]\\ (a_i,b_i)=(L,-)}}\Theta_{c_l}(ux_i)
 \prod_{\substack{i\in[l..r]\\ (a_i,b_i)=(R,-)}}\omega_{q,t}\Theta_{c_l}(ux_i).
\end{align*}
After this reflection, each parameter $x_i$ is multiplied by $u^2$. Reflection at the right boundary contributes
\begin{align*}
 \prod_{\substack{i\in[l..r]\\ (a_i,b_i)=(L,-)}}\Theta_{c_r}(u^2vx_i)
 \prod_{\substack{i\in[l..r]\\ (a_i,b_i)=(R,-)}}\omega_{q,t}\Theta_{c_r}(u^2vx_i).
\end{align*}
After reflection at the left boundary, $b_i=-$ becomes $+$; when it moves to the right, it interchanges with all the other $-$ signs. By Lemma~\ref{le17}, these interchanges contribute
\begin{align*}
 &\prod_{\substack{i<j\in[l..r]\\ b_i=b_j=-,\,a_i=a_j}}
 \Pi(u^2x_i,x_j;q,t)\Pi(u^2v^2x_i,u^2x_j;q,t)\\
 &\qquad\times
 \prod_{\substack{i<j\in[l..r]\\ b_i=b_j=-,\,a_i\neq a_j}}
 (1+u^2x_ix_j)(1+u^4v^2x_ix_j).
\end{align*}
After reflection at the right boundary, the resulting $+$ becomes $-$ and, as it moves back to the left, it interchanges with all the $+$ signs. Again by Lemma~\ref{le17}, this contributes
\begin{align*}
 &\prod_{\substack{i,j\in[l..r]\\ b_i=+,\,b_j=-,\,a_i=a_j}}
 \Pi(u^2v^2x_i,x_j;q,t)
 \prod_{\substack{i,j\in[l..r]\\ b_i=+,\,b_j=-,\,a_i\neq a_j}}
 (1+u^2v^2x_ix_j).
\end{align*}

\textit{Step 3.} For each $b_i=+$ (from right to left), reflect it at the right boundary, move it to the left boundary, reflect it there, and finally return it to its original position.

By Corollary~\ref{c16}, reflection at the right boundary contributes
\begin{align*}
 \prod_{\substack{i\in[l..r]\\ (a_i,b_i)=(L,+)}}\Theta_{c_r}(vx_i)
 \prod_{\substack{i\in[l..r]\\ (a_i,b_i)=(R,+)}}\omega_{q,t}\Theta_{c_r}(vx_i).
\end{align*}
After this reflection, each parameter $x_i$ is multiplied by $v^2$. Reflection at the left boundary contributes
\begin{align*}
 \prod_{\substack{i\in[l..r]\\ (a_i,b_i)=(L,+)}}\Theta_{c_l}(uv^2x_i)
 \prod_{\substack{i\in[l..r]\\ (a_i,b_i)=(R,+)}}\omega_{q,t}\Theta_{c_l}(uv^2x_i).
\end{align*}
After reflection at the right boundary, $b_i=+$ becomes $-$; when it moves to the left, it interchanges with all the other $+$ signs. By Lemma~\ref{le17}, this contributes
\begin{align*}
 &\prod_{\substack{i<j\in[l..r]\\ b_i=b_j=+,\,a_i=a_j}}
 \Pi(v^2x_i,x_j;q,t)\Pi(u^2v^2x_i,v^2x_j;q,t)\\
 &\qquad\times
 \prod_{\substack{i<j\in[l..r]\\ b_i=b_j=+,\,a_i\neq a_j}}
 (1+v^2x_ix_j)(1+u^2v^4x_ix_j).
\end{align*}
After reflection at the left boundary, the resulting $-$ becomes $+$ and, as it moves back to the right, it interchanges with all the $-$ signs. This contributes
\begin{align*}
 &\prod_{\substack{i,j\in[l..r]\\ b_i=+,\,b_j=-,\,a_i=a_j}}
 \Pi(u^2v^2x_i,u^2v^2x_j;q,t)
 \prod_{\substack{i,j\in[l..r]\\ b_i=+,\,b_j=-,\,a_i\neq a_j}}
 (1+u^4v^4x_ix_j).
\end{align*}

\textit{Step 4.} Repeat Steps~2 and~3. After repeating them $n$ times, the $(n+1)$st repetition contributes
\begin{align*}
 R_n
 &:=\prod_{\substack{i\in[l..r]\\ (a_i,b_i)=(L,-)}}\Theta_{c_l}(u^{2n+1}v^{2n}x_i)\Theta_{c_r}(u^{2n+2}v^{2n+1}x_i)\\
 &\quad\times\prod_{\substack{i\in[l..r]\\ (a_i,b_i)=(R,-)}}\omega_{q,t}\Theta_{c_l}(u^{2n+1}v^{2n}x_i)\,\omega_{q,t}\Theta_{c_r}(u^{2n+2}v^{2n+1}x_i)\\
 &\quad\times\prod_{\substack{i<j\in[l..r]\\ b_i=b_j=-,\,a_i=a_j}}\Pi(u^{2n+2}v^{2n}x_i,u^{2n}v^{2n}x_j;q,t)\Pi(u^{2n+2}v^{2n+2}x_i,u^{2n+2}v^{2n}x_j;q,t)\\
 &\quad\times\prod_{\substack{i<j\in[l..r]\\ b_i=b_j=-,\,a_i\neq a_j}}(1+u^{4n+2}v^{4n}x_ix_j)(1+u^{4n+4}v^{4n+2}x_ix_j)\\
 &\quad\times\prod_{\substack{i,j\in[l..r]\\ b_i=+,\,b_j=-,\,a_i=a_j}}\Pi(u^{2n+2}v^{2n+2}x_i,u^{2n}v^{2n}x_j;q,t)
 \prod_{\substack{i,j\in[l..r]\\ b_i=+,\,b_j=-,\,a_i\neq a_j}}(1+u^{4n+2}v^{4n+2}x_ix_j)\\
 &\quad\times\prod_{\substack{i\in[l..r]\\ (a_i,b_i)=(L,+)}}\Theta_{c_r}(u^{2n}v^{2n+1}x_i)\Theta_{c_l}(u^{2n+1}v^{2n+2}x_i)\\
 &\quad\times\prod_{\substack{i\in[l..r]\\ (a_i,b_i)=(R,+)}}\omega_{q,t}\Theta_{c_r}(u^{2n}v^{2n+1}x_i)\,\omega_{q,t}\Theta_{c_l}(u^{2n+1}v^{2n+2}x_i)\\
 &\quad\times\prod_{\substack{i<j\in[l..r]\\ b_i=b_j=+,\,a_i=a_j}}\Pi(u^{2n}v^{2n+2}x_i,u^{2n}v^{2n}x_j;q,t)\Pi(u^{2n+2}v^{2n+2}x_i,u^{2n}v^{2n+2}x_j;q,t)\\
 &\quad\times\prod_{\substack{i<j\in[l..r]\\ b_i=b_j=+,\,a_i\neq a_j}}(1+u^{4n}v^{4n+2}x_ix_j)(1+u^{4n+2}v^{4n+4}x_ix_j)\\
 &\quad\times\prod_{\substack{i,j\in[l..r]\\ b_i=+,\,b_j=-,\,a_i=a_j}}\Pi(u^{2n+2}v^{2n+2}x_i,u^{2n+2}v^{2n+2}x_j;q,t)
 \prod_{\substack{i,j\in[l..r]\\ b_i=+,\,b_j=-,\,a_i\neq a_j}}(1+u^{4n+4}v^{4n+4}x_ix_j).
\end{align*}
After iterating Steps~2 and~3 $n+1$ times, the diagonal-edge weights become $u^{2n+2}v^{2n+2}x_i$. Since $u,v\in(0,1)$,
\begin{align*}
 \lim_{n\to\infty}u^{2n+2}v^{2n+2}x_i=0.
\end{align*}
Therefore, in the limit only configurations with no diagonal edges contribute. Under free boundary conditions, this limiting contribution is
\begin{align*}
 \mathfrak Z_{c_l,c_r}(uv).
\end{align*}
Hence,
\begin{align*}
 Z_{c_l,c_r}=\mathfrak Z_{c_l,c_r}(uv)\prod_{n=0}^{\infty}R_n,
\end{align*}
which gives \eqref{ffp}. Formula \eqref{fep} follows from \eqref{ffp} by taking the limit $v\downarrow0$.
\end{proof}

\begin{remark}[Terminal boundary normalization]\label{rem:terminal-boundary-normalization}
The factor \(\mathfrak Z_{c_l,c_r}(uv)\) in Proposition~\ref{p23} is the terminal no-diagonal-edge free-boundary scalar product left after the reflection--commutation procedure.  We keep it as a separate normalization factor rather than claim a closed product formula for it.  This is enough for the normalized observable identities used below: the same terminal scalar product is part of the partition function and is therefore included in the normalization.
\end{remark}

\begin{lemma}[Terminal normalization cancels from normalized observables]
\label{lem:terminal-normalization-cancels}
The terminal factor \(\mathfrak Z_{c_l,c_r}(uv)\) in
Proposition~\ref{p23} does not contribute to any normalized Negu\c t
observable used in Theorems~\ref{l61}, \ref{t58}, and \ref{t77}.  More
precisely, in the unnormalized expectations obtained by inserting the
Negu\c t operators and commuting the same boundary and interlacing operators,
the final no-diagonal-edge boundary scalar product is again
\(\mathfrak Z_{c_l,c_r}(uv)\), independent of the contour variables.  Hence it
appears both in the numerator and in the partition-function denominator and
cancels identically.
\end{lemma}

\begin{proof}
The proof of Proposition~\ref{p23} commutes all interlacing operators through
the two boundary reflections until the diagonal parameters have been multiplied
by powers of \(uv\) tending to zero.  The remaining scalar product is exactly
\(\mathfrak Z_{c_l,c_r}(uv)\).  In the observable formulae, the Negu\c t
operators act before this final terminal step and contribute the contour
factors displayed in Propositions~\ref{p57} and Theorem~\ref{t58}.  After
those contour factors have been extracted, the same reflection--commutation
procedure leaves the same terminal no-diagonal-edge scalar product.  Since the
probability expectation is normalized by \(Z_{c_l,c_r}\), which contains the
identical terminal factor by Proposition~\ref{p23}, the factor cancels.
\end{proof}
\section{Laplace Transform of Height Function}\label{sect:lthf}

In this section, we derive an integral formula for the expectation of the Laplace transform of the random height function. The key idea is to relate the Laplace transform to the Negu\c t operator and then evaluate its expectation by means of an explicit contour integral. The infinite-product formula for the partition function obtained in Section~\ref{sect:pf} makes it possible to carry out this program in the doubly free-boundary setting.

We begin by recalling a convenient Macdonald-process-type formalism adapted to the present model.

\begin{definition}\label{df22}
Let
\begin{align*}
 \bA&=(A^{(l)},A^{(l+1)},\ldots,A^{(r)},A^{(r+1)}),\\
 \bB&=(B^{(l)},B^{(l+1)},\ldots,B^{(r)},B^{(r+1)})
\end{align*}
be two collections of countable sets of variables. Let $u,v$ be two parameters, and let $\mathcal P=\{\mathcal L,\mathcal R\}$ be a partition of $[l..r]$, i.e.
\begin{align*}
 \mathcal L\cup\mathcal R=[l..r],
 \qquad
 \mathcal L\cap\mathcal R=\emptyset.
\end{align*}
For $c_l,c_r\in\{el,oa,deel,eoa\}$ and $q,t\in(0,1)$, define a formal probability measure on sequences of partitions
\begin{align*}
 (\lambda^{(l)},\lambda^{(l+1)},\ldots,\lambda^{(r)},\lambda^{(r+1)})
\end{align*}
by
\begin{align}
 &\MP_{\bA,\bB,\mathcal P,c_l,c_r,q,t}(\lambda^{(l)},\ldots,\lambda^{(r+1)})
 \propto
 \mathbf 1_{c_l}(\lambda^{(l)})\mathbf 1_{c_r}(\lambda^{(r+1)})
 \frac{b_{\lambda^{(r+1)}}^{c_r}}{\overline b_{\lambda^{(l)}}^{\,c_l}}
 u^{|\lambda^{(l)}|}v^{|\lambda^{(r+1)}|}
 \label{pmd}\\
 &\qquad\times
 \prod_{i\in\mathcal L}\Psi_{\lambda^{(i)},\lambda^{(i+1)}}(A^{(i)},B^{(i+1)};q,t)
 \prod_{j\in\mathcal R}\Phi_{[\lambda^{(j)}]',[\lambda^{(j+1)}]'}(A^{(j)},B^{(j+1)};q,t),
 \notag
\end{align}
where, for partitions $\lambda,\mu\in\YY$ and countable sets of variables $A,B$,
\begin{align*}
 \Psi_{\lambda,\mu}(A,B;q,t)
 &:=\sum_{\nu\in\YY}Q_{\lambda/\nu}(A;q,t)P_{\mu/\nu}(B;q,t),\\
 \Phi_{\lambda,\mu}(A,B;q,t)
 &:=\sum_{\nu\in\YY}P_{\lambda/\nu}(A;t,q)Q_{\mu/\nu}(B;t,q).
\end{align*}
The factors $\mathbf 1_{c_l}$ and $\mathbf 1_{c_r}$ are the boundary-admissibility indicators defined in \eqref{eq:boundary-indicator}; equivalently, the measure is supported only on boundary-admissible sequences.
\end{definition}

See Section~\ref{sc:dmp} for the definitions of Macdonald polynomials $P_\lambda$, $Q_\lambda$, $P_{\lambda/\mu}$, and $Q_{\lambda/\mu}$.

\begin{remark}
In terms of the Macdonald scalar product \eqref{dsp}, one may equivalently write
\begin{align*}
 \Psi_{\lambda,\mu}(A,B;q,t)
 &=\bigl\langle Q_\lambda(A,Y;q,t),P_\mu(Y,B;q,t)\bigr\rangle_{Y;q,t},\\
 \Phi_{\lambda,\mu}(A,B;q,t)
 &=\bigl\langle Q_\mu(Y,B;t,q),P_\lambda(A,Y;t,q)\bigr\rangle_{Y;t,q},
\end{align*}
where $Y$ is a countable set of variables.
\end{remark}

\begin{lemma}[Induced generalized Macdonald process]\label{l23}
Under the correspondence of Lemma~\ref{lem:dimer-partition-correspondence}, the measure
\(\mathbb P_{c_l,c_r}\) in \eqref{dpm} induces on
\[
        \bigl(\lambda^{(M,l)},\lambda^{(M,l+1)},\ldots,
        \lambda^{(M,r)},\lambda^{(M,r+1)}\bigr)
\]
the generalized Macdonald process \eqref{pmd} with
\[
        \mathcal L=\{i\in[l..r]:a_i=L\},
        \qquad
        \mathcal R=\{i\in[l..r]:a_i=R\},
\]
and
\[
\begin{array}{c|cc}
        b_i & A^{(i)} & B^{(i+1)} \\ \hline
        + & \{0\} & \{x_i\} \\
        - & \{x_i\} & \{0\}.
\end{array}
\]
The endpoint restrictions are the boundary indicators
\[
        \mathbf 1_{c_l}(\lambda^{(M,l)})
        \mathbf 1_{c_r}(\lambda^{(M,r+1)})
\]
defined in \eqref{eq:boundary-indicator}.
\end{lemma}

\begin{proof}
By Lemma~\ref{lem:dimer-partition-correspondence}, it remains only to compare the
weights of a partition sequence
\[
        \pmb\mu=(\mu^{(l)},\ldots,\mu^{(r+1)}).
\]
With the choices of \(\mathcal L,\mathcal R,A^{(i)},B^{(i+1)}\) stated above, the product of
local \(\Psi\)- and \(\Phi\)-kernels in \eqref{pmd} is exactly the bulk factor
\(A(\pmb\mu)\) in \eqref{ppt}.  The remaining endpoint factor in \eqref{pmd} is
\[
        \mathbf 1_{c_l}(\mu^{(l)})
        \mathbf 1_{c_r}(\mu^{(r+1)})
        u^{|\mu^{(l)}|}v^{|\mu^{(r+1)}|}
        \frac{b_{\mu^{(r+1)}}^{c_r}}
             {\overline b_{\mu^{(l)}}^{c_l}},
\]
which is exactly the endpoint factor in the numerator of \eqref{dpm}.  Thus the
unnormalized weights in \eqref{pmd} and \eqref{dpm} agree term by term, and division by
the same normalizing constant gives the claimed induced process.
\end{proof}

For $k\in\ZZ_{>0}$ and $\lambda\in\YY$, define
\begin{align}
 \gamma_k(\lambda;q,t)
 :=(1-t^{-k})\sum_{i=1}^{l(\lambda)}q^{k\lambda_i}t^{k(-i+1)}+t^{-kl(\lambda)}.
 \label{dgm}
\end{align}

We shall also use the abbreviations
\begin{align}
 \Pi_{L,L}(X,Y)&=\Pi(X,Y;q,t),
 \notag\\
 \Pi_{R,R}(X,Y)&=\Pi(X,Y;t,q),
 \label{PiLR}\\
 \Pi_{L,R}(X,Y)&=\Pi_{R,L}(X,Y)=\Pi_0(X,Y),
 \notag
\end{align}
where $\Pi$ and $\Pi_0$ are defined in \eqref{dpxy} and \eqref{dp0xy}.

To study the scaling limit, for each free-boundary dimer covering $M$ we use the following column-wise coordinate system for the rail-yard graph.

\begin{definition}[Charge-centered \((q,t)\)-column coordinate]
\label{def:qt-column-embedding}
\label{def:charge-centered-qt-coordinate}
Let \(M\) be a free-boundary dimer covering, and set
\[
        \rho:=\frac{\log q}{\log t}>0 .
\]
We construct the coordinate used in the height-Laplace observables column by column.

On each odd column \(x=2m-1\), place the odd vertices on a vertical line, preserving their
original vertical order.  The auxiliary vertical coordinate \(\widetilde y\) is determined up
to an additive constant by the following rule for consecutive odd vertices:
\[
        \text{distance}=
        \begin{cases}
        1, &\text{if both vertices are particles},\\[1mm]
        \rho, &\text{if both vertices are holes},\\[1mm]
        \dfrac{1+\rho}{2}, &\text{if one vertex is a particle and the other is a hole}.
        \end{cases}
\]

Let \(\widetilde Y_i\) be the auxiliary \(\widetilde y\)-coordinate of the
\(i\)-th highest particle on the column \(x=2m-1\).  The quantity
\[
        \widetilde c^{(M,m,q,t)}
        :=
        \widetilde Y_i+i-\frac12-\rho\,\lambda_i^{(M,m)}
\]
is independent of \(i\); it is the \((q,t)\)-charge of the auxiliary coordinate on this
column.  The vertical coordinate used below is the centered coordinate
\[
        y:=\widetilde y-\widetilde c^{(M,m,q,t)}.
\]
Equivalently, the \((q,t)\)-charge of the column in the coordinate \(y\) is zero.  Changing
the auxiliary origin translates both \(\widetilde y\) and
\(\widetilde c^{(M,m,q,t)}\) by the same constant, and hence leaves \(y\) unchanged.

After the odd vertices have been placed, let \(\mathfrak v=(2i,y_0)\) be an even vertex of
the original rail-yard graph, with horizontal odd neighbors
\[
        \mathfrak v^-=(2i-1,y_0),
        \qquad
        \mathfrak v^+=(2i+1,y_0).
\]
Place \(\mathfrak v\) at the midpoint of the new positions of
\(\mathfrak v^-\) and \(\mathfrak v^+\).  The edge set is unchanged: all
adjacencies are inherited from \(RYG(l,r,a,b)\).

For a marked odd column \(m\), write
\[
        h^{(q,t)}_M(m,y)
        :=
        h_M\!\left(\widetilde x_m,
        y+\widetilde c^{(M,m,q,t)}\right),
\]
where \(\widetilde x_m\) is the coordinate vertical line associated with the original column
\(x=2m-1\).  Thus \(h^{(q,t)}_M\) always denotes height in the charge-centered
\((q,t)\)-column coordinate.  No assertion is made about the same observables in the original
deterministic geometric embedding.
\end{definition}
\begin{lemma}[Charge-centered height--Negu\c t identity]
\label{le25}
\label{cor:charge-centered-height-negut}
Let \(M\) be a free-boundary dimer covering and let \(m\in[l..r+1]\).  In the
charge-centered \((q,t)\)-column coordinate of
Definition~\ref{def:qt-column-embedding}, for every \(k\in\mathbb Z_{>0}\),
\begin{equation}\label{eq:charge-centered-height-negut}
        \int_{-\infty}^{\infty}
        h^{(q,t)}_M(m,y)t^{ky}\,dy
        =
        \frac{2}{k^2\log t\log q}\,
        \gamma_k(\lambda^{(M,m)};q,t).
\end{equation}
\end{lemma}

\begin{proof}
Write \(\lambda=\lambda^{(M,m)}\), and let \(Y_i\) be the ordinate, in the
charge-centered coordinate, of the \(i\)-th highest particle on the odd
column \(x=2m-1\).  Then
the definition of \(\lambda\) gives
\begin{equation}\label{eq:centered-particle-coordinate}
        \frac{\log q}{\log t}\lambda_i
        =
        Y_i+i-\frac12 .
\end{equation}
Let
\[
        B:=Y_{\ell(\lambda)+1}+\frac12 .
\]
Below \(B\) the normalized height is zero, while for \(y>B\),
\[
        \frac{d}{dy}h^{(q,t)}_M(m,y)
        =
        \frac{2\log t}{\log q}
        \left(
        1-\sum_{i=1}^{\ell(\lambda)}
        \mathbf 1_{[Y_i-\frac12,Y_i+\frac12]}(y)
        \right).
\]
Moreover (\ref{eq:centered-particle-coordinate}) gives
\[
        B=-\ell(\lambda).
\]
Since \(t\in(0,1)\), integration by parts yields
\[
\begin{aligned}
\int_{-\infty}^{\infty}h^{(q,t)}_M(m,y)t^{ky}\,dy
&=
\frac{2t^{kB}}{k^2\log t\log q}
+
\frac{2}{k^2\log t\log q}
\sum_{i=1}^{\ell(\lambda)}
\left(t^{k(Y_i+\frac12)}-t^{k(Y_i-\frac12)}\right).
\end{aligned}
\]
Using \eqref{eq:centered-particle-coordinate}, we have
\[
        t^{k(Y_i+\frac12)}-t^{k(Y_i-\frac12)}
        =
        (1-t^{-k})q^{k\lambda_i}t^{k(-i+1)}.
\]
Together with \(B=-\ell(\lambda)\), this gives
\[
\int_{-\infty}^{\infty}h^{(q,t)}_M(m,y)t^{ky}\,dy
=
\frac{2}{k^2\log t\log q}
\left[
t^{-k\ell(\lambda)}
+
(1-t^{-k})\sum_{i=1}^{\ell(\lambda)}
q^{k\lambda_i}t^{k(-i+1)}
\right],
\]
which is \eqref{eq:charge-centered-height-negut}.
\end{proof}

\begin{lemma}\label{l41}
For every partition $\lambda\in\YY$,
\begin{align*}
 \gamma_k(\lambda';t,q)=\gamma_k\!\left(\lambda;\frac1q,\frac1t\right).
\end{align*}
\end{lemma}

\begin{proof}
This is the same argument as in Lemma~4.1 of \cite{LV21}.
\end{proof}

Recall also the standard involutive symmetry of Macdonald polynomials (see p.~324 of \cite{MG95}):
\begin{align}
 P_\lambda(X;q,t)&=P_\lambda\!\left(X;\frac1q,\frac1t\right),
 \notag\\
 Q_\lambda(X;q,t)&=\left(\frac tq\right)^{|\lambda|}Q_\lambda\!\left(X;\frac1q,\frac1t\right).
 \label{pqr}
\end{align}

\begin{lemma}\label{l210}
For any countable set of variables $X$ and $c\in\{L,R\}$, define the specialization
\begin{align*}
 \rho_{X,1,c}(p_n):=
 \begin{cases}
  p_n(X), & c=L,\\
  \omega_{q,t}(p_n(X)), & c=R.
 \end{cases}
\end{align*}
Write $\rho_c(B):=\rho_{B,1,c}$ and $\rho_d(W):=\rho_{W,1,d}$. Then
\begin{small}
\begin{align*}
 \Pi(X,Y;q,t)&=H(\rho_{X,1,L};\rho_{Y,1,L};q,t),\\
 \Pi_0(X,Y)&=H(\rho_{X,1,L};\rho_{Y,1,R};q,t),\\
 L(W,A;q,t)&=\exp\left(\sum_{n=1}^{\infty}\frac{q^n(1-t^{-n})}{n}p_n(A)p_n(W^{-1})\right),\\
 \Pi_{c,d}(B,W)&=\exp\left(\sum_{n=1}^{\infty}\frac{1-t^n}{1-q^n}\frac{p_n(\rho_c(B))p_n(\rho_d(W))}{n}\right),\\
 \frac{\Pi_{c,d}(B,W)}{\Pi_{c,d}(B,\xi(q,t,c)W)}
 &=\exp\left(\sum_{n=1}^{\infty}\frac{(1-t^n)(1-\xi(q,t,c)^n)}{n(1-q^n)}p_n(\rho_c(B))p_n(\rho_d(W))\right),\\
 L\left(W,A;\frac1t,\frac1q\right)
 &=H\bigl(\rho_{(t^{-1}W^{-1}),1,L}\cup((-1)*\rho_{W^{-1},1,L});A;t,q\bigr)\\
 &=H\bigl(\omega_{q,t}(\rho_{(t^{-1}W^{-1}),1,L}\cup((-1)*\rho_{W^{-1},1,L}));\omega_{q,t}(A);q,t\bigr).
\end{align*}
\end{small}
Moreover,
\begin{align}
 \omega(q,t,a_i)&=
 \begin{cases}
  (q,t), & a_i=L,\\
  \left(\dfrac1t,\dfrac1q\right), & a_i=R,
 \end{cases}
 \notag\\
 \omega_*(q,t,a_i)&=
 \begin{cases}
  (q,t), & a_i=L,\\
  (t,q), & a_i=R,
 \end{cases}
 \notag\\
 \xi(q,t,a_i)&=
 \begin{cases}
  q^{-1}, & a_i=L,\\
  t, & a_i=R.
 \end{cases}
 \label{ia3}
\end{align}
\end{lemma}

\begin{proof}
The first identity follows by comparing the exponential form of the Macdonald Cauchy kernel with the definition of the kernel \(H\):
\[
 \Pi(X,Y;q,t)
 =\exp\left(\sum_{n\ge1}\frac{1-t^n}{n(1-q^n)}p_n(X)p_n(Y)\right)
 =H(\rho_{X,1,L};\rho_{Y,1,L};q,t).
\]
For the second identity, applying \(\omega_{q,t}\) to the second variable gives
\[
 \frac{1-t^n}{1-q^n}p_n(X)\omega_{q,t}(p_n(Y))
 =(-1)^{n-1}p_n(X)p_n(Y),
\]
and hence
\[
 H(\rho_{X,1,L};\rho_{Y,1,R};q,t)
 =\exp\left(\sum_{n\ge1}\frac{(-1)^{n-1}}{n}p_n(X)p_n(Y)\right)
 =\Pi_0(X,Y).
\]
The formula for \(L(W,A;q,t)\) is obtained by expanding the logarithm of
\[
 \prod_{w\in W}\prod_{a\in A}\frac{w-q a/t}{w-q a}
\]
in powers of \(a/w\), which gives
\[
 \log L(W,A;q,t)=
 \sum_{n\ge1}\frac{q^n(1-t^{-n})}{n}p_n(A)p_n(W^{-1}).
\]
The two displayed formulas involving \(\Pi_{c,d}\) then follow immediately from the same exponential definition of \(H\), and from replacing \(W\) by \(\xi(q,t,c)W\) in the second one.  Finally, the last two identities are obtained by writing the specialization
\(\rho_{(t^{-1}W^{-1}),1,L}\cup((-1)*\rho_{W^{-1},1,L})\) in power sums and comparing its kernel with the preceding logarithmic expansion of \(L(W,A;1/t,1/q)\).  Applying Lemma~\ref{la6} to both arguments gives the equivalent form with \(\omega_{q,t}\).
\end{proof}

Define, for $d\in\{L,R\}$ and $v\in\CC$,
\begin{align}
 \rho_{v,2,d}(X;q,t)
 :=
 \begin{cases}
  \rho_{\frac{q^2v}{t}X^{-1},1,L}\cup\bigl((-1)*\rho_{\frac{qv}{t}X^{-1},1,L}\bigr), & d=L,\\
  \omega_{q,t}\!\left(\bigl((-1)*\rho_{vX^{-1},1,L}\bigr)\cup\rho_{t^{-1}vX^{-1},1,L}\right), & d=R,
 \end{cases}
 \label{drt}
\end{align}
and
\begin{align}
 \rho_{v,3,d}(X;q,t)
 :=
 \begin{cases}
  \rho_{vX,1,L}\cup\bigl((-1)*\rho_{q^{-1}vX,1,L}\bigr), & d=L,\\
  \rho_{vX,1,R}\cup\bigl((-1)*\rho_{tvX,1,R}\bigr), & d=R.
 \end{cases}
 \label{drs}
\end{align}
For $v=1$, we abbreviate these by $\rho_{2,d}(X;q,t)$ and $\rho_{3,d}(X;q,t)$.

\begin{lemma}\label{l31}
Let $g_i\in\ZZ_{\ge0}$ for $i\in[l+1..r]$, and for each such $i$ let $W^{(i)}$ be an ordered set of variables with $|W^{(i)}|=g_i$. Define
\begin{align*}
 \rho_A&:=\bigcup_{i=l}^{r}\rho_{A^{(i)},1,a_i},
 \qquad
 \rho_B:=\bigcup_{i=l+1}^{r+1}\rho_{B^{(i)},1,a_{i-1}},\\
 \rho_W&:=\bigcup_{i=l+1}^{r}\rho_{2,a_i}(W^{(i)};q,t),
 \qquad
 \rho_W^{\circ}:=\bigcup_{i=l+1}^{r}\rho_{3,a_i}(W^{(i)};q,t).
\end{align*}
Then
\begin{small}
\begin{align*}
&\EE_{\MP_{\bA,\bB,\mathcal P,c_l,c_r,q,t}}
 \left[\prod_{i=l+1}^{r}\gamma_{g_i}(\lambda^{(i)};q,t)\right]\\
&=\oint\cdots\oint
 \prod_{i=l+1}^{r}D\bigl(W^{(i)};\omega(q,t,a_i)\bigr)
 \prod_{l+1\le i\le j\le r}H\bigl(\rho_{2,a_i}(W^{(i)};q,t);\rho_{A^{(j)},1,a_j};q,t\bigr)\\
&\qquad\times
 \prod_{l\le i<j\le r}H\bigl(\rho_{3,a_j}(W^{(j)};q,t);\rho_{B^{(i+1)},1,a_i};q,t\bigr)
 \prod_{l+1\le i<j\le r}H\bigl(\rho_{2,a_i}(W^{(i)};q,t);\rho_{3,a_j}(W^{(j)};q,t);q,t\bigr)\\
&\qquad\times
 \prod_{k\ge1}
 \Theta_{c_l}([u^{2k-2}v^{2k-1}]\rho_W;q,t)
 \Theta_{c_r}([u^{2k-1}v^{2k}]\rho_W;q,t)
 \Theta_{c_l}([u^{2k-1}v^{2k-2}]\rho_W^{\circ};q,t)
 \Theta_{c_r}([u^{2k}v^{2k-1}]\rho_W^{\circ};q,t)\\
&\qquad\times
 \prod_{k\ge1}
 H(u^{2k}\rho_B;v^{2k}\rho_W^{\circ};q,t)
 H(u^{2k-2}\rho_B;v^{2k}\rho_W;q,t)
 H(u^{2k}\rho_W;v^{2k}\rho_W^{\circ};q,t)\\
&\qquad\times
 \prod_{k\ge1}
 H(u^{2k}\rho_A;v^{2k}\rho_W;q,t)
 H(u^{2k}\rho_A;v^{2k-2}\rho_W^{\circ};q,t).
\end{align*}
\end{small}
Here the integral contours are given by $\{\mathcal{C}_{i,j}\}_{i\in [l+1..r],s\in[g_i]}$ such that
\begin{enumerate}
    \item $\mathcal{C}_{i,s}$ is the integral contour for the variable $w^{(i)}_s\in W^{(i)}$;
    \item $\mathcal{C}_{i,s}$ encloses 0 and every singular point of 
    \begin{align*}
&\prod_{i\leq j;i,j\in[l+1..r]}H(\rho_{2,a_i}(W^{(i)};q,t);\rho_{A^{(j)},1,a_j})\prod_{k\geq 1}\left[H(u^{2k}\rho_A;v^{2k}\rho_W))H(u^{2k-2}\rho_B;v^{2k}\rho_W)\right]\\
&\prod_{k\geq 1}\Theta_{cl}([u^{2k-2}v^{2k-1}]\rho_W)
\Theta_{cr}([u^{2k-1}v^{2k}]\rho_W)
\end{align*}
but no other singular points of the integrand.
\item the contour $\mathcal{C}_{i,j}$ is contained in the domain bounded by $t\mathcal{C}_{i',j'}$ whenever $(i,j)<(i',j')$ in lexicographical ordering;
\end{enumerate}
$D(W;q,t)$, $H(W,X;q,t)$, and $\Pi_{c,d}(X,Y)$ are given by (\ref{ddf}), (\ref{dh}), and (\ref{PiLR}).
\end{lemma}

\begin{proof}
Let $Z$ denote the normalization constant of \eqref{pmd}. By Lemma~\ref{l41}, whenever $a_i=R$ we may rewrite the observable $\gamma_{g_i}(\lambda^{(i)};q,t)$ in terms of the conjugate partition and the pair $(1/t,1/q)$. Using Definition~\ref{df22}, Remark~4.2, and the identity \eqref{pqr}, the expectation in the statement becomes a normalized nested Macdonald scalar product.

For $i\in[l+1..r]$, define the local kernels
\begin{align*}
 \mathbf E_i:=
 \begin{cases}
  \displaystyle\sum_{\lambda^{(i)}\in\YY}
  \gamma_{g_i}(\lambda^{(i)};q,t)
  Q_{\lambda^{(i)}}(A^{(i)},Y^{(i)};q,t)
  P_{\lambda^{(i)}}(Y^{(i-1)},B^{(i)};q,t),
  & (a_{i-1},a_i)=(L,L),\\[2mm]
  \displaystyle\sum_{\lambda^{(i)}\in\YY}
  \gamma_{g_i}\!\left([\lambda^{(i)}]';\frac1t,\frac1q\right)
  P_{[\lambda^{(i)}]'}(A^{(i)},Y^{(i)};t,q)
  P_{\lambda^{(i)}}(Y^{(i-1)},B^{(i)};q,t),
  & (a_{i-1},a_i)=(L,R),\\[2mm]
  \displaystyle\sum_{\lambda^{(i)}\in\YY}
  \gamma_{g_i}(\lambda^{(i)};q,t)
  Q_{\lambda^{(i)}}(A^{(i)},Y^{(i)};q,t)
  Q_{[\lambda^{(i)}]'}(Y^{(i-1)},B^{(i)};t,q),
  & (a_{i-1},a_i)=(R,L),\\[2mm]
  \displaystyle\sum_{\lambda^{(i)}\in\YY}
  \gamma_{g_i}\!\left([\lambda^{(i)}]';\frac1t,\frac1q\right)
  P_{[\lambda^{(i)}]'}(A^{(i)},Y^{(i)};t,q)
  Q_{[\lambda^{(i)}]'}(Y^{(i-1)},B^{(i)};t,q),
  & (a_{i-1},a_i)=(R,R).
 \end{cases}
\end{align*}
Similarly, define the boundary terms
\begin{align*}
 \mathbf E_l&:=
 \begin{cases}
  \displaystyle\sum_{\lambda^{(l)}\in\YY}\mathbf 1_{c_l}(\lambda^{(l)})u^{|\lambda^{(l)}|}
  \bigl[\overline b_{\lambda^{(l)}}^{\,c_l}\bigr]^{-1}
  Q_{\lambda^{(l)}}(A^{(l)},Y^{(l)};q,t),
  & a_l=L,\\[2mm]
  \displaystyle\sum_{\lambda^{(l)}\in\YY}\mathbf 1_{c_l}(\lambda^{(l)})u^{|\lambda^{(l)}|}
  \bigl[\overline b_{\lambda^{(l)}}^{\,c_l}\bigr]^{-1}
  P_{[\lambda^{(l)}]'}(A^{(l)},Y^{(l)};t,q),
  & a_l=R,
 \end{cases}\\
 \mathbf E_{r+1}&:=
 \begin{cases}
  \displaystyle\sum_{\lambda^{(r+1)}\in\YY}\mathbf 1_{c_r}(\lambda^{(r+1)})v^{|\lambda^{(r+1)}|}
  b_{\lambda^{(r+1)}}^{c_r}
  P_{\lambda^{(r+1)}}(Y^{(r)},B^{(r+1)};q,t),
  & a_r=L,\\[2mm]
  \displaystyle\sum_{\lambda^{(r+1)}\in\YY}\mathbf 1_{c_r}(\lambda^{(r+1)})v^{|\lambda^{(r+1)}|}
  b_{\lambda^{(r+1)}}^{c_r}
  Q_{[\lambda^{(r+1)}]'}(Y^{(r)},B^{(r+1)};t,q),
  & a_r=R.
 \end{cases}
\end{align*}
The indicators in $\mathbf E_l$ and $\mathbf E_{r+1}$ are essential when $c_l$ or $c_r$ equals $deel$ or $eoa$; they are exactly the support restrictions in Definition~\ref{df22} and are retained until the final boundary scalar product in Lemma~\ref{l218}.
Then
\begin{align*}
 \EE_{\MP_{\bA,\bB,\mathcal P,c_l,c_r,q,t}}
 \left[\prod_{i=l+1}^{r}\gamma_{g_i}(\lambda^{(i)};q,t)\right]
 =\frac1Z\Bigl\langle \mathbf E_l,
 \bigl\langle \mathbf E_{l+1},\ldots,\langle \mathbf E_r,\mathbf E_{r+1}\rangle_{Y^{(r)}}\cdots\bigr\rangle_{Y^{(l+1)}}
 \Bigr\rangle_{Y^{(l)}}.
\end{align*}

For each $i\in[l+1..r]$, the kernel $\mathbf E_i$ is obtained by applying the Negu\c t operator to the corresponding Cauchy kernel:
\begin{align*}
 \mathbf E_i=
 \begin{cases}
  D_{-g_i,(A^{(i)},Y^{(i)});q,t}
  \Pi_{a_{i-1},a_i}\bigl((A^{(i)},Y^{(i)}),(Y^{(i-1)},B^{(i)})\bigr),
  & a_i=L,\\[2mm]
  D_{-g_i,(A^{(i)},Y^{(i)});1/t,1/q}
  \Pi_{a_{i-1},a_i}\bigl((A^{(i)},Y^{(i)}),(Y^{(i-1)},B^{(i)})\bigr),
  & a_i=R.
 \end{cases}
\end{align*}

By Lemma~\ref{l23}, the dimer measure induces the generalized Macdonald
process \eqref{pmd}.  We apply Proposition~\ref{pa2} to each inserted
observable \(\mathbf E_i\).  After interchanging the resulting contour
integrals with the scalar products, Lemma~\ref{l210} rewrites all
\(\Pi\)-, \(\Pi_0\)-, and \(L\)-kernels as \(H\)-kernels.  We obtain a multiple
contour integral whose integrand contains the \(D\)-kernels and the nested
scalar product
\[
\Bigl\langle
F_l,\,
\bigl\langle F_{l+1},\ldots,
\bigl\langle F_r,F_{r+1}\bigr\rangle_{Y^{(r)}}
\cdots\bigr\rangle_{Y^{(l+1)}}
\Bigr\rangle_{Y^{(l)}} .
\]
Here
\[
F_l=
\widetilde H\bigl(\rho_{uA^{(l)},1,a_l}\cup
\rho_{uY^{(l)},1,a_l}\bigr),
\qquad
F_{r+1}=
\widetilde H\bigl(\rho_{vB^{(r+1)},1,a_r}\cup
\rho_{vY^{(r)},1,a_r}\bigr),
\]
and, for \(i\in[l+1..r]\),
\begin{align*}
F_i
&=
H\bigl(\rho_{Y^{(i-1)},1,a_{i-1}};
        \rho_{A^{(i)},1,a_i}\bigr)
H\bigl(\rho_{Y^{(i-1)},1,a_{i-1}};
        \rho_{Y^{(i)},1,a_i}\bigr)
H\bigl(\rho_{Y^{(i)},1,a_i};
        \rho_{B^{(i)},1,a_{i-1}}\bigr)\\
&\qquad\times
H\bigl(\rho_{2,a_i}(W^{(i)};q,t);
        \rho_{Y^{(i)},1,a_i}\bigr)
H\bigl(\rho_{3,a_i}(W^{(i)};q,t);
        \rho_{Y^{(i-1)},1,a_{i-1}}\bigr).
\end{align*}

We now evaluate the scalar products from right to left.  Each active alphabet
\(Y^{(j)}\) occurs only through \(H\)-kernels and the endpoint
\(\widetilde H\)-terms.  Hence Lemma~\ref{211} contracts the \(H\)-kernels,
while Lemma~\ref{l214} transports the endpoint \(\widetilde H\)-terms through
the contraction.  Repeating this for \(j=r,r-1,\ldots,l\) gives three types
of factors.

\begin{enumerate}
\item the pair-interaction factors
\[
\prod_{l+1\le i<j\le r}
H\bigl(
\rho_{2,a_i}(W^{(i)};q,t);
\rho_{3,a_j}(W^{(j)};q,t);q,t
\bigr);
\]

\item the one-body factors
\[
\prod_{l+1\le i\le j\le r}
H\bigl(
\rho_{2,a_i}(W^{(i)};q,t);
\rho_{A^{(j)},1,a_j};q,t
\bigr),
\]
and
\[
\prod_{l\le i<j\le r}
H\bigl(
\rho_{3,a_j}(W^{(j)};q,t);
\rho_{B^{(i+1)},1,a_i};q,t
\bigr);
\]

\item the final boundary contraction
\[
\left\langle
\widetilde H\bigl(\rho_{uA^{(l)},1,a_l}\cup\rho_{uY^{(l)},1,a_l}\bigr),
\widetilde H\bigl(\rho_{vB^{(r+1)},1,a_r}\cup\rho_{vY^{(r)},1,a_r}\bigr)
\right\rangle,
\]
after all intermediate \(Y\)-variables have been contracted.
\end{enumerate}
The boundary indicators remain attached to the endpoint sums throughout this
recursion.  Hence the last scalar product is precisely the boundary scalar
product evaluated in Lemma~\ref{l218}, which gives the four infinite products
of \(\Theta\)- and \(H\)-factors appearing in the statement.

Finally, the normalizing constant \(Z\) is the partition function in
Proposition~\ref{p23}.  The boundary normalization produced by
Lemma~\ref{l218} is the same one appearing in \(Z\), and therefore cancels in
the normalized expectation.  The remaining factors are precisely the
\(D\)-kernels, one-body factors, pair-interaction factors, and boundary
products displayed in the statement.
\end{proof}

\begin{lemma}\label{211}
Let $\{d_k\}_{k\ge1}$, $\{s_k\}_{k\ge1}$, and $\{u_k\}_{k\ge1}$ be sequences in graded algebras such that
\begin{align*}
\lim_{k\to\infty}\mathrm{ldeg}(d_k)=\lim_{k\to\infty}\mathrm{ldeg}(s_k)=\lim_{k\to\infty}\mathrm{ldeg}(u_k)=\infty.
\end{align*}
Then
\begin{align*}
 &\left\langle
 \exp\left(\sum_{k=1}^{\infty}\frac{d_kp_k(Y)+s_k[p_k(Y)]^2}{k}\right),
 \exp\left(\sum_{k=1}^{\infty}\frac{u_kp_k(Y)}{k}\right)
 \right\rangle_{Y;q,t}\\
 &\qquad=
 \exp\left(\sum_{k=1}^{\infty}\frac{1-q^k}{1-t^k}\frac{d_ku_k}{k}
 +\sum_{k=1}^{\infty}\left(\frac{1-q^k}{1-t^k}\right)^2\frac{s_ku_k^2}{k}\right).
\end{align*}
\end{lemma}

\begin{proof}

If $s_k=0$, the lemma follows from Proposition 2.3 of \cite{bcgs13}. For general $s_k$, let
\begin{eqnarray*}
c_{(k_1,l_1),(k_2,l_2),\ldots,(k_m,l_m)}[p_{k_1}(Y)]^{l_1}[p_{k_2}(Y)]^{l_2}\cdots [p_{k_m}(Y)]^{l_m}:=f(Y)
\end{eqnarray*}
be a monomial in the expansion of 
\begin{eqnarray*}
\exp\left(\sum_{k=1}^{\infty}\frac{d_kp_k(Y)+s_k[p_k(Y)]^2}{k}\right);
\end{eqnarray*}
such that $k_1<k_2<\ldots <k_m$. Note that 
\begin{align}
R(Y):=\exp\left(\sum_{k=1}^{\infty}\left(\frac{u_kp_k(Y)}{k}\right)\right)=\sum_{n=0}^{\infty}
\frac{1}{n!}\left(\sum_{k=1}^{\infty}\left(\frac{u_kp_k(Y)}{k}\right)\right)^n
\end{align}
Let $L:=l_1+\ldots+l_m$. The monomial
$[p_{k_1}(Y)]^{l_1}[p_{k_2}(Y)]^{l_2}\cdots [p_{k_m}(Y)]^{l_m}$ has coefficient
\begin{eqnarray*}
\frac{1}{L!}\frac{L!}{(l_1)!(l_2)!\cdot\ldots\cdot(l_m)!}
\frac{u_{k_1}^{l_1}u_{k_2}^{l_2}\cdot\ldots\cdot u_{k_m}^{l_m}}{k_1^{l_1}k_2^{l_2}\cdot\ldots\cdot k_m^{l_m}}:=
\frac{g(Y)}{[p_{k_1}(Y)]^{l_1}[p_{k_2}(Y)]^{l_2}\cdots [p_{k_m}(Y)]^{l_m}}
\end{eqnarray*}
in (\ref{ps2}).
Then by (\ref{dsp})
\begin{align*}
\langle f(Y),R(Y) \rangle_{Y;q,t}&=
\langle f(Y),g(Y) \rangle_{Y;q,t}=c_{(k_1,l_1),(k_2,l_2),\ldots,(k_m,l_m)}\prod_{i=1}^{m}
\left(\frac{(1-q^{k_i})u_{k_i}}{1-t^{k_i}}\right)^{l_i}
\end{align*}
which can be obtained from $f(Y)$ by replacing each $p_{k_i}(Y)$ by $\frac{(1-q^{k_i})u_{k_i}}{1-t^{k_i}}$ for $1\leq i\leq m$. Then the lemma follows.
\end{proof}

\begin{lemma}\label{l214}
Let $\rho_1,\rho_2,\rho_3$ be specializations such that $\rho_1$ and $\rho_3$ are independent of the variables in $Y$, and assume that $\rho_2(p_k)=p_k(Y)$ for all $k\ge1$. Let $*\in\{oa,el,deel,eoa\}$. Then
\begin{align}
 \left\langle \Theta_*(\rho_1\cup\rho_2;q,t),H(\rho_2;\rho_3;q,t)\right\rangle_{Y;q,t}
 &=\Theta_*(\rho_1\cup\rho_3;q,t),
 \label{l214a}\\
 \left\langle \Theta_*(\rho_1\cup\omega_{q,t}(\rho_2);q,t),H(\rho_2;\omega_{t,q}(\rho_3);t,q)\right\rangle_{Y;t,q}
 &=\Theta_*(\rho_1\cup\rho_3;q,t).
 \label{l2102}
\end{align}
\end{lemma}

\begin{proof}
Insert the explicit exponential expressions for $\Theta_*$ from Lemma~\ref{le110} and for $H$ from \eqref{dh1}, and then apply Lemma~\ref{211}.  For $*=deel$ and $*=eoa$ this uses $\Theta_{deel}=\Theta_{de,el}$ and $\Theta_{eoa}=\Theta_{e,oa}$; the same linear-quadratic exponential computation applies, with the missing one-body factor in the even cases corresponding to a zero linear coefficient.  The quadratic and linear terms match exactly, yielding \eqref{l214a} and \eqref{l2102}.
\end{proof}

\begin{lemma}\label{l218}
Let $\rho_1,\rho_2$ be specializations, let $c_l,c_r\in\{oa,el,deel,eoa\}$, and let $Y$ be a countable set of variables. Let $\rho$ be the specialization defined by
\begin{align}
 \rho(p_k)=p_k(Y).
 \label{rpk}
\end{align}
Let $u,v\in(0,1)$. Then
\begin{small}
\begin{align}
 &\left\langle \Theta_{c_l}(u[\rho_1\cup\rho];q,t),\Theta_{c_r}(v[\rho_2\cup\rho];q,t)\right\rangle_{Y;q,t}
 \notag\\
 &\quad=
 \left[\prod_{i\ge1}
 \Theta_{c_l}([u^{2i-1}v^{2i-2}]\rho_1;q,t)
 \Theta_{c_r}([u^{2i}v^{2i-1}]\rho_1;q,t)
 \Theta_{c_r}([u^{2i-2}v^{2i-1}]\rho_2;q,t)
 \Theta_{c_l}([u^{2i-1}v^{2i}]\rho_2;q,t)
 \right]
 \label{l218eq}\\
 &\qquad\times
 \left[\prod_{j\ge1}H([u^{2j}v^{2j}]\rho_1;\rho_2;q,t)\right]
 \mathfrak Z_{c_l,c_r}(uv).
 \notag
\end{align}
\end{small}
If instead $\rho$ is defined by
\begin{align}
 \rho(p_k)=\omega_{q,t}(p_k(Y)),
 \label{drpk}
\end{align}
then the scalar product
\begin{align*}
 \left\langle \Theta_{c_l}(u[\rho_1\cup\rho];q,t),\Theta_{c_r}(v[\rho_2\cup\rho];q,t)\right\rangle_{Y;t,q}
\end{align*}
is given by the same right-hand side.
\end{lemma}

\begin{proof}
Define
\begin{align*}
 \mathcal I(\rho_1,\rho_2)
 :=\sum_{\lambda,\mu,\nu\in\YY}
 \mathbf 1_{c_l}(\lambda)\mathbf 1_{c_r}(\mu)
 \frac{b_\mu^{c_r}}{\overline b_\lambda^{\,c_l}}
 u^{|\lambda|}v^{|\mu|}
 Q_{\lambda/\nu}(\rho_1;q,t)
 P_{\mu/\nu}(\rho_2;q,t).
\end{align*}
By Corollary~\ref{c16} with $\eta=\emptyset$,
\begin{align*}
 \Theta_{c_l}(u[\rho_1\cup\rho];q,t)
 =\sum_{\lambda\in\YY}\mathbf 1_{c_l}(\lambda)\frac{u^{|\lambda|}}{\overline b_\lambda^{\,c_l}}Q_\lambda(\rho_1\cup\rho;q,t),
\end{align*}
and by Corollary~\ref{cl14} with $\eta=\emptyset$,
\begin{align*}
 \Theta_{c_r}(v[\rho_2\cup\rho];q,t)
 =\sum_{\mu\in\YY}\mathbf 1_{c_r}(\mu)b_\mu^{c_r}v^{|\mu|}P_\mu(\rho_2\cup\rho;q,t).
\end{align*}
Therefore the left-hand side of \eqref{l218eq} is exactly $\mathcal I(\rho_1,\rho_2)$. The same conclusion holds in the case \eqref{drpk}, because
\begin{align*}
 \bigl\langle P_\nu(\omega_{q,t}),Q_\nu(\omega_{q,t})\bigr\rangle_{Y;t,q}=1
\end{align*}
for every $\nu\in\YY$.

We now derive a recursion for $\mathcal I$. First apply Corollary~\ref{c16} to the sum over $\lambda$ and Corollary~\ref{cl14} to the sum over $\mu$:
\begin{align*}
 \mathcal I(\rho_1,\rho_2)
 &=\Theta_{c_l}(u\rho_1;q,t)\Theta_{c_r}(v\rho_2;q,t)
 \sum_{\eta,\zeta\in\YY}
 \mathbf 1_{c_l}(\eta)\mathbf 1_{c_r}(\zeta)
 \frac{u^{|\eta|}}{\overline b_\eta^{\,c_l}}
 b_\zeta^{c_r}v^{|\zeta|}
 \sum_{\nu\in\YY}
 P_{\nu/\eta}(u^2\rho_1;q,t)
 Q_{\nu/\zeta}(v^2\rho_2;q,t).
\end{align*}
Next, by Lemma~\ref{le17},
\begin{align*}
 \sum_{\nu\in\YY}P_{\nu/\eta}(u^2\rho_1;q,t)Q_{\nu/\zeta}(v^2\rho_2;q,t)
 =H([u^2v^2]\rho_1;\rho_2;q,t)
 \sum_{\sigma\in\YY}P_{\zeta/\sigma}(u^2\rho_1;q,t)Q_{\eta/\sigma}(v^2\rho_2;q,t).
\end{align*}
Applying Corollary~\ref{cl14} to the sum over $\zeta$ and Corollary~\ref{c16} to the sum over $\eta$, we obtain
\begin{align*}
 \mathcal I(\rho_1,\rho_2)
 &=\Theta_{c_l}(u\rho_1;q,t)
 \Theta_{c_r}(v\rho_2;q,t)
 \Theta_{c_r}(u^2v\rho_1;q,t)
 \Theta_{c_l}(uv^2\rho_2;q,t)
 H([u^2v^2]\rho_1;\rho_2;q,t)
 \mathcal J(\rho_1,\rho_2),
\end{align*}
where
\begin{align*}
 \mathcal J(\rho_1,\rho_2)
 :=\sum_{\mu,\tau,\sigma\in\YY}
 \mathbf 1_{c_l}(\tau)\mathbf 1_{c_r}(\mu)
 \frac{b_\mu^{c_r}}{\overline b_\tau^{\,c_l}}
 u^{|\tau|}v^{|\mu|}
 P_{\sigma/\tau}([u^2v^2]\rho_2;q,t)
 Q_{\sigma/\mu}([u^2v^2]\rho_1;q,t).
\end{align*}
A second application of Lemma~\ref{le17} shows that
\begin{align*}
 \mathcal J(\rho_1,\rho_2)
 =H([u^4v^4]\rho_1;\rho_2;q,t)
 \mathcal I([u^2v^2]\rho_1,[u^2v^2]\rho_2).
\end{align*}
Hence
\begin{align*}
 \mathcal I(\rho_1,\rho_2)
 &=\Theta_{c_l}(u\rho_1;q,t)
 \Theta_{c_r}(u^2v\rho_1;q,t)
 \Theta_{c_r}(v\rho_2;q,t)
 \Theta_{c_l}(uv^2\rho_2;q,t)\\
 &\qquad\times H([u^2v^2]\rho_1;\rho_2;q,t)
 H([u^4v^4]\rho_1;\rho_2;q,t)
 \mathcal I([u^2v^2]\rho_1,[u^2v^2]\rho_2).
\end{align*}
Iterating this identity $N$ times yields
\begin{align*}
 \mathcal I(\rho_1,\rho_2)
 &=\left[\prod_{i=1}^{N}
 \Theta_{c_l}([u^{2i-1}v^{2i-2}]\rho_1;q,t)
 \Theta_{c_r}([u^{2i}v^{2i-1}]\rho_1;q,t)\right.\\
 &\left.\Theta_{c_r}([u^{2i-2}v^{2i-1}]\rho_2;q,t)
 \Theta_{c_l}([u^{2i-1}v^{2i}]\rho_2;q,t)
 \right]\\
 &\qquad\times
 \left[\prod_{j=1}^{N}H([u^{2j}v^{2j}]\rho_1;\rho_2;q,t)\right]
 \mathcal I([u^{2N}v^{2N}]\rho_1,[u^{2N}v^{2N}]\rho_2).
\end{align*}
Since $u,v\in(0,1)$, the specializations $[u^{2N}v^{2N}]\rho_1$ and $[u^{2N}v^{2N}]\rho_2$ converge to the zero specialization as $N\to\infty$, and therefore
\begin{align*}
 \lim_{N\to\infty}\mathcal I([u^{2N}v^{2N}]\rho_1,[u^{2N}v^{2N}]\rho_2)
 =\mathfrak Z_{c_l,c_r}(uv).
\end{align*}
Passing to the limit proves \eqref{l218eq}. The case \eqref{drpk} is identical, since the reduction to $\mathcal I(\rho_1,\rho_2)$ is unchanged.  The parity constraints for $deel$ and $eoa$ are preserved at each reflection step by Corollaries~\ref{cl14} and~\ref{c16}; at the terminal step they give precisely the two indicators in \eqref{dZbdry}.
\end{proof}

Combining Lemmas~\ref{l23}, \ref{le25}, \ref{l41}, and \ref{l31} yields the moment formula for the Laplace transform of the height function under the doubly free-boundary Gibbs measure.

\section{Scaling Limits}\label{sect:as}

In this section, we analyze the asymptotic behavior of the observables introduced in
Section~\ref{sect:lthf} and prove that the centered fluctuations are Gaussian in the scaling
limit. The main inputs are the moment formula from Lemma~\ref{l31}, the periodicity
assumptions from Section~\ref{sect:mr}, and the $q$-Pochhammer ratio estimate in
Lemma~\ref{al51}. A central point is that, after a careful reorganization of the integrand,
the one-point free-boundary correction factors converge uniformly to the infinite product
encoded by the functions $\mathcal F_{u,v,k}$; this is the content of Lemma~\ref{le36}.
At the two-point level the quotient by the one-point factors has a non-trivial
second-order limit, which gives the covariance correction kernel $\mathscr B_{u,v}$.

\subsection{The interaction kernel}

\begin{lemma}[Finite non-reflected pair factor between two Negu\c t blocks]
\label{lem:mixed-interaction-kernel}
Let \(Z\) and \(U\) be the ordered Negu\c t variable sets attached to two
marked columns of types \(c,d\in\{L,R\}\).  In the multi-point contour
formula of Lemma~\ref{l31}, the finite non-reflected two-block factor comes
from the ordered-column product
\[
\prod_{l+1\le i<j\le r}
H\bigl(
\rho_{2,a_i}(W^{(i)};q,t);
\rho_{3,a_j}(W^{(j)};q,t);
q,t
\bigr).
\]
For two fixed blocks \(Z=W^{(i)}\) and \(U=W^{(j)}\), with \(i<j\),
\(a_i=c\), and \(a_j=d\), this factor is
\[
 H\bigl(
\rho_{2,c}(Z;q,t);
\rho_{3,d}(U;q,t);
q,t
\bigr)
=
T_{c,d}(Z,U)
=
\prod_{z\in Z}\prod_{u\in U} t_{c,d}(z,u),
\]
where
\[
t_{c,d}(z,u):=
\begin{cases}
\displaystyle
\frac{(z-u)(qz-tu)}{(z-tu)(qz-u)},
& (c,d)=(L,L),\\[3mm]
\displaystyle
\frac{(u+z)^2}{(u+tz)\left(u+\frac{z}{t}\right)},
& (c,d)=(L,R),\\[3mm]
\displaystyle
\frac{(tu+q^2z)(tu+z)}{(tu+qz)^2},
& (c,d)=(R,L),\\[3mm]
\displaystyle
\frac{(u-qz)(tu-z)}{(u-z)(tu-qz)},
& (c,d)=(R,R).
\end{cases}
\]
This factor is the finite bulk/Negu\c t interaction between the two marked
blocks; it does not include the boundary reflection factors coming from the
two free endpoints in Lemma~\ref{l31}.
\end{lemma}

\begin{proof}
 Since \(H\) is multiplicative over unions of one-variable
specializations, the contribution of two fixed blocks factors over pairs
\((z,u)\in Z\times U\).  It is therefore enough to compute the contribution
of one pair.

For \((c,d)=(L,L)\), the definitions \eqref{drt}--\eqref{drs} and the
\(H\)-kernel identity of Lemma~\ref{l210} give the quotient
\[
        \frac{1-u/z}{1-tu/z}\cdot
        \frac{1-(t/q)u/z}{1-(1/q)u/z}
        =
        \frac{(z-u)(qz-tu)}{(z-tu)(qz-u)}.
\]
This is \(t_{L,L}(z,u)\).

For \((c,d)=(R,R)\), the same computation with conjugated specializations,
using the ordered-column convention in Lemma~\ref{l31}, gives
\[
        t_{R,R}(z,u)
        =
        \frac{(u-qz)(tu-z)}{(u-z)(tu-qz)}.
\]

For the mixed cases, Lemma~\ref{l210} converts the conjugate Cauchy factor
into the corresponding plus-kernel.  The single-pair contributions are
\[
        t_{L,R}(z,u)
        =
        \frac{(u+z)^2}{(u+tz)(u+z/t)}
\]
and
\[
        t_{R,L}(z,u)
        =
        \frac{(tu+q^2z)(tu+z)}{(tu+qz)^2}.
\]
Multiplying these single-pair contributions over all
\((z,u)\in Z\times U\) gives the claimed formula for \(T_{c,d}(Z,U)\).
\end{proof}

\subsection{Moment formulas specialized to dimer coverings}

\begin{lemma}[Moment formula after the dimer specialization]\label{l331}
Let
\[
        i_1\le i_2\le\cdots\le i_m
\]
be integers in \([l+1..r]\), and let \(g_1,\ldots,g_m\in\mathbb Z_{>0}\).
For each occurrence \(s\in[m]\), let \(W^{[s]}\) be an ordered set of variables
with \(|W^{[s]}|=g_s\), attached to the marked column \(i_s\).  Write
\[
        \mathbf i:=(i_1,\ldots,i_m).
\]
All unmarked columns are understood to have empty Negu\c t variable sets.  Thus
the active marked-block specializations are
\[
        \rho_{W,\mathbf i}
        :=
        \bigcup_{s=1}^m
        \rho_{2,a_{i_s}}\!\left(W^{[s]};q,t\right),
        \qquad
        \rho_{W,\mathbf i}^{\circ}
        :=
        \bigcup_{s=1}^m
        \rho_{3,a_{i_s}}\!\left(W^{[s]};q,t\right).
\]
We keep the graph specializations \(\rho_A,\rho_B\) from Lemma~\ref{l31}.
For \(k\ge1\), define
\(\Theta_{k,\mathbf i}=\Theta_{k,\mathbf i}(W^{[1]},\ldots,W^{[m]})\) by
\begin{align}
\Theta_{k,\mathbf i}
&:=
\Theta_{c_l}\!\left([u^{2k-2}v^{2k-1}]\rho_{W,\mathbf i};q,t\right)
\Theta_{c_r}\!\left([u^{2k-1}v^{2k}]\rho_{W,\mathbf i};q,t\right)
\notag\\
&\quad\times
\Theta_{c_l}\!\left([u^{2k-1}v^{2k-2}]\rho_{W,\mathbf i}^{\circ};q,t\right)
\Theta_{c_r}\!\left([u^{2k}v^{2k-1}]\rho_{W,\mathbf i}^{\circ};q,t\right)
\notag\\
&\quad\times
H\!\left(u^{2k}\rho_B;v^{2k}\rho_{W,\mathbf i}^{\circ};q,t\right)
H\!\left(u^{2k-2}\rho_B;v^{2k}\rho_{W,\mathbf i};q,t\right)
H\!\left(u^{2k}\rho_{W,\mathbf i};v^{2k}\rho_{W,\mathbf i}^{\circ};q,t\right)
\notag\\
&\quad\times
H\!\left(u^{2k}\rho_A;v^{2k}\rho_{W,\mathbf i};q,t\right)
H\!\left(u^{2k}\rho_A;v^{2k-2}\rho_{W,\mathbf i}^{\circ};q,t\right).
\label{eq:Theta-k-i}
\end{align}
Then
\begin{small}
\begin{align*}
&\mathbb E_{\Pr}\Biggl[
        \prod_{s=1}^{m}\gamma_{g_s}\bigl(\lambda^{(i_s)};q,t\bigr)
        \Biggr]
\\
&=
\oint\cdots\oint
\prod_{s=1}^{m}
D\bigl(W^{[s]};\omega(q,t,a_{i_s})\bigr)
\\
&\quad\times
\prod_{s=1}^{m}
\prod_{\substack{j\in[i_s..r]\\ b_j=-}}
H\bigl(
\rho_{2,a_{i_s}}(W^{[s]};q,t);
\rho_{\{x_j\},1,a_j};q,t
\bigr)
\\
&\quad\times
\prod_{s=1}^{m}
\prod_{\substack{j\in[l..i_s-1]\\ b_j=+}}
H\bigl(
\rho_{3,a_{i_s}}(W^{[s]};q,t);
\rho_{\{x_j\},1,a_j};q,t
\bigr)
\\
&\quad\times
\prod_{1\le p<q\le m}
T_{a_{i_p},a_{i_q}}\bigl(W^{[p]},W^{[q]}\bigr)
\prod_{k\ge1}\Theta_{k,\mathbf i}.
\end{align*}
\end{small}
Here \(\Pr=\mathbb P_{c_l,c_r}\) is the probability measure \eqref{dpm}, and
the contours are those inherited from Lemma~\ref{l31} under the specialization
of Lemma~\ref{l23}.
\end{lemma}

\begin{proof}
Apply Lemma~\ref{l31} and specialize the generalized Macdonald-process
parameters as in Lemma~\ref{l23}.  All unmarked columns have empty Negu\c t
variable sets, so the global specializations \(\rho_W,\rho_W^\circ\) of
Lemma~\ref{l31} reduce to the active finite unions
\(\rho_{W,\mathbf i},\rho_{W,\mathbf i}^{\circ}\) above.

Under the dimer specialization, the nonzero \(A\)-specializations occur
exactly at indices \(j\) with \(b_j=-\), while the nonzero \(B\)-specializations
occur exactly at indices \(j\) with \(b_j=+\).  This gives the two displayed
one-body products.  The finite interaction between two marked blocks is
\(T_{a_{i_p},a_{i_q}}\) by Lemma~\ref{lem:mixed-interaction-kernel}.  Finally,
the reflected boundary part of Lemma~\ref{l31} is precisely the product over
\(k\ge1\) of the factors \(\Theta_{k,\mathbf i}\) in \eqref{eq:Theta-k-i}.
The normalization is the same as in \eqref{dpm}, so the stated formula follows.
\end{proof}

For later use, we record the explicit one-body kernels coming from the specialization by a
single edge weight:
\begin{align}
G_{1,LL}(W,x;q,t)&:=\prod_{w\in W}\frac{w-\frac{qx}{t}}{w-qx},
&
G_{1,LR}(W,x;q,t)&:=\prod_{w\in W}\frac{tw+q^2x}{tw+qx},
\label{dg1}\\
G_{1,RL}(W,x;q,t)&:=\prod_{w\in W}\frac{w+t^{-1}x}{w+x},
&
G_{1,RR}(W,x;q,t)&:=\prod_{w\in W}\frac{tw-qx}{tw-x},
\label{dg3}\\
G_{0,LL}(W,x;q,t)&:=\prod_{w\in W}\frac{q-wx}{q-twx},
&
G_{0,LR}(W,x;q,t)&:=\prod_{w\in W}\frac{1+wx}{1+q^{-1}wx},
\label{dg2}\\
G_{0,RL}(W,x;q,t)&:=\prod_{w\in W}\frac{1+wx}{1+twx},
&
G_{0,RR}(W,x;q,t)&:=\prod_{w\in W}\frac{1-twx}{1-qtwx}.
\label{dg4}
\end{align}

\begin{lemma}[Left-chart moment formula]\label{lLeftMoments}
Let \(i_1\le\cdots\le i_m\) be integers in \([l+1..r]\), and assume
\[
        a_{i_1}=\cdots=a_{i_m}=L.
\]
For each occurrence \(s\in[m]\), let \(W^{[s]}\) be an ordered set of variables
with \(|W^{[s]}|=g_s\).  Let \(\mathbf i=(i_1,\ldots,i_m)\), and let
\(\Theta_{k,\mathbf i}\) be the active reflected factor defined in
\eqref{eq:Theta-k-i}.  Then
\begin{align*}
&\mathbb E_{\Pr}\Biggl[
        \prod_{s=1}^{m}\gamma_{g_s}\bigl(\lambda^{(i_s)};q,t\bigr)
        \Biggr]
\\
&=
\oint\cdots\oint
\prod_{s=1}^{m}
D\bigl(W^{[s]};q,t\bigr)\,
J_s\bigl(W^{[s]}\bigr)
\prod_{1\le p<q\le m}
T_{L,L}\bigl(W^{[p]},W^{[q]}\bigr)
\prod_{k\ge1}\Theta_{k,\mathbf i},
\end{align*}
where
\begin{align*}
J_s(W)
&:=
\prod_{\substack{j\in[i_s..r]\\ b_j=-,\ a_j=L}}
G_{1,LL}(W,x_j;q,t)
\prod_{\substack{j\in[i_s..r]\\ b_j=-,\ a_j=R}}
G_{1,LR}(W,x_j;q,t)
\\
&\quad\times
\prod_{\substack{j\in[l..i_s-1]\\ b_j=+,\ a_j=L}}
G_{0,LL}(W,x_j;q,t)
\prod_{\substack{j\in[l..i_s-1]\\ b_j=+,\ a_j=R}}
G_{0,LR}(W,x_j;q,t).
\end{align*}
\end{lemma}

\begin{proof}
This is the \(L\)-marked specialization of Lemma~\ref{l331}.  Since
\(a_{i_s}=L\) for every marked occurrence \(s\), the \(D\)-kernel becomes
\(D(W^{[s]};q,t)\), the one-body \(H\)-factors reduce to the four \(G\)-kernels
appearing in the definition of \(J_s\), and the finite interaction between two
marked blocks is always \(T_{L,L}\).  The reflected boundary factors are
unchanged and are exactly the active factors \(\Theta_{k,\mathbf i}\).  This
gives the formula.
\end{proof}

\subsection{Periodic regrouping and asymptotic one-body factors}

Let $\epsilon>0$ and suppose we are in the setting of Assumption~\ref{ap5}. For
$p\in[m]$, $j\in[n]$, and $i\in[l^{(\epsilon)}..r^{(\epsilon)}]$, define
\begin{align*}
I_{j,p,1,L}^{(\epsilon)}
&:=
\Bigl\{
u\in[v_{p-1}^{(\epsilon)}+1..v_{p}^{(\epsilon)}]\cap(n\ZZ+j)\cap[i..r^{(\epsilon)}]:
b_u^{(\epsilon)}=-,\ a_u=L
\Bigr\},\\
I_{j,p,0,L}^{(\epsilon)}
&:=
\Bigl\{
u\in[v_{p-1}^{(\epsilon)}+1..v_{p}^{(\epsilon)}]\cap(n\ZZ+j)\cap[l^{(\epsilon)}..i-1]:
b_u^{(\epsilon)}=+,\ a_u=L
\Bigr\},\\
I_{j,p,1,R}^{(\epsilon)}
&:=
\Bigl\{
u\in[v_{p-1}^{(\epsilon)}+1..v_{p}^{(\epsilon)}]\cap(n\ZZ+j)\cap[i..r^{(\epsilon)}]:
b_u^{(\epsilon)}=-,\ a_u=R
\Bigr\},\\
I_{j,p,0,R}^{(\epsilon)}
&:=
\Bigl\{
u\in[v_{p-1}^{(\epsilon)}+1..v_{p}^{(\epsilon)}]\cap(n\ZZ+j)\cap[l^{(\epsilon)}..i-1]:
b_u^{(\epsilon)}=+,\ a_u=R
\Bigr\}.
\end{align*}

\begin{lemma}\label{aps5}
Suppose Assumption~\ref{ap5} and \eqref{dci} hold. Fix $j\in[n]$ and $p\in[m]$. Then, for
each of the four families
\[
I_{j,p,1,L}^{(\epsilon)},\qquad I_{j,p,0,L}^{(\epsilon)},\qquad
I_{j,p,1,R}^{(\epsilon)},\qquad I_{j,p,0,R}^{(\epsilon)},
\]
either the set is empty for all sufficiently small $\epsilon$, or it is nonempty for all
sufficiently small $\epsilon$.
\end{lemma}

\begin{proof}
This follows immediately from the piecewise periodicity in Assumption~\ref{ap5}(1) and
the fact that the residue class $(i^{(\epsilon)})_{\equiv_n}$ is fixed in~$\epsilon$ by
\eqref{dci}.
\end{proof}

Accordingly, we define the eventual nonemptiness indicators
\begin{align*}
\mathbf 1_{\mathcal E_{j,p,1,L}}
&:=
\begin{cases}
1,& I_{j,p,1,L}^{(\epsilon)}\neq\emptyset\ \text{for all sufficiently small }\epsilon,\\
0,& \text{otherwise},
\end{cases}\\
\mathbf 1_{\mathcal E_{j,p,0,L}}
&:=
\begin{cases}
1,& I_{j,p,0,L}^{(\epsilon)}\neq\emptyset\ \text{for all sufficiently small }\epsilon,\\
0,& \text{otherwise},
\end{cases}\\
\mathbf 1_{\mathcal E_{j,p,1,R}}
&:=
\begin{cases}
1,& I_{j,p,1,R}^{(\epsilon)}\neq\emptyset\ \text{for all sufficiently small }\epsilon,\\
0,& \text{otherwise},
\end{cases}\\
\mathbf 1_{\mathcal E_{j,p,0,R}}
&:=
\begin{cases}
1,& I_{j,p,0,R}^{(\epsilon)}\neq\emptyset\ \text{for all sufficiently small }\epsilon,\\
0,& \text{otherwise}.
\end{cases}
\end{align*}

\begin{lemma}\label{l55}
Suppose Assumption~\ref{ap5} and \eqref{dci} hold, and let
$i^{(\epsilon)}\in[l^{(\epsilon)}..r^{(\epsilon)}]$ satisfy
\[
\epsilon i^{(\epsilon)}\to\chi,\qquad a_{i^{(\epsilon)}}=L.
\]
Assume
\begin{align}
 q=t^{\alpha}
 \label{jacks}
\end{align}
for some fixed $\alpha>0$. Define
\begin{align*}
G_{1,L}^{(\epsilon)}(W)
&:=
\prod_{\substack{j\in[i^{(\epsilon)}..r^{(\epsilon)}]\\ b_j^{(\epsilon)}=-,\,a_j=L}}
G_{1,LL}(W,x_j^{(\epsilon)};q,t),\\
G_{0,L}^{(\epsilon)}(W)
&:=
\prod_{\substack{j\in[l^{(\epsilon)}..i^{(\epsilon)}-1]\\ b_j^{(\epsilon)}=+,\,a_j=L}}
G_{0,LL}(W,x_j^{(\epsilon)};q,t),\\
G_{1,R}^{(\epsilon)}(W)
&:=
\prod_{\substack{j\in[i^{(\epsilon)}..r^{(\epsilon)}]\\ b_j^{(\epsilon)}=-,\,a_j=R}}
G_{1,LR}(W,x_j^{(\epsilon)};q,t),\\
G_{0,R}^{(\epsilon)}(W)
&:=
\prod_{\substack{j\in[l^{(\epsilon)}..i^{(\epsilon)}-1]\\ b_j^{(\epsilon)}=+,\,a_j=R}}
G_{0,LR}(W,x_j^{(\epsilon)};q,t).
\end{align*}
Further define
\begin{align}
\mathcal G_{>\chi,L}(w)
&:=
\prod_{\substack{p\in[m]\\ V_p>\chi}}
\prod_{j=1}^{n}
\left(
\frac{1-e^{V_p}(w\tau_j)^{-1}}
     {1-e^{\max\{V_{p-1},\chi\}}(w\tau_j)^{-1}}
\right)^{\mathbf 1_{\mathcal E_{j,p,1,L}}},
\label{dgs1}\\
\mathcal G_{<\chi,L}(w)
&:=
\prod_{\substack{p\in[m]\\ V_{p-1}<\chi}}
\prod_{j=1}^{n}
\left(
\frac{1-we^{-V_{p-1}}\tau_j}
     {1-we^{-\min\{V_p,\chi\}}\tau_j}
\right)^{\mathbf 1_{\mathcal E_{j,p,0,L}}},
\label{dgs2}\\
\mathcal G_{>\chi,R}(w)
&:=
\prod_{\substack{p\in[m]\\ V_p>\chi}}
\prod_{j=1}^{n}
\left(
\frac{1+e^{\max\{V_{p-1},\chi\}}(w\tau_j)^{-1}}
     {1+e^{V_p}(w\tau_j)^{-1}}
\right)^{\mathbf 1_{\mathcal E_{j,p,1,R}}},
\label{dgs3}\\
\mathcal G_{<\chi,R}(w)
&:=
\prod_{\substack{p\in[m]\\ V_{p-1}<\chi}}
\prod_{j=1}^{n}
\left(
\frac{1+we^{-\min\{V_p,\chi\}}\tau_j}
     {1+we^{-V_{p-1}}\tau_j}
\right)^{\mathbf 1_{\mathcal E_{j,p,0,R}}}.
\label{dgs4}
\end{align}
Then, uniformly for $W$ on a fixed compact set avoiding the singularities,
\begin{align*}
 \lim_{\epsilon\to0}G_{1,L}^{(\epsilon)}(W)
 &=
 \left[\prod_{w\in W}\mathcal G_{>\chi,L}(w)\right]^{\beta},\\
 \lim_{\epsilon\to0}G_{0,L}^{(\epsilon)}(W)
 &=
 \left[\prod_{w\in W}\mathcal G_{<\chi,L}(w)\right]^{\beta},\\
 \lim_{\epsilon\to0}G_{1,R}^{(\epsilon)}(W)
 &=
 \left[\prod_{w\in W}\mathcal G_{>\chi,R}(w)\right]^{\alpha\beta},\\
 \lim_{\epsilon\to0}G_{0,R}^{(\epsilon)}(W)
 &=
 \left[\prod_{w\in W}\mathcal G_{<\chi,R}(w)\right]^{\alpha\beta}.
\end{align*}
We choose the logarithmic branches so that $\log z$ is real on the positive real axis.
\end{lemma}

\begin{proof}
Let $N_{j,p,1,L}^{(\epsilon)}:=|I_{j,p,1,L}^{(\epsilon)}|$, and similarly for the other three
families. Also let $q_{j,p,1,L}^{(\epsilon)}:=\max I_{j,p,1,L}^{(\epsilon)}$ when the set is
nonempty, and use the same convention for the other families; when a set is empty we set
the corresponding endpoint to $\pm\infty$ and interpret the associated weight as~$0$.

Grouping the product according to the periodic blocks and residue classes modulo~$n$
gives, for example,
\begin{align*}
G_{1,L}^{(\epsilon)}(W)
=
\prod_{p=1}^{m}\prod_{j=1}^{n}\prod_{w\in W}
\frac{\left(\frac{q}{t}w^{-1}x_{q_{j,p,1,L}^{(\epsilon)}}^{(\epsilon)};e^{-n\epsilon}\right)_{N_{j,p,1,L}^{(\epsilon)}}}
     {\left(qw^{-1}x_{q_{j,p,1,L}^{(\epsilon)}}^{(\epsilon)};e^{-n\epsilon}\right)_{N_{j,p,1,L}^{(\epsilon)}}},
\end{align*}
and analogous formulas for $G_{0,L}^{(\epsilon)}$, $G_{1,R}^{(\epsilon)}$, and
$G_{0,R}^{(\epsilon)}$. Here 
\begin{align}
(a;q)_{N}=\prod_{i=0}^{N-1}(1-aq^{i}).
\end{align}

Applying Lemma~\ref{al51} and using Assumption~\ref{ap5}(2) we
obtain
\begin{align*}
n\epsilon N_{j,p,1,L}^{(\epsilon)}&\to V_p-\max\{V_{p-1},\chi\},\\
n\epsilon N_{j,p,0,L}^{(\epsilon)}&\to \min\{V_p,\chi\}-V_{p-1},\\
n\epsilon N_{j,p,1,R}^{(\epsilon)}&\to V_p-\max\{V_{p-1},\chi\},\\
n\epsilon N_{j,p,0,R}^{(\epsilon)}&\to \min\{V_p,\chi\}-V_{p-1},
\end{align*}
whenever the corresponding sets are eventually nonempty. At the same time, when the corresponding set is nonempty,
\begin{align*}
x_{q_{j,p,1,L}^{(\epsilon)}}^{(\epsilon)}
&\to \tau_j^{-1}e^{V_p},\\
x_{q_{j,p,0,L}^{(\epsilon)}}^{(\epsilon)}
&\to \tau_je^{-\min\{V_p,\chi\}},\\
x_{q_{j,p,1,R}^{(\epsilon)}}^{(\epsilon)}
&\to \tau_j^{-1}e^{V_p},\\
x_{q_{j,p,0,R}^{(\epsilon)}}^{(\epsilon)}
&\to \tau_je^{-\min\{V_p,\chi\}}.
\end{align*}
Combining these limits gives the stated formulas.
\end{proof}

For $w\in\CC$ and $\chi\in\RR$, define
\begin{align}
 \mathcal G_{\chi}(w)
 :=
 \mathcal G_{>\chi,L}(w)\,
 \mathcal G_{<\chi,L}(w)\,
 \Bigl[\mathcal G_{>\chi,R}(w)\mathcal G_{<\chi,R}(w)\Bigr]^{\alpha}.
 \label{dgc}
\end{align}

The following zero/pole sets will be used to specify the limiting contours:
\begin{align*}
\mathcal R_{\chi,1,1}
&:=
\left\{e^{\max\{V_{p-1},\chi\}}\tau_j^{-1}\right\}_{\substack{V_p>\chi,\ j\in[n]\\ \mathbf 1_{\mathcal E_{j,p,1,L}}=1}},
&
\mathcal R_{\chi,1,2}
&:=
\left\{e^{V_p}\tau_j^{-1}\right\}_{\substack{V_p>\chi,\ j\in[n]\\ \mathbf 1_{\mathcal E_{j,p,1,L}}=1}},\\
\mathcal R_{\chi,2,1}
&:=
\left\{e^{\min\{V_p,\chi\}}\tau_j^{-1}\right\}_{\substack{V_{p-1}<\chi,\ j\in[n]\\ \mathbf 1_{\mathcal E_{j,p,0,L}}=1}},
&
\mathcal R_{\chi,2,2}
&:=
\left\{e^{V_{p-1}}\tau_j^{-1}\right\}_{\substack{V_{p-1}<\chi,\ j\in[n]\\ \mathbf 1_{\mathcal E_{j,p,0,L}}=1}},\\
\mathcal R_{\chi,3,1}
&:=
\left\{-e^{V_p}\tau_j^{-1}\right\}_{\substack{V_p>\chi,\ j\in[n]\\ \mathbf 1_{\mathcal E_{j,p,1,R}}=1}},
&
\mathcal R_{\chi,3,2}
&:=
\left\{-e^{\max\{V_{p-1},\chi\}}\tau_j^{-1}\right\}_{\substack{V_p>\chi,\ j\in[n]\\ \mathbf 1_{\mathcal E_{j,p,1,R}}=1}},\\
\mathcal R_{\chi,4,1}
&:=
\left\{-e^{V_{p-1}}\tau_j^{-1}\right\}_{\substack{V_{p-1}<\chi,\ j\in[n]\\ \mathbf 1_{\mathcal E_{j,p,0,R}}=1}},
&
\mathcal R_{\chi,4,2}
&:=
\left\{-e^{\min\{V_p,\chi\}}\tau_j^{-1}\right\}_{\substack{V_{p-1}<\chi,\ j\in[n]\\ \mathbf 1_{\mathcal E_{j,p,0,R}}=1}}.
\end{align*}
For $k_1,k_2\in\ZZ_{\ge0}$, define also
\begin{align*}
\mathcal R_{5,1,k_1,k_2}
&:=
\left\{u^{2k_1}v^{2k_2}e^{V_{p-1}}\tau_j^{-1}\right\}_{\substack{p\in[m],\ j\in[n]\\ \mathbf 1_{\mathcal E_{j,p,1,L}}=1}},\\
\mathcal R_{5,2,k_1,k_2}
&:=
\left\{u^{2k_1}v^{2k_2}e^{V_p}\tau_j^{-1}\right\}_{\substack{p\in[m],\ j\in[n]\\ \mathbf 1_{\mathcal E_{j,p,1,L}}=1}},\\
\mathcal R_{6,1,k_1,k_2}
&:=
\left\{u^{-2k_1}v^{-2k_2}e^{V_p}\tau_j^{-1}\right\}_{\substack{p\in[m],\ j\in[n]\\ \mathbf 1_{\mathcal E_{j,p,0,L}}=1}},\\
\mathcal R_{6,2,k_1,k_2}
&:=
\left\{u^{-2k_1}v^{-2k_2}e^{V_{p-1}}\tau_j^{-1}\right\}_{\substack{p\in[m],\ j\in[n]\\ \mathbf 1_{\mathcal E_{j,p,0,L}}=1}},\\
\mathcal R_{7,1,k_1,k_2}
&:=
\left\{-u^{2k_1}v^{2k_2}e^{V_p}\tau_j^{-1}\right\}_{\substack{p\in[m],\ j\in[n]\\ \mathbf 1_{\mathcal E_{j,p,1,R}}=1}},\\
\mathcal R_{7,2,k_1,k_2}
&:=
\left\{-u^{2k_1}v^{2k_2}e^{V_{p-1}}\tau_j^{-1}\right\}_{\substack{p\in[m],\ j\in[n]\\ \mathbf 1_{\mathcal E_{j,p,1,R}}=1}},\\
\mathcal R_{8,1,k_1,k_2}
&:=
\left\{-u^{-2k_1}v^{-2k_2}e^{V_{p-1}}\tau_j^{-1}\right\}_{\substack{p\in[m],\ j\in[n]\\ \mathbf 1_{\mathcal E_{j,p,0,R}}=1}},\\
\mathcal R_{8,2,k_1,k_2}
&:=
\left\{-u^{-2k_1}v^{-2k_2}e^{V_p}\tau_j^{-1}\right\}_{\substack{p\in[m],\ j\in[n]\\ \mathbf 1_{\mathcal E_{j,p,0,R}}=1}}.
\end{align*}

\subsection{One-point asymptotics}

We shall obtain asymptotic formulas only for columns of type~$L$. This is enough for the
proof of the Gaussian fluctuation theorem, and the case of type~$R$ follows from the same
argument together with Lemma~\ref{l41}.
Let
\begin{small}
\begin{align}
\mathcal F_{u,v,k}(w)
&:=
\mathcal G_{>V_0,L}(u^{-2k}v^{-2k}w)\,
\mathcal G_{<V_m,L}(u^{2k}v^{2k}w)\,
\Bigl[
\mathcal G_{>V_0,R}(u^{-2k}v^{-2k}w)\,
\mathcal G_{<V_m,R}(u^{2k}v^{2k}w)
\Bigr]^{\alpha}
\notag\\
&\qquad\times
\mathcal G_{>V_0,L}(u^{2-2k}v^{-2k}w)\,
\mathcal G_{<V_m,L}(u^{2k}v^{2k-2}w)\,
\Bigl[
\mathcal G_{>V_0,R}(u^{2-2k}v^{-2k}w)\,
\mathcal G_{<V_m,R}(u^{2k}v^{2k-2}w)
\Bigr]^{\alpha}.
\label{dfuvk}
\end{align}
\end{small}

\begin{lemma}\label{le36}
Fix \(g\in\mathbb Z_{>0}\), and let
\[
        W=(w_1,\ldots,w_g)
\]
range over the product of an admissible one-point Negu\c t contour family.  We assume that
the contour paths are contained in a fixed compact annulus
\[
        \{R_-<|w|<R_+\}.
\]

Fix the marked column \(i=i^{(\epsilon)}\).  In the one-point specialization of
Lemma~\ref{l31}, we take
\[
        W^{(i)}=W,
        \qquad
        W^{(j)}=\varnothing \quad \text{for } j\ne i.
\]
Set
\[
        \rho_{W,i}:=\rho_{2,a_i}(W;q,t),
        \qquad
        \rho_{W,i}^{\circ}:=\rho_{3,a_i}(W;q,t).
\]
The graph specializations \(\rho_A,\rho_B\) are the full graph specializations from
Lemma~\ref{l31}, after the dimer specialization of Lemma~\ref{l23}.

For \(k\ge1\), define the \(k\)-th reflected one-point correction factor by
\begin{align}
\Theta_k^{(\epsilon)}(W)
&:=
\Theta_{c_l}\!\left([u^{2k-2}v^{2k-1}]\rho_{W,i};q,t\right)
\Theta_{c_r}\!\left([u^{2k-1}v^{2k}]\rho_{W,i};q,t\right)
\notag\\
&\quad\times
\Theta_{c_l}\!\left([u^{2k-1}v^{2k-2}]\rho_{W,i}^{\circ};q,t\right)
\Theta_{c_r}\!\left([u^{2k}v^{2k-1}]\rho_{W,i}^{\circ};q,t\right)
\notag\\
&\quad\times
H\!\left(u^{2k}\rho_B;v^{2k}\rho_{W,i}^{\circ};q,t\right)
H\!\left(u^{2k-2}\rho_B;v^{2k}\rho_{W,i};q,t\right)
H\!\left(u^{2k}\rho_{W,i};v^{2k}\rho_{W,i}^{\circ};q,t\right)
\notag\\
&\quad\times
H\!\left(u^{2k}\rho_A;v^{2k}\rho_{W,i};q,t\right)
H\!\left(u^{2k}\rho_A;v^{2k-2}\rho_{W,i}^{\circ};q,t\right).
\label{eq:Theta-epsilon-k}
\end{align}
Here \(q,t,\rho_A,\rho_B,\rho_{W,i},\rho_{W,i}^{\circ}\) depend on
\(\epsilon\), but this dependence is suppressed from the notation.

If \(u,v\in(0,1)\) are sufficiently small, then:
\begin{enumerate}
\item for each fixed \(k\ge1\),
\[
        \Theta_k^{(\epsilon)}(W)
        \longrightarrow
        \prod_{s=1}^{g}\bigl[\mathcal F_{u,v,k}(w_s)\bigr]^\beta
\]
uniformly on the product contour;

\item there exist constants \(C>0\) and \(\mathfrak r\in(0,1)\), independent of
\(k\) and \(\epsilon\), such that for all sufficiently small \(\epsilon>0\),
\[
        \sup_W\bigl|\log \Theta_k^{(\epsilon)}(W)\bigr|
        \le C\mathfrak r^k;
\]

\item consequently,
\[
        \prod_{k\ge1}\Theta_k^{(\epsilon)}(W)
        \longrightarrow
        \prod_{s=1}^{g}\prod_{k\ge1}
        \bigl[\mathcal F_{u,v,k}(w_s)\bigr]^\beta
\]
uniformly on the product contour.
\end{enumerate}
\end{lemma}

\begin{proof}
Throughout the proof we omit the repeated parameters \((q,t)\) from
\(H\), \(\Theta_\ast\), and the specializations; all of them are evaluated at
the finite-\(\epsilon\) parameters
\[
        q=t^\alpha,\qquad t=e^{-n\beta\epsilon}.
\]

We first record the two power-sum estimates used below.  Let \(d\ge1\).  Since
the \(W\)-contours lie in a fixed compact annulus, there is a constant
\(A>1\) such that
\[
        |w_s^d|+|w_s^{-d}|\le A^d,
        \qquad s=1,\ldots,g .
\]
In this one-point specialization only the marked block \(W\) is active, so
\(\rho_{W,i}\) and \(\rho_{W,i}^{\circ}\) contain only the fixed set
\(W=(w_1,\ldots,w_g)\), not a union over \(O(\epsilon^{-1})\) columns.
Using \eqref{drt}--\eqref{drs} and the power-sum rules for scaled and
\((-1)\)-specializations, we have
\[
\begin{aligned}
p_d(\rho_{2,L}(W;q,t))
&=
\left(\frac qt\right)^d(q^d-1)\,p_d(W^{-1}),\\
p_d(\rho_{3,L}(W;q,t))
&=
(1-q^{-d})\,p_d(W),\\
p_d(\rho_{2,R}(W;q,t))
&=
(-1)^{d-1}\frac{1-q^d}{1-t^d}(t^{-d}-1)\,p_d(W^{-1}),\\
p_d(\rho_{3,R}(W;q,t))
&=
(-1)^{d-1}(1-q^d)\,p_d(W).
\end{aligned}
\]
Under the exact Jack scaling, the difference factors
\[
        q^d-1,\qquad 1-q^{-d},\qquad t^{-d}-1,\qquad 1-q^d
\]
are bounded by \(C A^d\epsilon\), after increasing \(A\).  Also
\[
        \left|\frac{1-q^d}{1-t^d}\right|\le C A^d .
\]
Hence there exist constants \(C_0,A_0>0\), independent of \(d,k,\epsilon\), and
\(W\), such that
\begin{equation}\label{W-small-bound}
        |p_d(\rho_{W,i})|+|p_d(\rho_{W,i}^{\circ})|
        \le C_0A_0^d\epsilon .
\end{equation}

For the graph specializations there is no difference factor.  Under the dimer
specialization of Lemma~\ref{l23}, the nonzero parts of \(\rho_A\) and
\(\rho_B\) are the graph one-variable specializations.  For a single graph
variable \(x\),
\[
        p_d(\rho_{\{x\},1,L})=x^d,
        \qquad
        p_d(\rho_{\{x\},1,R})
        =
        (-1)^{d-1}\frac{1-q^d}{1-t^d}x^d .
\]
The periodic weight assumption on bounded macroscopic intervals gives
\[
        |x_j|^d\le C A^d
\]
uniformly in \(j\) and \(\epsilon\).  Since the number of graph variables is
\(O(\epsilon^{-1})\), after increasing \(C_0,A_0\) we get
\begin{equation}\label{AB-bound}
        \epsilon\bigl(|p_d(\rho_A)|+|p_d(\rho_B)|\bigr)
        \le C_0A_0^d .
\end{equation}

We now prove the summable bound in \(k\).  By Lemma~\ref{le110}, each
boundary factor \(\Theta_\ast\) in \eqref{eq:Theta-epsilon-k} is the
exponential of a linear-quadratic expression in the power sums of its
argument.  If
\[
        \eta=\sigma_k\rho_{W,i}
        \quad\text{or}\quad
        \eta=\sigma_k\rho_{W,i}^{\circ},
\]
where \(\sigma_k\) is one of
\[
        u^{2k-2}v^{2k-1},\quad
        u^{2k-1}v^{2k},\quad
        u^{2k-1}v^{2k-2},\quad
        u^{2k}v^{2k-1},
\]
then \eqref{W-small-bound} gives
\[
        |p_d(\eta)|
        \le C_0(A_0\sigma_k)^d\epsilon .
\]
Choosing \(u,v\) sufficiently small, the corresponding linear and quadratic
power-sum series are uniformly summable.  Thus the logarithm of the product
of the four boundary \(\Theta\)-factors is bounded by
\[
        C\epsilon\mathfrak r^k
\]
for some \(\mathfrak r\in(0,1)\).  In particular, these boundary factors tend
to \(1\) for fixed \(k\), and satisfy the weaker bound \(C\mathfrak r^k\)
uniformly in \(\epsilon\).

For the Cauchy factors we use
\[
        \log H(\rho_1;\rho_2)
        =
        \sum_{d\ge1}
        \frac{1-t^d}{1-q^d}
        \frac{p_d(\rho_1)p_d(\rho_2)}{d}.
\]
The factor
\[
        H(u^{2k}\rho_{W,i};v^{2k}\rho_{W,i}^{\circ})
\]
has two \(W\)-type specializations.  By \eqref{W-small-bound}, its logarithm
is \(O(\epsilon^2\mathfrak r^k)\), after decreasing \(\mathfrak r\) if necessary.

The remaining four \(H\)-factors are different.  In each of them, one
specialization is of graph type and the other is of \(W\)-type.  Thus the
graph side contributes \(O(\epsilon^{-1})\) by \eqref{AB-bound}, while the
\(W\)-side contributes \(O(\epsilon)\) by \eqref{W-small-bound}.  These powers
of \(\epsilon\) cancel.  For instance,
\[
p_d(u^{2k}\rho_A)\,p_d(v^{2k}\rho_{W,i})
=
u^{2kd}v^{2kd}\,
p_d(\rho_A)p_d(\rho_{W,i})
\]
is bounded by a geometrically summable majorant in \(d\), uniformly in
\(\epsilon\).  The same argument applies to the other three graph--\(W\)
factors, with scale factors
\[
        u^{2k}v^{2k},\qquad
        u^{2k-2}v^{2k},\qquad
        u^{2k}v^{2k-2}.
\]
For \(u,v\) sufficiently small, these scale factors give geometric decay in
\(k\).  Combining the four boundary terms, the \(W\)--\(W\) Cauchy factor, and
the four graph--\(W\) Cauchy factors, we obtain
\[
        \sup_W\bigl|\log\Theta_k^{(\epsilon)}(W)\bigr|
        \le C\mathfrak r^k ,
\]
for all sufficiently small \(\epsilon\).  This proves \textup{(2)}.

It remains to identify the fixed-\(k\) limit.  The four boundary factors have
logarithms \(O(\epsilon)\), and the \(W\)--\(W\) Cauchy factor has logarithm
\(O(\epsilon^2)\).  Hence these five factors converge to \(1\).  The only
nontrivial limits come from the four graph--\(W\) factors
\[
H(u^{2k}\rho_B;v^{2k}\rho_{W,i}^{\circ}),\qquad
H(u^{2k-2}\rho_B;v^{2k}\rho_{W,i}),
\]
\[
H(u^{2k}\rho_A;v^{2k}\rho_{W,i}),\qquad
H(u^{2k}\rho_A;v^{2k-2}\rho_{W,i}^{\circ}).
\]
By Lemma~\ref{l55}, the logarithms of these four factors have uniform
periodic Riemann-sum limits on the admissible contour neighborhoods.  Splitting
the graph columns into \(L\)- and \(R\)-types gives precisely the eight
reflected products appearing in the definition \eqref{dfuvk} of
\(\mathcal F_{u,v,k}\).  Therefore, for each fixed \(k\),
\[
        \Theta_k^{(\epsilon)}(W)
        \longrightarrow
        \prod_{s=1}^{g}
        \bigl[\mathcal F_{u,v,k}(w_s)\bigr]^\beta
\]
uniformly on the product contour.  This proves \textup{(1)}.

Finally, the estimate in \textup{(2)} gives a summable majorant in \(k\),
uniformly for all sufficiently small \(\epsilon\) and uniformly on the product
contour.  Hence, for any \(N\),
\[
        \prod_{k=1}^{N}\Theta_k^{(\epsilon)}(W)
        \longrightarrow
        \prod_{s=1}^{g}\prod_{k=1}^{N}
        \bigl[\mathcal F_{u,v,k}(w_s)\bigr]^\beta
\]
uniformly, and the tails are controlled by
\[
        \sum_{k>N} C\mathfrak r^k .
\]
Letting first \(\epsilon\to0\) and then \(N\to\infty\) proves \textup{(3)}.
\end{proof}

\noindent\textbf{One-point pole sets.}
For the one-point asymptotics below, we use the limiting pole families
\(\mathcal R_{\chi,a,b}\) and \(\mathcal R_{j,a,k_1,k_2}\) defined above.
The \emph{required} limiting pole set is
\begin{align}
\mathcal D_{\chi}
&:=
(\mathcal R_{\chi,1,1}\setminus\mathcal R_{\chi,1,2})
\cup
(\mathcal R_{\chi,3,1}\setminus\mathcal R_{\chi,3,2})
\notag\\
&\quad\cup
\bigcup_{k\ge1}
\Bigl[
(\mathcal R_{5,1,k,k}\setminus\mathcal R_{5,2,k,k})
\cup
(\mathcal R_{7,1,k,k}\setminus\mathcal R_{7,2,k,k})
\notag\\
&\hspace{30mm}\cup
(\mathcal R_{5,1,k-1,k}\setminus\mathcal R_{5,2,k-1,k})
\cup
(\mathcal R_{7,1,k-1,k}\setminus\mathcal R_{7,2,k-1,k})
\Bigr].
\label{eq:required-poles}
\end{align}

The set \(\mathcal D_\chi\) consists of the limiting poles which must be
enclosed by the one-point Negu\c t contour.  

\noindent\textbf{Excluded one-point family.}
Besides the required pole set \(\mathcal D_\chi\), we use the following
excluded point family:
\begin{align}
\mathcal E_\chi
&:=
\mathcal R_{\chi,2,1}\cup\mathcal R_{\chi,2,2}
\cup
\mathcal R_{\chi,4,1}\cup\mathcal R_{\chi,4,2}
\notag\\
&\quad\cup
\bigcup_{k\ge1}
\Bigl[
\mathcal R_{6,1,k,k}\cup\mathcal R_{6,2,k,k}
\cup
\mathcal R_{8,1,k,k}\cup\mathcal R_{8,2,k,k}
\notag\\
&\hspace{32mm}\cup
\mathcal R_{6,1,k,k-1}\cup\mathcal R_{6,2,k,k-1}
\cup
\mathcal R_{8,1,k,k-1}\cup\mathcal R_{8,2,k,k-1}
\Bigr].
\label{rp}
\end{align}
The elements of \(\mathcal E_\chi\) are the points which must remain outside
the one-point contour.  They are not assumed to be all poles.

\medskip

\noindent\textbf{Contour and branch convention for the one-point asymptotics.}
In Proposition~\ref{p57} and in the subsequent one-point limit-shape and
Stieltjes-transform formulae, we work under the one-point contour and branch
admissibility hypotheses of
Assumption~\ref{ass:contour-branch-admissibility}.  Thus the admissible
one-point contour \(\mathcal C_\chi\) separates the required pole set
\(\mathcal D_\chi\) from the excluded point family \(\mathcal E_\chi\).  The
integration path lies in a pole-free annular neighborhood.  It may enclose
infinitely many reflected required poles accumulating at \(0\), while excluded
reflected points may escape to \(\infty\).

On the annular neighborhood of \(\mathcal C_\chi\), the
function
\[
        S_\chi(w)
        =
        \mathcal G_\chi(w)\prod_{k\ge1}\mathcal F_{u,v,k}(w)
\]
is required to be holomorphic and non-vanishing, and all powers of \(S_\chi\)
used below are taken with the admissible logarithm specified in
Assumption~\ref{ass:contour-branch-admissibility}.  

\subsection{Contour and branch admissibility}
\label{subsec:contour-branch-admissibility}

The one-point and multi-point contour formulae below are stated under the
following analytic admissibility hypothesis.  

\begin{assumption}[Contour and zero-winding branch admissibility]
\label{ass:contour-branch-admissibility}
Fix boundary types \(c_l,c_r\in\{el,oa,deel,eoa\}\) and
\(u,v\in(0,1)\).

\begin{enumerate}[label=\textup{(\roman*)}]
\item \textup{One-point contours.}
For every macroscopic \(L\)-type marked column \(\chi\) used below, there is a
positively oriented contour \(\mathcal C_\chi\) enclosing \(0\) and separating
the one-point families
\[
        \mathcal D_\chi\subset \operatorname{Int}(\mathcal C_\chi),
        \qquad
        \mathcal E_\chi\cap \overline{\operatorname{Int}(\mathcal C_\chi)}
        =\varnothing .
\]
Moreover, there is an open annular neighborhood \(U_\chi\) of
\(\mathcal C_\chi\), independent of \(\epsilon\) for all sufficiently small
\(\epsilon\), which contains no finite-\(\epsilon\) pole of the one-point
integrand and no pole of the limiting one-point integrand.

\item \textup{Branches.}
On \(U_\chi\), the function
\[
        S_\chi(w):=\mathcal G_\chi(w)\prod_{k\ge1}\mathcal F_{u,v,k}(w)
\]
is holomorphic and non-vanishing.  Moreover,
\begin{align}
        \frac{1}{2\pi\mathbf i}
        \oint_{\mathcal C_\chi}\frac{S_\chi'(w)}{S_\chi(w)}\,dw=0.\label{hmbc}
\end{align}
Thus \(S_\chi\) admits a single-valued logarithm on \(U_\chi\), denoted
\(\mathrm{Log}_{\chi,\mathcal C_\chi}S_\chi\), and all non-integer powers
below are defined through this logarithm.

\item \textup{Multi-point contours.}
For every finite family of \(L\)-type marked columns used in the Gaussian
theorem, the admissible one-point contours may be chosen simultaneously,
nested when necessary, so that each one-point pole-separation and branch
condition is preserved and every product contour avoids the two-variable
singularities of the \(LL\)-covariance kernel.  The branch choices are kept
fixed along the contour deformations used in the proof.
\end{enumerate}
\end{assumption}

\begin{proposition}[Automatic one-point branch from separated zero--pole data]
\label{prop:automatic-branch-jack}
Fix a macroscopic \(L\)-type marked location
\[
   \chi\in (V_0,V_m)\setminus\{V_1,\ldots,V_{m-1}\}.
\]
For \(K\ge1\), let
\[
   S_{\chi,K}(w)
   :=
\mathcal{G}_\chi(w)\prod_{k=1}^{K}\mathcal{F}_{u,v,k}(w)
\]
 and let
\[
   \omega_{\chi,K}
   :=
   d\log S_{\chi,K}
\]
be its logarithmic-derivative one-form, obtained by differentiating the
logarithms of the elementary factors.

Assume that there exist a positively oriented simple contour \(C_\chi\), a
connected annular neighbourhood \(U_\chi\Subset\mathbb C^\ast\) of \(C_\chi\),
and \(K_0\ge1\), such that the following hold.

\begin{enumerate}[label=\textup{(\alph*)}]
\item The contour separates the required and excluded one-point families:
\[
   \mathcal D_\chi\subset \operatorname{Int}(C_\chi),
   \qquad
   \mathcal E_\chi\cap \operatorname{Int}(C_\chi)=\varnothing .
\]

\item The closure of \(U_\chi\) contains no limiting zero or pole location of
the one-point factors.  For every \(K\ge K_0\), no zero or pole of the finite
truncation lies on \(U_\chi\).  Moreover,
\[
   \omega_{\chi,K}\longrightarrow \omega_\chi
\]
locally uniformly on \(U_\chi\), where
\[
   \omega_\chi
   :=
   d\log G_\chi+\sum_{k\ge1}d\log F_{u,v,k}.
\]
In the contour neighbourhoods used below, this convergence follows from
Lemma~\ref{le36}.

\item For every \(K\ge K_0\), the finite logarithmic derivative has a
factorization by exponent classes:
\[
   \omega_{\chi,K}
   =
   \alpha\left(
      \frac{A_K'(w)}{A_K(w)}
      -
      \frac{B_K'(w)}{B_K(w)}
   \right)dw
   +
   \left(
      \frac{C_K'(w)}{C_K(w)}
      -
      \frac{D_K'(w)}{D_K(w)}
   \right)dw,
\]
where \(A_K,B_K,C_K,D_K\) are polynomials with no zero on \(U_\chi\).  Their
zeros inside \(C_\chi\), counted with multiplicity, cancel separately in the
two exponent classes:
\[
   N_A(C_\chi)=N_B(C_\chi),
   \qquad
   N_C(C_\chi)=N_D(C_\chi).
\]
\end{enumerate}

Then
\[
   \frac{1}{2\pi i}\oint_{C_\chi}\omega_\chi=0.
\]
Consequently \(\omega_\chi\) has a single-valued holomorphic primitive on
\(U_\chi\).  After fixing its value at one point of \(U_\chi\), this primitive
defines a holomorphic non-vanishing branch
\[
   S_\chi^{U}(w)
   :=
   \exp\!\left(\int^w\omega_\chi\right)
\]
of the formal infinite product
\[
   G_\chi(w)\prod_{k\ge1}F_{u,v,k}(w)
\]
on \(U_\chi\).  Hence the one-point logarithmic-branch condition (\ref{hmbc}) in
Assumption~\ref{ass:contour-branch-admissibility}\textup{(ii)} holds at
\(\chi\).
\end{proposition}

\begin{proof}
For finite \(K\), the logarithmic derivative on \(U_\chi\) is
\[
   \frac{S_{\chi,K}'(w)}{S_{\chi,K}(w)}
   =
   \alpha\left(
      \frac{A_K'(w)}{A_K(w)}
      -
      \frac{B_K'(w)}{B_K(w)}
   \right)
   +
   \left(
      \frac{C_K'(w)}{C_K(w)}
      -
      \frac{D_K'(w)}{D_K(w)}
   \right).
\]
Therefore, by the argument principle,
\[
\begin{aligned}
   \frac{1}{2\pi i}
   \oint_{C_\chi}
   \frac{S_{\chi,K}'(w)}{S_{\chi,K}(w)}\,dw
   &=
   \alpha\bigl(N_A(C_\chi)-N_B(C_\chi)\bigr)
   +
   \bigl(N_C(C_\chi)-N_D(C_\chi)\bigr)  \\
   &=0 .
\end{aligned}
\]
This calculation is made at the level of logarithmic derivatives and therefore
does not require \(\alpha\) to be an integer.

By the normal convergence of the reflected products on \(U_\chi\),
\[
   S_{\chi,K}\longrightarrow S_\chi
\]
locally uniformly on \(U_\chi\).  Since \(\overline U_\chi\) is disjoint from
the limiting zero and pole locations, the limit \(S_\chi\) is holomorphic and
non-vanishing on \(U_\chi\).  The logarithmic derivatives converge uniformly
on \(C_\chi\), for instance by Cauchy's estimates applied on a slightly
smaller annular neighbourhood.  Passing to the limit in the finite-\(K\)
zero-winding identity gives (\ref{hmbc}).
Thus \(S_\chi\) has zero winding on the annular neighbourhood \(U_\chi\), and
therefore admits a single-valued holomorphic logarithm there.
\end{proof}

\begin{proposition}[Multi-point admissibility from annular one-point contours]
\label{prp59}
Fix finitely many \(L\)-type marked locations
\[
   \chi_1,\ldots,\chi_s .
\]
Assume that for each \(\chi_d\) the one-point contour and branch conditions
of Assumption~\ref{ass:contour-branch-admissibility}\textup{(i)--(ii)} hold.
Assume moreover that the one-point contours may be chosen as pairwise disjoint
simple contours
\[
   C_1,\ldots,C_s
\]
contained in the annulus
\begin{align}
   A_{u,v}:=\{w\in\mathbb C:\ v<|w|<u^{-1}\},\label{dauv}
\end{align}
with sufficiently small annular neighbourhoods preserving the one-point
separation and branch conditions.  Repeated marked locations are represented by
sufficiently close nested parallel copies.

Then the contours satisfy Assumption~\ref{ass:contour-branch-admissibility}
\textup{(iii)}.  In particular, for every \(d\ne e\), the product contour
\(C_d\times C_e\) avoids all two-variable singularity loci of the
\(LL\)-covariance kernel.
\end{proposition}

\begin{proof}
Let \(z\in C_d\), \(w\in C_e\), with \(d\ne e\).  Since the contours are disjoint,
\(z\ne w\).  Put
\[
   \mathfrak q=(uv)^2,\qquad
   c_r=\mathfrak q^r,\qquad
   a_r=v^2\mathfrak q^{r-1},\qquad
   b_r=u^2\mathfrak q^{r-1},
   \qquad r\ge1 .
\]
Because \(C_d,C_e\subset A_{u,v}\), we have
\[
   v<|z|,|w|<u^{-1}.
\]
Therefore
\[
   |c_r w|
   <
   c_r u^{-1}
   =
   v(uv)^{2r-1}
   <
   v
   <
   |z|,
\]
so \(z\ne c_r w\).  Interchanging \(z\) and \(w\) gives \(w\ne c_r z\).
Moreover,
\[
   |zw|>v^2\ge v^2\mathfrak q^{r-1}=a_r,
\]
so \(zw\ne a_r\).  Finally,
\[
   |b_rzw|
   <
   b_r u^{-2}
   =
   \mathfrak q^{r-1}
   \le 1,
\]
with strict inequality in the first step; hence \(b_rzw\ne1\).  Thus none of
the \(LL\)-interaction singularities
\[
   z=w,\qquad z=c_r w,\qquad w=c_r z,\qquad zw=a_r,\qquad b_rzw=1
\]
occurs on \(C_d\times C_e\).  

The one-point pole-separation and branch conditions are preserved by the
choice of annular neighbourhoods.  Since all product contours are compact and
the limiting \(LL\)-singular loci are avoided by a positive distance, the
finite-\(\epsilon\) two-variable singularities in the Macdonald contour formula
are also avoided for all sufficiently small \(\epsilon\).  Hence the simultaneous
multi-point contour condition in Assumption~\ref{ass:contour-branch-admissibility}
\textup{(iii)} holds.
\end{proof}

\begin{remark}[Concrete separated zero--pole criteria]
\label{rem:separated-zero-pole-criteria}
Assumption~\ref{ass:contour-branch-admissibility} is stated in analytic form
because the one-point and multi-point asymptotic theorems only need the existence
of suitable contours, logarithmic branches, and product contours.  Propositions
\ref{prop:automatic-branch-jack} and
\ref{prp59} give convenient sufficient criteria for
the two nontrivial parts of this assumption.

A concrete source of these criteria is the separated zero--pole setup of the
Schur two-free-boundary model.  More precisely, suppose that the finite
truncations have the separated inward/outward zero--pole family structure of
\cite[Assumptions~5.7--5.8 and Lemma~5.9]{zl23}.  Then the family order
there gives, inside each separating one-point contour \(C_\chi\), equality of
the numerator and denominator counts separately in the two exponent classes
appearing in the Jack logarithmic derivative.  In the notation of
Proposition~\ref{prop:automatic-branch-jack},
\[
   d\log S_{\chi,K}
   =
   \alpha\,d\log\frac{A_K}{B_K}
   +
   d\log\frac{C_K}{D_K},
\]
and the separated family order gives
\[
   N_A(C_\chi)=N_B(C_\chi),
   \qquad
   N_C(C_\chi)=N_D(C_\chi).
\]
The proof of Proposition \ref{prop:automatic-branch-jack} shows that the one-point
branch condition (\ref{hmbc}) in
Assumption~\ref{ass:contour-branch-admissibility}\textup{(ii)} follows.

For a finite collection of marked \(L\)-type columns, the same separated
zero--pole setup in \cite[Assumptions~5.7--5.8 and Lemma~5.9] {zl23} allows the one-point contours, including repeated marked
locations, to be chosen as pairwise disjoint nested parallel copies inside (\ref{dauv})
Proposition~\ref{prp59} then excludes all
two-variable singularity loci of the \(LL\)-covariance kernel on the product
contours.  Thus the multi-point contour condition in
Assumption~\ref{ass:contour-branch-admissibility}\textup{(iii)} also follows.

The point of keeping Assumption~\ref{ass:contour-branch-admissibility} in its
analytic form is that the asymptotic formulas apply whenever such contours and
branches exist, not only in this particular separated family regime.  The
separated zero--pole assumptions above are only a concrete sufficient criterion.
\end{remark}

\begin{proposition}[One-point Negu\c t asymptotics]\label{p57}
Fix boundary types \(c_l,c_r\in\{el,oa,deel,eoa\}\) and parameters
\(u,v\in(0,1)\).  Suppose Assumption~\ref{ap5} holds.  Let
\(i^{(\epsilon)}\) be a marked column satisfying \eqref{dci}, and assume that
\[
        a_{i^{(\epsilon)}}=L
\]
for all sufficiently small \(\epsilon\).  Write
\[
        \chi=\lim_{\epsilon\to0}\epsilon i^{(\epsilon)}
        \in (V_0,V_m)\setminus\{V_1,\ldots,V_{m-1}\}.
\]
Let
$
        \Pr^{(\epsilon)}
        =
        \Pr^{(\epsilon)}_{c_l,c_r,u,v}
$
be the Macdonald-deformed probability measure \eqref{dpm}.  Assume that the
one-point contour and branch admissibility hypotheses of
Assumption~\ref{ass:contour-branch-admissibility} hold at \(\chi\).

Let \(\mathcal C_\chi\) be an admissible positively oriented contour enclosing
\(0\) and separating the one-point families
\[
        \mathcal D_\chi\subset \operatorname{Int}(\mathcal C_\chi),
        \qquad
        \mathcal E_\chi\cap
        \overline{\operatorname{Int}(\mathcal C_\chi)}
        =
        \varnothing ,
\]
where \(\mathcal D_\chi\) is the required pole set defined in
\eqref{eq:required-poles}, and \(\mathcal E_\chi\) is the excluded point
family defined in \eqref{rp}.

Set
\[
        S_\chi(w)
        :=
        \mathcal G_\chi(w)
        \prod_{k\ge1}\mathcal F_{u,v,k}(w),
\]
where \(\mathcal F_{u,v,k}\) is defined in \eqref{dfuvk}.  Let
\[
        T_\chi(w)
        :=
        \exp\!\left\{
        \beta\,\mathrm{Log}_{\chi,\mathcal C_\chi} S_\chi(w)
        \right\}
\]
be the admissible spectral branch supplied by
Assumption~\ref{ass:contour-branch-admissibility}.  Then, for every
\(g\in\mathbb Z_{>0}\),
\begin{equation}\label{skm}
 \lim_{\epsilon\to0}
 \mathbb E_{\Pr^{(\epsilon)}}
 \bigl[
        \gamma_g(\lambda^{(M,i^{(\epsilon)})};q,t)
 \bigr]
 =
 \frac{1}{2\pi\mathbf i}
 \oint_{\mathcal C_\chi}
        T_\chi(w)^g
        \frac{dw}{w}.
\end{equation}
\end{proposition}

\begin{proof}
Apply Lemma~\ref{lLeftMoments} with $m=1$ and $g_1=g$. Let
$W=(w_1,\dots,w_g)$ and let $\mathcal C_1,\dots,\mathcal C_g$ be nested contours satisfying
the contour conditions in Lemma~\ref{l31}; choose them once and for all inside the region
bounded by~$\mathcal C$.

By the description of the limiting pole sets above, the poles of the prelimit integrand
converge to the limiting pole sets. Hence for all sufficiently small $\epsilon>0$, the same
family of contours remains admissible for the prelimit integrand. Let
$\Theta_k^{(\epsilon)}(W)$ denote the $k$-th correction factor from the one-point
specialization of Lemma~\ref{lLeftMoments}. By Lemma~\ref{l55},
\[
G_{1,L}^{(\epsilon)}(W)\,
G_{0,L}^{(\epsilon)}(W)\,
G_{1,R}^{(\epsilon)}(W)\,
G_{0,R}^{(\epsilon)}(W)
\longrightarrow
\prod_{s=1}^{g}\bigl[\mathcal G_{\chi}(w_s)\bigr]^{\beta}
\]
uniformly on $\mathcal C_1\times\cdots\times\mathcal C_g$. By Lemma~\ref{le36},
\[
\prod_{k\ge1}\Theta_k^{(\epsilon)}(W)
\longrightarrow
\prod_{s=1}^{g}\prod_{k\ge1}\bigl[\mathcal F_{u,v,k}(w_s)\bigr]^{\beta}
\]
uniformly on the same product contour. Therefore dominated convergence applies to the
$g$-fold contour integral.

Finally, Lemma~\ref{lb2} collapses the limiting $g$-fold contour integral to a single contour
integral and yields \eqref{skm}.
\end{proof}

\begin{remark}[Boundary degenerations]
In the empty-boundary degeneration \(u=v=0\), the reflected factors in
\eqref{skm} disappear and the formula reduces to the one-point asymptotic
formula for pure rail-yard dimer coverings in \cite[Proposition~5.7]{LV21}.
If \(v=0\) and \(u\in(0,1)\), the right boundary becomes empty while the left
boundary remains free.  This is the one-free-boundary, or half-space,
degeneration of the doubly free-boundary formula; under the exact Jack scaling
\(q=t^\alpha\), \(t=e^{-n\beta\epsilon}\), it gives the Jack specialization of
the half-space Macdonald-process framework of
Barraquand--Borodin--Corwin~\cite{BBC20}.
\end{remark}

\subsection{Gaussian fluctuations and the annular kernel}

For later use define, for $\chi\in\RR$ and $g\in\ZZ_{>0}$,
\begin{equation}\label{dPhiFluc}
\Phi_{\chi,g}(z)
:=\bigl[\mathcal G_{\chi}(z)\bigr]^{g\beta}
\prod_{k\ge1}\bigl[\mathcal F_{u,v,k}(z)\bigr]^{g\beta}.
\end{equation}
We shall also use the following free-boundary covariance kernel.  For $r\ge1$, set
\begin{align}
\mathfrak a_r
&:=\sum_{k\ge1}\left(u^{4k-4}v^{4k-2}\right)^r
  +\sum_{k\ge1}\left(u^{4k-2}v^{4k}\right)^r
 =\frac{v^{2r}+u^{2r}v^{4r}}{1-u^{4r}v^{4r}},
\label{dabcorra}\\
\mathfrak b_r
&:=\sum_{k\ge1}\left(u^{4k-2}v^{4k-4}\right)^r
  +\sum_{k\ge1}\left(u^{4k}v^{4k-2}\right)^r
 =\frac{u^{2r}+u^{4r}v^{2r}}{1-u^{4r}v^{4r}},
\label{dabcorrb}\\
\mathfrak c_r
&:=\sum_{k\ge1}\left(u^{2k}v^{2k}\right)^r
 =\frac{u^{2r}v^{2r}}{1-u^{2r}v^{2r}}.
\label{dabcorrc}
\end{align}
For $z,w$ on the contours used below, define
\begin{equation}\label{dBuv}
\mathscr B_{u,v}(z,w)
:=\sum_{r\ge1}r\left[
\mathfrak a_r z^{-r}w^{-r}
+\mathfrak b_r z^rw^r
+\mathfrak c_r\left(z^{-r}w^r+z^rw^{-r}\right)
\right].
\end{equation}
The contours are contained in a fixed compact annulus.  We assume $u,v$ are sufficiently
small so that the series in \eqref{dBuv} is absolutely and uniformly convergent on every
product of contours appearing below.

\begin{definition}[Annular free-boundary image covariance for Laplace tests]
\label{def:annular-fb-gff}
Let \(u,v\in(0,1)\) and set
\[
        \mathfrak q:=u^2v^2 .
\]
We use the \(q\)-Pochhammer notation of \eqref{daq}, with \(q\) replaced by
\(\mathfrak q\).  On contour neighborhoods where the products below converge
normally and are non-vanishing, define the annular image potential
\begin{equation}\label{annular-image-potential}
\mathscr L_{u,v}(z,w)
:=
\log
\frac{(w/z;\mathfrak q)_\infty(\mathfrak q z/w;\mathfrak q)_\infty}
     {(u^2zw;\mathfrak q)_\infty(v^2/(zw);\mathfrak q)_\infty},
\end{equation}
with the logarithmic branch prescribed by the contour system.  We use the
term \emph{annular free-boundary image covariance} for the Laplace-test
Gaussian covariance whose potential on the tested contour algebra is
\(\mathscr L_{u,v}\).  
\end{definition}

\begin{proposition}[Annular free-boundary image kernel]
\label{prop:annular-gff-kernel}
Set \(\mathfrak q:=u^2v^2\) and
\[
\mathscr A_{\mathfrak q}(x)
:=
\sum_{r\ge1}\frac{r x^r}{1-\mathfrak q^r}
=
(x\partial_x)^2\!\left[-\log (x;\mathfrak q)_\infty\right].
\]
On every product of contours on which the following series converge normally, define
\begin{equation}\label{annular-gff-kernel}
\mathscr K^{\mathrm{ann}}_{u,v}(z,w)
:=
\mathscr A_{\mathfrak q}\!\left(\frac{w}{z}\right)
+\mathscr A_{\mathfrak q}\!\left(\mathfrak q\frac{z}{w}\right)
+\mathscr A_{\mathfrak q}(u^2zw)
+\mathscr A_{\mathfrak q}\!\left(\frac{v^2}{zw}\right).
\end{equation}
Then $\mathscr B_{u,v}(z,w)$ as defined in (\ref{dBuv}), satisfies
\begin{equation}\label{annular-Buv}
\mathscr B_{u,v}(z,w)
=
\mathscr A_{\mathfrak q}\!\left(\frac{v^2}{zw}\right)
+\mathscr A_{\mathfrak q}(u^2zw)
+\mathscr A_{\mathfrak q}\!\left(\mathfrak q\frac{w}{z}\right)
+\mathscr A_{\mathfrak q}\!\left(\mathfrak q\frac{z}{w}\right).
\end{equation}
Moreover, on nested contours with \(|w|<|z|\),
\begin{equation}\label{annular-full-kernel}
        \frac{zw}{(z-w)^2}
        +
        \mathscr B_{u,v}(z,w)
        =
        \mathscr K^{\mathrm{ann}}_{u,v}(z,w).
\end{equation}
Equivalently,
\begin{equation}\label{annular-green-derivative}
        (z\partial_z)(w\partial_w)\mathscr L_{u,v}(z,w)
        =
        \mathscr K^{\mathrm{ann}}_{u,v}(z,w),
\end{equation}
where \(\mathscr L_{u,v}\) is the potential in
\eqref{annular-image-potential}.

Thus the full bracketed covariance kernel in Theorems~\ref{t58} and
\ref{t77} is the logarithmic-derivative image kernel associated with the
annular free-boundary potential in Definition~\ref{def:annular-fb-gff}.
\end{proposition}

\begin{proof}
The coefficient identities
\[
        \mathfrak a_r=\frac{v^{2r}}{1-\mathfrak q^r},
        \qquad
        \mathfrak b_r=\frac{u^{2r}}{1-\mathfrak q^r},
        \qquad
        \mathfrak c_r=\frac{\mathfrak q^r}{1-\mathfrak q^r}
\]
follow from \eqref{dabcorra}--\eqref{dabcorrc}.  Substituting these identities
into \eqref{dBuv} gives \eqref{annular-Buv}.  If \(|w|<|z|\), then
\[
        \frac{zw}{(z-w)^2}
        =
        \sum_{r\ge1} r\left(\frac wz\right)^r .
\]
Adding this to the
\(\mathscr A_{\mathfrak q}(\mathfrak q w/z)\)-term in
\eqref{annular-Buv} changes the coefficient of \((w/z)^r\) to
\(r/(1-\mathfrak q^r)\).  This gives the
\(\mathscr A_{\mathfrak q}(w/z)\)-term in \eqref{annular-gff-kernel}, while
the other three terms are already the remaining image terms.  Hence
\eqref{annular-full-kernel} follows.

It remains to prove \eqref{annular-green-derivative}.  If
\(x=c z^a w^b\), then
\[
        (z\partial_z)(w\partial_w)\log(x;\mathfrak q)_\infty
        =
        -ab\,\mathscr A_{\mathfrak q}(x),
\]
and the sign is reversed for \(-\log(x;\mathfrak q)_\infty\).  Applying this
identity to the four factors
\[
        x=\frac wz,\qquad
        x=\mathfrak q\frac zw,\qquad
        x=u^2zw,\qquad
        x=\frac{v^2}{zw}
\]
in \eqref{annular-image-potential} gives \eqref{annular-green-derivative}.
Normal convergence justifies the termwise differentiations and summations.
\end{proof}

\begin{theorem}
\label{t58}
Fix boundary types \(c_l,c_r\in\{el,oa,deel,eoa\}\) and parameters
\(u,v\in(0,1)\).  Suppose Assumption~\ref{ap5} holds.  Let
\(s\in\mathbb Z_{>0}\), and let
\[
        i_1^{(\epsilon)},\ldots,i_s^{(\epsilon)}
        \in [l^{(\epsilon)}..r^{(\epsilon)}]
\]
satisfy
\[
        \epsilon i_d^{(\epsilon)}\to \chi_d,
        \qquad
        \chi_1\le\chi_2\le\cdots\le\chi_s,
        \qquad d\in[s],
\]
with each residue class \((i_d^{(\epsilon)})_{\equiv n}\) independent of
\(\epsilon\).  Assume that all marked columns are of \(L\)-type:
\begin{equation}\label{ael}
        a_{i_1^{(\epsilon)}}=
        a_{i_2^{(\epsilon)}}=
        \cdots=
        a_{i_s^{(\epsilon)}}=L
\end{equation}
for all sufficiently small \(\epsilon\).

Assume further that the multi-point contour and branch admissibility
hypotheses of Assumption~\ref{ass:contour-branch-admissibility} hold for the
finite family of marked columns
\[
        \chi_1,\ldots,\chi_s .
\]
Let
$
        \Pr^{(\epsilon)}
        =
        \Pr^{(\epsilon)}_{c_l,c_r,u,v}
$
be the Macdonald-deformed probability measure \eqref{dpm}.  For
\(g_d\in\mathbb Z_{>0}\), define
\[
Q_{g_d}^{(\epsilon)}(\epsilon i_d^{(\epsilon)})
:=
\frac{1}{\epsilon}
\left(
\gamma_{g_d}(\lambda^{(M,i_d^{(\epsilon)})};q,t)
-
\mathbb E_{\Pr^{(\epsilon)}}
\bigl[\gamma_{g_d}(\lambda^{(M,i_d^{(\epsilon)})};q,t)\bigr]
\right),
\qquad d\in[s].
\]

For each \(d\), set
\[
        S_{\chi_d}(z)
        :=
        \mathcal G_{\chi_d}(z)
        \prod_{r\ge1}\mathcal F_{u,v,r}(z),
\]
and let
\[
        T_{\chi_d}(z)
        :=
        \exp\!\left\{
        \beta\,\mathrm{Log}_{\chi_d,\mathcal C_d}S_{\chi_d}(z)
        \right\}
\]
be the admissible spectral branch supplied by
Assumption~\ref{ass:contour-branch-admissibility}.  We write
\[
        \Phi_{\chi_d,g_d}(z):=T_{\chi_d}(z)^{g_d}.
\]

Then, as \(\epsilon\to0\),
\[
\bigl(
Q_{g_1}^{(\epsilon)}(\epsilon i_1^{(\epsilon)}),
\ldots,
Q_{g_s}^{(\epsilon)}(\epsilon i_s^{(\epsilon)})
\bigr)
\]
converges in distribution to a centered Gaussian vector
\[
        \bigl(Q_{g_1}(\chi_1),\ldots,Q_{g_s}(\chi_s)\bigr).
\]
Its covariances are
\begin{align}
&\operatorname{Cov}
\bigl[
Q_{g_d}(\chi_d),
Q_{g_h}(\chi_h)
\bigr]
\notag\\
&\quad=
\frac{n^2\alpha\beta^2 g_dg_h}{(2\pi\mathbf i)^2}
\oint_{\mathcal C_d}\oint_{\mathcal C_h}
\Phi_{\chi_d,g_d}(z)\,
\Phi_{\chi_h,g_h}(w)
\left[
        \frac{zw}{(z-w)^2}
        +
        \mathscr B_{u,v}(z,w)
\right]
\frac{dz}{z}\frac{dw}{w}.
\label{corr-cov-t58}
\end{align}
Here \(\mathcal C_d\) and \(\mathcal C_h\) are admissible pairwise contours
from Assumption~\ref{ass:contour-branch-admissibility}: each
\(\mathcal C_d\) is positively oriented, encloses \(0\) and the required pole
set \(\mathcal D_{\chi_d}\), and separates it from the excluded point family
\(\mathcal E_{\chi_d}\).  The pair of contours is chosen disjoint and
pairwise admissible for the kernel in \eqref{corr-cov-t58}.  If
\(\chi_d=\chi_h\), one may take two disjoint nested copies of the same
one-point contour; the value of the integral is independent of this choice
under admissible contour deformations.
\end{theorem}

\begin{lemma}[Two-point free-boundary correction]\label{le59corr}
Under the assumptions of Theorem~\ref{t58}, let
\[
        Z=(z_1,\ldots,z_{g_d}),
        \qquad
        W=(w_1,\ldots,w_{g_h})
\]
lie on the product contours used in the two-point formula.  Set
\[
        \rho_2^Z:=\rho_{2,L}(Z;q,t),
        \qquad
        \rho_3^Z:=\rho_{3,L}(Z;q,t),
\]
and define \(\rho_2^W,\rho_3^W\) analogously.

Let
\[
        \Theta_k^{Z,W}
        :=
        \Theta_{k,(i_d,i_h)}(Z,W),
        \qquad
        \Theta_k^Z
        :=
        \Theta_{k,(i_d)}(Z),
        \qquad
        \Theta_k^W
        :=
        \Theta_{k,(i_h)}(W)
\]
denote the \(k\)-th reflected factors from Lemma~\ref{lLeftMoments} in the
two-point and one-point specializations.  Define
\begin{equation}\label{eq:Xi-k-as-Theta-quotient}
        \Xi_k^{(\epsilon)}(Z,W)
        :=
        \frac{\Theta_k^{Z,W}}
             {\Theta_k^Z\Theta_k^W},
        \qquad
        \Xi^{(\epsilon)}(Z,W)
        :=
        \prod_{k\ge1}\Xi_k^{(\epsilon)}(Z,W).
\end{equation}
Equivalently, by multiplicativity of \(H\), the graph--marked Cauchy factors
cancel in this quotient, and
\begin{align}
\Xi_k^{(\epsilon)}(Z,W)
&=
\frac{
\Theta_{c_l}\!\left([u^{2k-2}v^{2k-1}]
(\rho_2^Z\cup\rho_2^W);q,t\right)}
{
\Theta_{c_l}\!\left([u^{2k-2}v^{2k-1}]\rho_2^Z;q,t\right)
\Theta_{c_l}\!\left([u^{2k-2}v^{2k-1}]\rho_2^W;q,t\right)
}
\notag\\
&\quad\times
\frac{
\Theta_{c_r}\!\left([u^{2k-1}v^{2k}]
(\rho_2^Z\cup\rho_2^W);q,t\right)}
{
\Theta_{c_r}\!\left([u^{2k-1}v^{2k}]\rho_2^Z;q,t\right)
\Theta_{c_r}\!\left([u^{2k-1}v^{2k}]\rho_2^W;q,t\right)
}
\notag\\
&\quad\times
\frac{
\Theta_{c_l}\!\left([u^{2k-1}v^{2k-2}]
(\rho_3^Z\cup\rho_3^W);q,t\right)}
{
\Theta_{c_l}\!\left([u^{2k-1}v^{2k-2}]\rho_3^Z;q,t\right)
\Theta_{c_l}\!\left([u^{2k-1}v^{2k-2}]\rho_3^W;q,t\right)
}
\notag\\
&\quad\times
\frac{
\Theta_{c_r}\!\left([u^{2k}v^{2k-1}]
(\rho_3^Z\cup\rho_3^W);q,t\right)}
{
\Theta_{c_r}\!\left([u^{2k}v^{2k-1}]\rho_3^Z;q,t\right)
\Theta_{c_r}\!\left([u^{2k}v^{2k-1}]\rho_3^W;q,t\right)
}
\notag\\
&\quad\times
\frac{
H\!\left(
u^{2k}(\rho_2^Z\cup\rho_2^W);
v^{2k}(\rho_3^Z\cup\rho_3^W);q,t
\right)}
{
H\!\left(u^{2k}\rho_2^Z;v^{2k}\rho_3^Z;q,t\right)
H\!\left(u^{2k}\rho_2^W;v^{2k}\rho_3^W;q,t\right)
}.
\label{eq:Xi-k-eps}
\end{align}
Then, uniformly on the product contours,
\begin{equation}\label{xi-correct-expansion}
        \frac{\Xi^{(\epsilon)}(Z,W)-1}{\epsilon^2}
        \longrightarrow
        \mathfrak K_{u,v}(Z,W),
\end{equation}
where
\begin{align}\label{dKuvZW}
\mathfrak K_{u,v}(Z,W)
:=
\alpha n^2\beta^2
\sum_{r\ge1} r\Big[
&\mathfrak a_r\,p_r(Z^{-1})p_r(W^{-1})
+\mathfrak b_r\,p_r(Z)p_r(W)
\notag\\
&+\mathfrak c_r
\bigl(p_r(Z^{-1})p_r(W)+p_r(Z)p_r(W^{-1})\bigr)
\Big].
\end{align}
Here
\[
        p_r(Z)=\sum_{a=1}^{g_d}z_a^r,
        \qquad
        p_r(Z^{-1})=\sum_{a=1}^{g_d}z_a^{-r},
\]
and similarly for \(W\).
\end{lemma}

\begin{proof}
The definition \eqref{eq:Xi-k-as-Theta-quotient} takes the \(k\)-th reflected
factor in the two-point formula of Lemma~\ref{lLeftMoments} and divides it by
the two corresponding one-point reflected factors.  We first verify the
explicit expression \eqref{eq:Xi-k-eps}.

In the two-point formula the active marked specializations are
\[
        \rho_2^Z\cup\rho_2^W,
        \qquad
        \rho_3^Z\cup\rho_3^W .
\]
The boundary \(\Theta\)-terms give the four \(\Theta\)-quotients in
\eqref{eq:Xi-k-eps}.  The marked--marked Cauchy term gives the last quotient
in \eqref{eq:Xi-k-eps}.  The remaining graph--marked Cauchy terms cancel.
For example,
\[
\frac{
H\!\left(u^{2k}\rho_B;
v^{2k}(\rho_3^Z\cup\rho_3^W);q,t\right)}
{
H\!\left(u^{2k}\rho_B;v^{2k}\rho_3^Z;q,t\right)
H\!\left(u^{2k}\rho_B;v^{2k}\rho_3^W;q,t\right)}
=1,
\]
by multiplicativity of \(H\) in the second specialization.  The other three
graph--marked Cauchy factors cancel in the same way.  This proves
\eqref{eq:Xi-k-eps}.

We now compute the second-order limit.  From \eqref{drt}--\eqref{drs}, for
every \(r\ge1\),
\[
p_r(\rho_2^Z)
=
\left(\frac qt\right)^r(q^r-1)p_r(Z^{-1}),
\qquad
p_r(\rho_3^Z)
=
(1-q^{-r})p_r(Z),
\]
and the same formula holds with \(Z\) replaced by \(W\).  Since
\(q=t^\alpha\) and \(t=e^{-n\beta\epsilon}\),
\begin{equation}\label{jack-small-corrections}
\frac{1-t^r}{1-q^r}\to\frac1\alpha,
\qquad
\frac{p_r(\rho_2^Z)}{\epsilon}
\to
-\alpha n\beta r\,p_r(Z^{-1}),
\qquad
\frac{p_r(\rho_3^Z)}{\epsilon}
\to
-\alpha n\beta r\,p_r(Z),
\end{equation}
uniformly on the product contours, and similarly for \(W\).

We now take logarithms in \eqref{eq:Xi-k-eps}.  By the exponential formula for
the boundary factors in Lemma~\ref{le110}, for each boundary type \(c\) there
are coefficients \(A_{c,r}(q,t)\) and \(B_{c,r}(q,t)\) such that, for every
specialization \(\eta\),
\[
        \log \Theta_c(\eta;q,t)
        =
        \sum_{r\ge1}\frac{A_{c,r}(q,t)}{r}p_r(\eta)
        +
        \sum_{r\ge1}\frac{B_{c,r}(q,t)}{r}p_r(\eta)^2 .
\]
Thus, for any scale \(\sigma\) and any two specializations
\(\eta_Z,\eta_W\),
\begin{align*}
&\log
\frac{
\Theta_c(\sigma(\eta_Z\cup\eta_W);q,t)}
{
\Theta_c(\sigma\eta_Z;q,t)\Theta_c(\sigma\eta_W;q,t)}
\\
&\qquad =
\sum_{r\ge1}
\frac{B_{c,r}(q,t)}{r}\,
\sigma^{2r}
\left[
(p_r(\eta_Z)+p_r(\eta_W))^2
-p_r(\eta_Z)^2-p_r(\eta_W)^2
\right]
\\
&\qquad =
2\sum_{r\ge1}
\frac{B_{c,r}(q,t)}{r}\,
\sigma^{2r}p_r(\eta_Z)p_r(\eta_W).
\end{align*}
In particular, all linear terms cancel in the boundary quotients, and only
\(Z\)-\(W\) cross terms remain.

The marked--marked Cauchy quotient is handled similarly.  By multiplicativity
of \(H\) and by \eqref{dh1},
\begin{align*}
&\log
\frac{
H\!\left(
u^{2k}(\rho_2^Z\cup\rho_2^W);
v^{2k}(\rho_3^Z\cup\rho_3^W);q,t
\right)}
{
H\!\left(u^{2k}\rho_2^Z;v^{2k}\rho_3^Z;q,t\right)
H\!\left(u^{2k}\rho_2^W;v^{2k}\rho_3^W;q,t\right)}
\\
&\qquad =
\log H\!\left(u^{2k}\rho_2^Z;v^{2k}\rho_3^W;q,t\right)
+
\log H\!\left(u^{2k}\rho_2^W;v^{2k}\rho_3^Z;q,t\right)
\\
&\qquad =
\sum_{r\ge1}
\frac{1-t^r}{1-q^r}
\frac{u^{2kr}v^{2kr}}{r}
\left[
p_r(\rho_2^Z)p_r(\rho_3^W)
+
p_r(\rho_2^W)p_r(\rho_3^Z)
\right].
\end{align*}
By \eqref{jack-small-corrections} and the analogous estimates with \(Z\)
replaced by \(W\), each \(Z\)-block power sum and each \(W\)-block power sum
appearing above is \(O(\epsilon)\), uniformly on the product contours for each
fixed \(r\).  Therefore every surviving cross term in
\(\log\Xi_k^{(\epsilon)}(Z,W)\) is \(O(\epsilon^2)\).  Using
\eqref{jack-small-corrections}, summing over \(k\ge1\), and applying
\eqref{dabcorra}--\eqref{dabcorrc}, we obtain uniformly on the product contours
\[
        \frac{\log \Xi^{(\epsilon)}(Z,W)}{\epsilon^2}
        \longrightarrow
        \mathfrak K_{u,v}(Z,W),
\]
with \(\mathfrak K_{u,v}\) given by \eqref{dKuvZW}.  The uniformity follows
from the compact-annulus bounds on the contours and the normal convergence of
the reflected products in the admissibility assumptions.

Therefore
\[
        \frac{\Xi^{(\epsilon)}(Z,W)-1}{\epsilon^2}
        =
        \frac{\log \Xi^{(\epsilon)}(Z,W)}{\epsilon^2}
        +o(1)
\]
uniformly on the product contours.  This proves \eqref{xi-correct-expansion}.
\end{proof}

\begin{lemma}\label{le59}
Let $d,h\in[s]$. Under the assumptions of Theorem~\ref{t58},
\begin{align*}
&\lim_{\epsilon\to0}
\mathrm{Cov}\Bigl[
Q_{g_d}^{(\epsilon)}(\epsilon i_d^{(\epsilon)}),
Q_{g_h}^{(\epsilon)}(\epsilon i_h^{(\epsilon)})
\Bigr]\\
&\quad=
\frac{n^2\alpha\beta^2 g_dg_h}{(2\pi\mathbf i)^2}
\oint_{\mathcal C_d}\oint_{\mathcal C_h}
\Phi_{\chi_d,g_d}(z)\Phi_{\chi_h,g_h}(w)
\left[\frac{zw}{(z-w)^2}+\mathscr B_{u,v}(z,w)\right]
\frac{dz}{z}\frac{dw}{w}.
\end{align*}
\end{lemma}

\begin{proof}
By definition,
\[
\operatorname{Cov}\Bigl[
Q_{g_d}^{(\epsilon)}(\epsilon i_d^{(\epsilon)}),
Q_{g_h}^{(\epsilon)}(\epsilon i_h^{(\epsilon)})
\Bigr]
=
\frac{1}{\epsilon^2}
\Bigl(
\mathbb E_{\Pr^{(\epsilon)}}[\Gamma_{d,h}^{(\epsilon)}]
-
\mathbb E_{\Pr^{(\epsilon)}}[\Gamma_d^{(\epsilon)}]\,
\mathbb E_{\Pr^{(\epsilon)}}[\Gamma_h^{(\epsilon)}]
\Bigr),
\]
where
\[
\Gamma_d^{(\epsilon)}
:=
\gamma_{g_d}\bigl(\lambda^{(M,i_d^{(\epsilon)})};q,t\bigr),
\qquad
\Gamma_h^{(\epsilon)}
:=
\gamma_{g_h}\bigl(\lambda^{(M,i_h^{(\epsilon)})};q,t\bigr),
\qquad
\Gamma_{d,h}^{(\epsilon)}
:=
\Gamma_d^{(\epsilon)}\Gamma_h^{(\epsilon)}.
\]

Let
\[
        Z=(z_1,\ldots,z_{g_d}),
        \qquad
        W=(w_1,\ldots,w_{g_h}).
\]
Applying Lemma~\ref{lLeftMoments} with two marked \(L\)-type blocks gives
\[
\mathbb E_{\Pr^{(\epsilon)}}[\Gamma_{d,h}^{(\epsilon)}]
=
\frac{1}{(2\pi\mathbf i)^{g_d+g_h}}
\oint\cdots\oint
\mathcal I_d^{(\epsilon)}(Z)\,
\mathcal I_h^{(\epsilon)}(W)\,
\Lambda^{(\epsilon)}(Z,W),
\]
where \(\mathcal I_d^{(\epsilon)}\) and \(\mathcal I_h^{(\epsilon)}\) are the
one-block integrands and
\[
        \Lambda^{(\epsilon)}(Z,W)
        =
        T_{L,L}(Z,W)\,\Xi^{(\epsilon)}(Z,W).
\]
The product
\[
\mathbb E_{\Pr^{(\epsilon)}}[\Gamma_d^{(\epsilon)}]\,
\mathbb E_{\Pr^{(\epsilon)}}[\Gamma_h^{(\epsilon)}]
\]
is represented by the same product of one-block integrals with
\(\Lambda^{(\epsilon)}(Z,W)\) replaced by \(1\).  Hence the covariance is
obtained by inserting
\[
        \frac{T_{L,L}(Z,W)\Xi^{(\epsilon)}(Z,W)-1}{\epsilon^2}
\]
into the product of the two one-block integrands.

Since the \(Z\)- and \(W\)-contours are disjoint, the variables \(z_a\) and
\(w_b\) stay uniformly separated.  Using
\[
        t=e^{-n\beta\epsilon},
        \qquad
        q=t^\alpha=e^{-\alpha n\beta\epsilon},
\]
we have, uniformly on the product contours,
\begin{align}\label{eq:TLL-single-pair-expansion}
\frac{1}{\epsilon^2}
\left[
\frac{(z-w)(qz-tw)}{(z-tw)(qz-w)}-1
\right]
\longrightarrow
n^2\alpha\beta^2\frac{zw}{(z-w)^2}.
\end{align}
Therefore
\begin{align}\label{eq:TLL-block-expansion}
\frac{T_{L,L}(Z,W)-1}{\epsilon^2}
\longrightarrow
n^2\alpha\beta^2
\sum_{a=1}^{g_d}\sum_{b=1}^{g_h}
\frac{z_aw_b}{(z_a-w_b)^2}
\end{align}
uniformly.  By Lemma~\ref{le59corr},
\[
\frac{\Xi^{(\epsilon)}(Z,W)-1}{\epsilon^2}
\longrightarrow
\mathfrak K_{u,v}(Z,W)
\]
uniformly.  Since \(T_{L,L}(Z,W)\to1\), it follows that
\begin{align}
\frac{T_{L,L}(Z,W)\Xi^{(\epsilon)}(Z,W)-1}{\epsilon^2}
\longrightarrow&
n^2\alpha\beta^2
\sum_{a=1}^{g_d}\sum_{b=1}^{g_h}
\frac{z_aw_b}{(z_a-w_b)^2}
+
\mathfrak K_{u,v}(Z,W).
\label{eq:two-point-insertion-limit}
\end{align}

We now identify the limit of the block integrals.  Set
\[
        T_\chi(z)
        :=
        \exp\!\left\{
        \beta\,\mathrm{Log}_{\chi,\mathcal C_\chi}S_\chi(z)
        \right\},
        \qquad
        \Phi_{\chi,g}(z):=T_\chi(z)^g .
\]
For a single \(L\)-type block \(V=(v_1,\ldots,v_g)\), the Jack limit of the
Negu\c t \(D\)-kernel gives the chain kernel required by Lemma~\ref{lb2}.  More
precisely, the \(L\)-chart \(D\)-kernel has the form
\[
D(V;q,t)
=
\frac{(-1)^{g-1}}{(2\pi\mathbf i)^g}
\frac{\displaystyle\sum_{a=1}^{g}
        \frac{v_g t^{g-a}}{v_a q^{g-a}}}
     {\displaystyle\prod_{a=1}^{g-1}
        \left(1-\frac{t v_{a+1}}{q v_a}\right)}
\prod_{1\le a<b\le g}
\frac{
\left(1-\frac{v_a}{v_b}\right)
\left(1-\frac{qv_a}{tv_b}\right)}
{
\left(1-\frac{v_a}{tv_b}\right)
\left(1-\frac{qv_a}{v_b}\right)}
\prod_{a=1}^g\frac{dv_a}{v_a}.
\]
On the nested admissible contours, the pair product tends uniformly to \(1\).
Moreover,
\[
        \sum_{a=1}^{g}
        \frac{v_g t^{g-a}}{v_a q^{g-a}}
        =
        v_g\sum_{a=1}^{g}\frac1{v_a}+O(\epsilon),
\]
and
\[
        1-\frac{t v_{a+1}}{q v_a}
        =
        1-\frac{v_{a+1}}{v_a}+O(\epsilon).
\]
Hence
\[
(-1)^{g-1}
\frac{v_g}{\prod_{a=1}^{g-1}\left(1-\frac{v_{a+1}}{v_a}\right)}
\prod_{a=1}^g\frac{dv_a}{v_a}
=
\frac{\prod_{a=1}^g dv_a}
     {(v_2-v_1)\cdots(v_g-v_{g-1})}.
\]
Together with the one-block asymptotics from Proposition~\ref{p57}, the
limiting \(g\)-fold block kernel is therefore
\begin{equation}\label{eq:limiting-block-kernel}
\frac{1}{(2\pi\mathbf i)^g}
\frac{\displaystyle\prod_{a=1}^g T_\chi(v_a)}
     {(v_2-v_1)\cdots(v_g-v_{g-1})}
\left(\sum_{a=1}^g\frac1{v_a}\right)
\prod_{a=1}^g dv_a .
\end{equation}

We next record the resulting collapse identity.  Let \(R\) be holomorphic on
the relevant contour neighborhood.  Applying Lemma~\ref{lb2} to
\eqref{eq:limiting-block-kernel}, with
\[
        k=g,
        \qquad
        f(v)=T_\chi(v),
        \qquad
        g_1(v)=\frac1v,
        \qquad
        g_2(v)=R(v),
\]
gives
\begin{align}
&\frac{1}{(2\pi\mathbf i)^g}
\oint\cdots\oint
\frac{\displaystyle\prod_{a=1}^g T_\chi(v_a)}
     {(v_2-v_1)\cdots(v_g-v_{g-1})}
\left(\sum_{a=1}^g\frac1{v_a}\right)
\left(\sum_{a=1}^g R(v_a)\right)
\prod_{a=1}^g dv_a
\notag\\
&\qquad =
\frac{g}{2\pi\mathbf i}
\oint_{\mathcal C_\chi}
T_\chi(v)^g R(v)\frac{dv}{v}.
\label{eq:one-block-collapse}
\end{align}
The admissible branch makes \(T_\chi\) holomorphic on the required contour
neighborhood, so Lemma~\ref{lb2} applies.

Applying \eqref{eq:one-block-collapse} first to the \(Z\)-block and then to
the \(W\)-block yields, for any admissible two-variable kernel \(R(z,w)\),
\begin{align}
&\oint\cdots\oint
\mathcal N_{\chi_d,g_d}(Z)\,
\mathcal N_{\chi_h,g_h}(W)
\sum_{a=1}^{g_d}\sum_{b=1}^{g_h}R(z_a,w_b)
\notag\\
&\qquad =
\frac{g_dg_h}{(2\pi\mathbf i)^2}
\oint_{\mathcal C_d}\oint_{\mathcal C_h}
\Phi_{\chi_d,g_d}(z)\Phi_{\chi_h,g_h}(w)
R(z,w)\frac{dz}{z}\frac{dw}{w},
\label{eq:two-block-collapse}
\end{align}
where \(\mathcal N_{\chi,g}\) denotes the limiting block kernel in
\eqref{eq:limiting-block-kernel}.

We apply \eqref{eq:two-block-collapse} to the two terms in
\eqref{eq:two-point-insertion-limit}.  For the bulk \(T_{L,L}\)-contribution,
take
\[
        R(z,w)=\frac{zw}{(z-w)^2}.
\]
This gives
\[
\frac{n^2\alpha\beta^2 g_dg_h}{(2\pi\mathbf i)^2}
\oint_{\mathcal C_d}\oint_{\mathcal C_h}
\Phi_{\chi_d,g_d}(z)\Phi_{\chi_h,g_h}(w)
\frac{zw}{(z-w)^2}
\frac{dz}{z}\frac{dw}{w}.
\]

For the free-boundary correction, expand \(\mathfrak K_{u,v}(Z,W)\) using
\eqref{dKuvZW}.  Each power-sum product is a sum over pairs:
\[
        p_r(Z^{-1})p_r(W^{-1})
        =
        \sum_{a=1}^{g_d}\sum_{b=1}^{g_h}z_a^{-r}w_b^{-r},
\]
and similarly for
\[
        p_r(Z)p_r(W),\qquad
        p_r(Z^{-1})p_r(W),\qquad
        p_r(Z)p_r(W^{-1}).
\]
Applying \eqref{eq:two-block-collapse} term by term gives
\[
\frac{n^2\alpha\beta^2 g_dg_h}{(2\pi\mathbf i)^2}
\oint_{\mathcal C_d}\oint_{\mathcal C_h}
\Phi_{\chi_d,g_d}(z)\Phi_{\chi_h,g_h}(w)
\mathscr B_{u,v}(z,w)
\frac{dz}{z}\frac{dw}{w},
\]
where \(\mathscr B_{u,v}\) is the series in \eqref{dBuv}.  Combining this
with the bulk contribution proves \eqref{corr-cov-t58}.
\end{proof}

\begin{proposition}[Pairwise expansion and Wick moments]
\label{prop:master-pairwise-wick}
Assume the hypotheses of Theorem~\ref{t58}.  Fix finitely many \(L\)-type
marked columns and positive integers \(g_1,\ldots,g_s\).  Group the contour
variables in the \(L\)-marked moment formula of Lemma~\ref{lLeftMoments} into
blocks
\[
        B_1,\ldots,B_s,
\]
where \(B_a\) is the block attached to the observable at column
\(i_a^{(\epsilon)}\).  Let
\[
        \mathcal I_\epsilon(B_1,\ldots,B_s)
\]
be the full multi-point integrand, and let
\[
        \mathcal I_{a,\epsilon}(B_a)
\]
be the corresponding one-block integrand.  Then, uniformly on the product of
the chosen contours,
\begin{equation}\label{master-pair-factorization}
        \frac{\mathcal I_\epsilon(B_1,\ldots,B_s)}
             {\prod_{a=1}^s\mathcal I_{a,\epsilon}(B_a)}
        =
        \prod_{1\le a<b\le s}
        \left(1+\epsilon^2A_{ab}^{(\epsilon)}(B_a,B_b)+o(\epsilon^2)\right),
\end{equation}
where the functions \(A_{ab}^{(\epsilon)}\) are uniformly bounded and converge
uniformly on the corresponding product contours.  

Consequently, for the centered observables
\[
        Q_{g_a}^{(\epsilon)}(\epsilon i_a^{(\epsilon)}),
        \qquad a=1,\ldots,s,
\]
the limiting centered mixed moments satisfy Wick's formula.  Namely, if \(s\)
is odd, then
\[
        \lim_{\epsilon\to0}
        \mathbb E_{\Pr^{(\epsilon)}}
        \left[
        \prod_{a=1}^s
        Q_{g_a}^{(\epsilon)}(\epsilon i_a^{(\epsilon)})
        \right]
        =0.
\]
If \(s\) is even, then
\begin{equation}\label{eq:wick-moment-limit}
        \lim_{\epsilon\to0}
        \mathbb E_{\Pr^{(\epsilon)}}
        \left[
        \prod_{a=1}^s
        Q_{g_a}^{(\epsilon)}(\epsilon i_a^{(\epsilon)})
        \right]
        =
        \sum_{\pi\in\mathcal P_2([s])}
        \prod_{\{a,b\}\in\pi}
        \operatorname{Cov}\!\left[
        Q_{g_a}(\chi_a),Q_{g_b}(\chi_b)
        \right],
\end{equation}
where \(\mathcal P_2([s])\) is the set of pairings of \([s]\), and the
covariance is the one in \eqref{corr-cov-t58}.
\end{proposition}

\begin{proof}
We first prove the pairwise factorization.  By the \(L\)-marked moment formula
of Lemma~\ref{lLeftMoments}, the full multi-point integrand is a product of
one-block factors and interaction factors between marked blocks.  More
precisely, after the one-block factors
\[
        \mathcal I_{1,\epsilon}(B_1),\ldots,
        \mathcal I_{s,\epsilon}(B_s)
\]
are removed, the remaining factors are of two types.

The factorization follows from the structure of the \(L\)-marked moment
formula in Lemma~\ref{lLeftMoments}.  The finite Negu\c t interaction is
already a product over pairs of marked blocks; since all marked columns are
of type \(L\), the interaction between \(B_a\) and \(B_b\) is
\(T_{L,L}(B_a,B_b)\).

It remains to check the reflected free-boundary part.  For a block \(B_a\),
write
\[
        \rho_{2}^{(a)}:=\rho_{2,L}(B_a;q,t),
        \qquad
        \rho_{3}^{(a)}:=\rho_{3,L}(B_a;q,t).
\]
By the exponential formula for the boundary factors in
Lemma~\ref{le110}, for each boundary type \(c\) there are coefficients
\(L_{c,r}(q,t)\) and \(Q_{c,r}(q,t)\) such that
\[
        \log\Theta_c(\eta;q,t)
        =
        \sum_{r\ge1}\frac{L_{c,r}(q,t)}{r}p_r(\eta)
        +
        \sum_{r\ge1}\frac{Q_{c,r}(q,t)}{r}p_r(\eta)^2 .
\]
Thus, for any scale \(\sigma\),
\begin{align*}
&\log
\frac{
\Theta_c\!\left(\sigma\bigcup_{a=1}^s \rho_\ast^{(a)};q,t\right)}
{\prod_{a=1}^s
\Theta_c\!\left(\sigma\rho_\ast^{(a)};q,t\right)}
\\
&\qquad =
\sum_{r\ge1}\frac{Q_{c,r}(q,t)}{r}\sigma^{2r}
\left[
\left(\sum_{a=1}^s p_r(\rho_\ast^{(a)})\right)^2
-
\sum_{a=1}^s p_r(\rho_\ast^{(a)})^2
\right]
\\
&\qquad =
2\sum_{1\le a<b\le s}
\sum_{r\ge1}
\frac{Q_{c,r}(q,t)}{r}\sigma^{2r}
p_r(\rho_\ast^{(a)})p_r(\rho_\ast^{(b)}),
\end{align*}
where \(\ast\in\{2,3\}\).  Hence the boundary contribution, after division by
the one-block boundary factors, is a product of two-block factors only.

The Cauchy factors have the same pairwise structure by multiplicativity of
\(H\).  The graph--marked factors are one-block factors and cancel after
division by \(\prod_a\mathcal I_{a,\epsilon}(B_a)\).  The marked--marked
Cauchy factors split as products over pairs of marked blocks.  Therefore the
quotient of the full integrand by the product of one-block integrands has no
genuine three-block or higher-block term, and can be written as
\[
        \frac{\mathcal I_\epsilon(B_1,\ldots,B_s)}
             {\prod_{a=1}^s\mathcal I_{a,\epsilon}(B_a)}
        =
        \prod_{1\le a<b\le s}
        \Lambda_{ab}^{(\epsilon)}(B_a,B_b).
\]

For a fixed pair \(a<b\), the pair interaction has the form
\[
        \Lambda_{ab}^{(\epsilon)}(B_a,B_b)
        =
        T_{L,L}(B_a,B_b)\,
        \Xi_{ab}^{(\epsilon)}(B_a,B_b),
\]
where \(\Xi_{ab}^{(\epsilon)}\) is the two-block free-boundary quotient of
Lemma~\ref{le59corr}.  The finite \(T_{L,L}\)-part satisfies the uniform
second-order expansion \eqref{eq:TLL-block-expansion}, and the reflected
free-boundary quotient satisfies
\[
        \frac{\Xi_{ab}^{(\epsilon)}(B_a,B_b)-1}{\epsilon^2}
        \longrightarrow
        \mathfrak K_{u,v}(B_a,B_b)
\]
by \eqref{xi-correct-expansion}.  Since
\[
        T_{L,L}(B_a,B_b)\to1,
        \qquad
        \Xi_{ab}^{(\epsilon)}(B_a,B_b)\to1
\]
uniformly, we get
\begin{equation}\label{eq:pair-factor-second-order}
\frac{\Lambda_{ab}^{(\epsilon)}(B_a,B_b)-1}{\epsilon^2}
\longrightarrow
A_{ab}(B_a,B_b)
\end{equation}
uniformly on the product contours, where
\[
A_{ab}(B_a,B_b)
:=
n^2\alpha\beta^2
\sum_{z\in B_a}\sum_{w\in B_b}
\frac{zw}{(z-w)^2}
+
\mathfrak K_{u,v}(B_a,B_b).
\]
Equivalently,
\[
        \Lambda_{ab}^{(\epsilon)}(B_a,B_b)
        =
        1+\epsilon^2 A_{ab}(B_a,B_b)+o(\epsilon^2),
\]
uniformly on the product contours.

It remains to show that the centered moments have Wick limits.  For
brevity write
\[
        \Gamma_a^{(\epsilon)}
        :=
        \gamma_{g_a}\bigl(\lambda^{(M,i_a^{(\epsilon)})};q,t\bigr),
        \qquad
        Q_a^{(\epsilon)}
        :=
        \frac{\Gamma_a^{(\epsilon)}
        -\mathbb E[\Gamma_a^{(\epsilon)}]}{\epsilon}.
\]
Then
\[
\mathbb E\!\left[\prod_{a=1}^s Q_a^{(\epsilon)}\right]
=
\epsilon^{-s}
\sum_{J\subseteq[s]}
(-1)^{s-|J|}
\mathbb E\!\left[\prod_{a\in J}\Gamma_a^{(\epsilon)}\right]
\prod_{b\notin J}\mathbb E[\Gamma_b^{(\epsilon)}].
\]
We insert the pairwise expansion of the moment formula into each term of this
inclusion--exclusion sum.

Expanding
\[
        \prod_{\{a,b\}\subset J}
        \left(1+\epsilon^2A_{ab}^{(\epsilon)}+o(\epsilon^2)\right)
\]
gives a finite sum over graphs \(G\) whose vertices are contained in \(J\).
Each edge of \(G\) contributes one pair factor and hence one factor
\(\epsilon^2\).  Thus a graph with \(e(G)\) edges contributes order
\[
        \epsilon^{2e(G)}
\]
before the external normalization \(\epsilon^{-s}\).

Now fix an edge set \(G\) on the full vertex set \([s]\), and let
\[
        V(G)\subseteq[s]
\]
be the set of vertices incident to at least one edge of \(G\).  The same edge
set appears in the inclusion--exclusion sum for every \(J\) satisfying
\[
        V(G)\subseteq J\subseteq[s].
\]
Its total coefficient from centering is therefore
\[
        \sum_{J:\,V(G)\subseteq J\subseteq[s]}(-1)^{s-|J|}.
\]
This coefficient is \(0\) unless \(V(G)=[s]\), and it is \(1\) when
\(V(G)=[s]\).  Hence all graphs with an isolated vertex are killed by
centering.  Only graphs with no isolated vertices can contribute.

After multiplying by the normalization \(\epsilon^{-s}\), a graph with
\(e(G)\) edges contributes order
\[
        \epsilon^{2e(G)-s}.
\]
If \(s\) is odd, every graph on \([s]\) with no isolated vertices has at least
\((s+1)/2\) edges.  Hence \(2e(G)-s\ge1\), and every contribution vanishes as
\(\epsilon\to0\).  Therefore the limiting centered \(s\)-th moment is \(0\)
for odd \(s\).

Assume now that \(s\) is even.  A graph on \([s]\) with no isolated vertices
has at least \(s/2\) edges.  The only graphs with exactly \(s/2\) edges and no
isolated vertices are pairings of \([s]\).  Graphs with more than \(s/2\)
edges have \(2e(G)-s>0\) and vanish in the limit.  Thus only pairings survive.

Let \(\pi\in\mathcal P_2([s])\) be a pairing.  The corresponding contribution
factors over its pairs:
\[
        \prod_{\{a,b\}\in\pi}
        \left[
        \lim_{\epsilon\to0}
        \epsilon^{-2}
        \left(
        \mathbb E[\Gamma_a^{(\epsilon)}\Gamma_b^{(\epsilon)}]
        -
        \mathbb E[\Gamma_a^{(\epsilon)}]
        \mathbb E[\Gamma_b^{(\epsilon)}]
        \right)
        \right].
\]
By Lemma~\ref{le59}, the bracketed limit is exactly
\[
        \operatorname{Cov}
        \bigl[
        Q_{g_a}(\chi_a),Q_{g_b}(\chi_b)
        \bigr],
\]
with the covariance given by \eqref{corr-cov-t58}.  Therefore, if \(s\) is
even,
\[
        \lim_{\epsilon\to0}
        \mathbb E\!\left[
        \prod_{a=1}^s Q_a^{(\epsilon)}
        \right]
        =
        \sum_{\pi\in\mathcal P_2([s])}
        \prod_{\{a,b\}\in\pi}
        \operatorname{Cov}
        \bigl[
        Q_{g_a}(\chi_a),Q_{g_b}(\chi_b)
        \bigr].
\]
This is Wick's formula.  The proposition follows.
\end{proof}

\begin{proof}[Proof of Theorem~\ref{t58}]
Lemma~\ref{le59} gives the convergence of all second moments and identifies
the limiting covariance as \eqref{corr-cov-t58}.  Proposition~\ref{prop:master-pairwise-wick}
gives convergence of all centered mixed moments to the Wick polynomials
determined by these covariances.  These are exactly the moments of the
centered Gaussian vector with covariance matrix \eqref{corr-cov-t58}.  Since
Gaussian laws are moment-determinate, the method of moments gives the claimed
finite-dimensional convergence.
\end{proof}

\section{Limit Shape and Frozen Boundary}\label{sect:fb}

In this section, we first prove the Laplace-transform law of large numbers stated in
Theorem~\ref{l61}.  We then convert the transform convergence into a limit-shape
statement by using the natural moment variable
\[
 x=e^{-n\beta\kappa}.
\]
This is the same scale as the exponents supplied by the Negu\c t-operator asymptotics, so no
Laplace extension to exponents outside \(\beta\mathbb Z_{>0}\) is needed.  Finally we identify
the Stieltjes transform of the limiting slope measure and derive the
frozen-boundary equation in this natural scale.

\subsection{Limit shape and the Stieltjes transform in the natural Laplace scale}

\begin{proposition}[Height-Laplace law of large numbers]\label{prop:height-laplace-lln}
Assume the hypotheses of Theorem~\ref{l61}.  For every \(k\in\mathbb Z_{>0}\),
set
\[
        H_\epsilon(\kappa)
        :=
        \epsilon h^{(q,t)}_{M^{(\epsilon)}}\!
        \left(i^{(\epsilon)},\frac{\kappa}{\epsilon}\right).
\]
Then
\begin{equation}\label{height-laplace-lln-prop}
\int_{\mathbb R}e^{-n\beta k\kappa}H_\epsilon(\kappa)\,d\kappa
\xrightarrow[\epsilon\to0]{\mathbb P}
\frac{1}{n^2\alpha k^2\beta^2\pi\mathbf i}
\oint_{\mathcal C_\chi}
T_\chi(w)^k\frac{dw}{w}.
\end{equation}
\end{proposition}

\begin{proof}
After the change of variables \(\kappa=\epsilon y\), the exact Jack
parametrization gives
\[
        e^{-n\beta k\kappa}
        =
        e^{-n\beta k\epsilon y}
        =
        t^{ky}.
\]
Hence
\[
\begin{aligned}
&\mathbb E_{\Pr^{(\epsilon)}}
\int_{\mathbb R}e^{-n\beta k\kappa}H_\epsilon(\kappa)\,d\kappa  \\
&\qquad =
\epsilon^2
\mathbb E_{\Pr^{(\epsilon)}}
\int_{\mathbb R}
h^{(q,t)}_{M^{(\epsilon)}}(i^{(\epsilon)},y)t^{ky}\,dy .
\end{aligned}
\]
By the charge-centered height--Negu\c t identity,
\[
        \int_{\mathbb R}
        h^{(q,t)}_{M}(m,y)t^{ky}\,dy
        =
        \frac{2}{k^2\log t\log q}\,
        \gamma_k(\lambda^{(M,m)};q,t),
\]
so
\[
\mathbb E_{\Pr^{(\epsilon)}}
\int_{\mathbb R}e^{-n\beta k\kappa}H_\epsilon(\kappa)\,d\kappa
=
\frac{2\epsilon^2}{k^2\log t\log q}
\mathbb E_{\Pr^{(\epsilon)}}
\left[
\gamma_k(\lambda^{(M^{(\epsilon)},i^{(\epsilon)})};q,t)
\right].
\]
Since
\[
        \log t=-n\beta\epsilon,
        \qquad
        \log q=\alpha\log t=-\alpha n\beta\epsilon,
\]
we have
\[
        \frac{2\epsilon^2}{k^2\log t\log q}
        =
        \frac{2}{k^2n^2\alpha\beta^2}.
\]
Proposition~\ref{p57}, applied with \(g=k\), gives the convergence of the
expectation to the right-hand side of \eqref{height-laplace-lln-prop}.

It remains to upgrade convergence of expectations to convergence in
probability.  Set
\[
X_\epsilon
:=
\int_{\mathbb R}e^{-n\beta k\kappa}H_\epsilon(\kappa)\,d\kappa .
\]
Again by the charge-centered height--Negu\c t identity,
\[
X_\epsilon-\mathbb E_{\Pr^{(\epsilon)}}[X_\epsilon]
=
\frac{2\epsilon^2}{k^2\log t\log q}
\left(
\gamma_k(\lambda^{(M^{(\epsilon)},i^{(\epsilon)})};q,t)
-
\mathbb E_{\Pr^{(\epsilon)}}
\bigl[
\gamma_k(\lambda^{(M^{(\epsilon)},i^{(\epsilon)})};q,t)
\bigr]
\right).
\]
Let
\[
Q_{k,\epsilon}
:=
\frac{1}{\epsilon}
\left(
\gamma_k(\lambda^{(M^{(\epsilon)},i^{(\epsilon)})};q,t)
-
\mathbb E_{\Pr^{(\epsilon)}}
\bigl[
\gamma_k(\lambda^{(M^{(\epsilon)},i^{(\epsilon)})};q,t)
\bigr]
\right).
\]
Then
\[
X_\epsilon-\mathbb E_{\Pr^{(\epsilon)}}[X_\epsilon]
=
\frac{2\epsilon^3}{k^2\log t\log q}\,Q_{k,\epsilon}.
\]
Therefore
\[
\operatorname{Var}_{\Pr^{(\epsilon)}}(X_\epsilon)
=
\left(
\frac{2\epsilon^3}{k^2\log t\log q}
\right)^2
\operatorname{Var}_{\Pr^{(\epsilon)}}(Q_{k,\epsilon}).
\]
By Lemma~\ref{le59}, the variances
\(\operatorname{Var}_{\Pr^{(\epsilon)}}(Q_{k,\epsilon})\) are bounded, while
\[
        \left(
        \frac{2\epsilon^3}{k^2\log t\log q}
        \right)^2
        =
        O(\epsilon^2).
\]
Thus \(\operatorname{Var}(X_\epsilon)\to0\).  This proves
\eqref{height-laplace-lln-prop}.
\end{proof}

\begin{lemma}[Prelimit positive slope measure and density bound]
\label{lem:positive-slope-density-bound}
Fix the assumptions and notation of Theorem~\ref{l61}, and set
\[
H_\epsilon(\kappa)
=
\epsilon h^{(q,t)}_{M^{(\epsilon)}}
\left(i^{(\epsilon)},\frac{\kappa}{\epsilon}\right).
\]
In the column-wise \((q,t)\)-coordinate system, with the canonical piecewise-affine
interpolation used in Lemma~\ref{le25}, the function \(H_\epsilon\) is locally absolutely
continuous and its slope density
\[
\rho_\epsilon(\kappa):=\frac{d}{d\kappa}H_\epsilon(\kappa)
\]
satisfies
\begin{equation}\label{rho-eps-bound}
0\le \rho_\epsilon(\kappa)\le \frac{2}{\alpha}
\qquad\text{for a.e. }\kappa\in\mathbb R.
\end{equation}
Consequently \(dH_\epsilon=\rho_\epsilon(\kappa)d\kappa\) is a positive finite-on-compact
Borel measure, and the push-forward measure \(\nu_\chi^{(\epsilon)}\) in \eqref{dnu-eps}
is absolutely continuous with density
\begin{equation}\label{prelimit-nu-density}
\frac{d\nu_\chi^{(\epsilon)}}{dx}(x)
=
\rho_\epsilon\!\left(-\frac{1}{n\beta}\log x\right),
\qquad x>0.
\end{equation}
In particular,
\begin{equation}\label{nu-eps-density-bound}
0\le \frac{d\nu_\chi^{(\epsilon)}}{dx}(x)\le \frac{2}{\alpha}
\qquad\text{for a.e. }x>0.
\end{equation}
\end{lemma}

\begin{proof}
Lemma~\ref{le25} gives, in the \((q,t)\)-coordinate \(\tilde y\),
\[
 \frac{d}{d\tilde y}h_M(2m-\tfrac12,\tilde y)
 =\frac{2\log t}{\log q}
 \left(1-
 \sum_{i=1}^{l(\lambda^{(M,m)})}
 \mathbf 1_{[Y^{(i,M,m)}-\frac12,\,Y^{(i,M,m)}+\frac12]}(\tilde y)
 \right)
\]
above the lower reference level, and the derivative is zero below it.  In the exact Jack
specialization \(q=t^\alpha\), we have \(\log t/\log q=1/\alpha\).  The intervals appearing in
this formula are disjoint up to endpoints in the canonical column coordinate, so the factor in
parentheses is either \(0\) or \(1\) a.e.  Hence
\[
0\le \frac{d}{d\tilde y}h_M(2m-\tfrac12,\tilde y)\le \frac{2}{\alpha}
\]
a.e.  Since \(\tilde y=\kappa/\epsilon\) and
\(H_\epsilon(\kappa)=\epsilon h_M(2m-\tfrac12,\kappa/\epsilon)\), differentiation gives
\[
\frac{d}{d\kappa}H_\epsilon(\kappa)
=
\frac{d}{d\tilde y}h_M(2m-\tfrac12,\tilde y)\bigg|_{\tilde y=\kappa/\epsilon},
\]
which proves \eqref{rho-eps-bound}.  The lower-tail normalization and
Lemma~\ref{lem:height-normalization} give local finiteness and the exponential integrability
needed below.  Finally, using \(x=e^{-n\beta\kappa}\), so that
\(dx=-n\beta x\,d\kappa\), the definition \eqref{dnu-eps} becomes
\[
\nu_\chi^{(\epsilon)}(B)
=
\int_B
\rho_\epsilon\!\left(-\frac{1}{n\beta}\log x\right)dx,
\]
which proves \eqref{prelimit-nu-density} and \eqref{nu-eps-density-bound}.
\end{proof}

\begin{lemma}[Moments of the pushed-forward slope measure]
\label{lem:slope-laplace-moments}
For every \(k\in\mathbb Z_{>0}\), define
\[
        M_k^{(\epsilon)}
        :=
        \int_0^\infty x^{k-1}\nu_\chi^{(\epsilon)}(dx).
\]
Then
\begin{equation}\label{eps-moment-by-laplace}
        M_k^{(\epsilon)}
        =
        n^2\beta^2 k
        \int_{\mathbb R}
        e^{-n\beta k\kappa}H_\epsilon(\kappa)\,d\kappa .
\end{equation}
Consequently,
\begin{equation}\label{eps-moment-convergence}
        M_k^{(\epsilon)}
        \xrightarrow[\epsilon\to0]{\mathbb P}
        \frac{1}{\alpha k\pi\mathbf i}
        \oint_{\mathcal C_\chi}
        T_\chi(w)^k\frac{dw}{w}.
\end{equation}
\end{lemma}

\begin{proof}
By the definition of \(\nu_\chi^{(\epsilon)}\) and the change of variables
\(x=e^{-n\beta\kappa}\),
\[
\begin{aligned}
M_k^{(\epsilon)}
&=
\int_0^\infty x^{k-1}\nu_\chi^{(\epsilon)}(dx)  \\
&=
n\beta
\int_{\mathbb R}
e^{-n\beta k\kappa}\,dH_\epsilon(\kappa).
\end{aligned}
\]
The lower-tail normalization and the height-Laplace finiteness from
Lemma~\ref{lem:height-normalization} make the integration-by-parts boundary
term vanish.  Hence
\[
        n\beta
        \int_{\mathbb R}
        e^{-n\beta k\kappa}\,dH_\epsilon(\kappa)
        =
        n^2\beta^2 k
        \int_{\mathbb R}
        e^{-n\beta k\kappa}H_\epsilon(\kappa)\,d\kappa,
\]
which proves \eqref{eps-moment-by-laplace}.

Now apply Proposition~\ref{prop:height-laplace-lln}.  Multiplying
\eqref{height-laplace-lln-prop} by \(n^2\beta^2k\) gives
\[
        M_k^{(\epsilon)}
        \xrightarrow[\epsilon\to0]{\mathbb P}
        n^2\beta^2k\cdot
        \frac{1}{n^2\alpha k^2\beta^2\pi\mathbf i}
        \oint_{\mathcal C_\chi}
        T_\chi(w)^k\frac{dw}{w},
\]
which is exactly \eqref{eps-moment-convergence}.
\end{proof}

\begin{lemma}[Compact moment criterion]\label{lem:compact-moment-criterion}
Let \(\mu_N\) be random finite positive Borel measures on \([0,\infty)\).
Assume that, for every \(j\ge0\),
\[
        M_j^{(N)}
        :=
        \int_0^\infty x^j\,\mu_N(dx)
        \xrightarrow[N\to\infty]{\mathbb P}
        m_j,
\]
and that there are constants \(C_0,C<\infty\) such that
\[
        0\le m_j\le C_0C^j,
        \qquad j\ge0.
\]
Then there exists a unique deterministic finite Borel measure \(\mu\) on
\([0,\infty)\) such that
\[
        \int_0^\infty x^j\,\mu(dx)=m_j,
        \qquad j\ge0,
\]
and
\[
        \mu_N\Rightarrow\mu
\]
in probability.  Moreover,
\[
        \operatorname{supp}\mu\subset[0,C].
\]
\end{lemma}

\begin{proof}
We first prove tightness in probability.  The total masses
\(M_0^{(N)}=\mu_N([0,\infty))\) converge in probability to \(m_0\), hence are
tight.  For the spatial tail, let \(\eta,\delta>0\).  Choose \(R>C\) and
\(j\ge1\) so that
\[
        2C_0\left(\frac{C}{R}\right)^j<\eta .
\]
By Markov's inequality,
\[
        \mu_N([R,\infty))
        \le
        R^{-j}M_j^{(N)} .
\]
Since \(M_j^{(N)}\to m_j\le C_0C^j\) in probability, for all large \(N\),
\[
        \mathbb P\left(M_j^{(N)}>2C_0C^j\right)<\delta .
\]
On the complementary event,
\[
        \mu_N([R,\infty))
        \le
        2C_0\left(\frac{C}{R}\right)^j
        <\eta .
\]
Thus \(\{\mu_N\}\) is tight in probability.

We now identify all subsequential limits.  By the subsequence criterion for
convergence in probability, it suffices to show that every subsequence has a
further subsequence converging weakly in probability to the same deterministic
measure.  Take an arbitrary subsequence.  By tightness and Prokhorov's
theorem, it has a further subsequence, still denoted by \(\mu_N\), which
converges in distribution, as a random finite measure, to some random finite
measure \(\mu_*\).  Passing to a further subsequence and using the Skorokhod
representation theorem and a diagonal extraction, we may assume that
\[
        \mu_N\Rightarrow\mu_*
        \quad\text{a.s.},
        \qquad
        M_j^{(N)}\to m_j
        \quad\text{a.s. for every }j\ge0.
\]

First, \(\mu_*\) is supported on \([0,C]\).  Indeed, by the Portmanteau theorem
for nonnegative lower semicontinuous functions,
\[
        \int_0^\infty x^j\,\mu_*(dx)
        \le
        \liminf_{N\to\infty} M_j^{(N)}
        =
        m_j
        \qquad\text{a.s.}
\]
Hence, for every \(R>C\) and \(j\ge1\),
\[
        \mu_*([R,\infty))
        \le
        R^{-j}\int_0^\infty x^j\,\mu_*(dx)
        \le
        C_0\left(\frac{C}{R}\right)^j .
\]
Letting \(j\to\infty\) gives \(\mu_*([R,\infty))=0\).  Since \(R>C\) was
arbitrary,
\[
        \operatorname{supp}\mu_*\subset[0,C]
        \quad\text{a.s.}
\]

Next we prove that \(\mu_*\) has moments \(m_j\).  Fix \(j\ge0\) and choose
\(R>C\).  Let \(\theta_R\) be a continuous cutoff such that
\[
        0\le\theta_R\le1,\qquad
        \theta_R(x)=1\ \text{for }0\le x\le R,\qquad
        \theta_R(x)=0\ \text{for }x\ge2R .
\]
Since \(x^j\theta_R(x)\) is bounded and continuous, the almost sure weak
convergence gives
\[
        \int_0^\infty x^j\theta_R(x)\,\mu_N(dx)
        \longrightarrow
        \int_0^\infty x^j\theta_R(x)\,\mu_*(dx).
\]
For every \(\ell\ge1\),
\[
\begin{aligned}
0
&\le
M_j^{(N)}
-
\int_0^\infty x^j\theta_R(x)\,\mu_N(dx)  \\
&\le
\int_R^\infty x^j\,\mu_N(dx)
\le
R^{-\ell}M_{j+\ell}^{(N)} .
\end{aligned}
\]
Taking \(N\to\infty\) and using \(M_{j+\ell}^{(N)}\to m_{j+\ell}\le
C_0C^{j+\ell}\), we obtain
\[
\left|
m_j-
\int_0^\infty x^j\theta_R(x)\,\mu_*(dx)
\right|
\le
C_0C^j\left(\frac{C}{R}\right)^\ell .
\]
Letting \(\ell\to\infty\) gives
\[
        \int_0^\infty x^j\theta_R(x)\,\mu_*(dx)=m_j .
\]
Since \(\operatorname{supp}\mu_*\subset[0,C]\subset[0,R]\), the cutoff is
identically \(1\) on the support of \(\mu_*\).  Therefore
\[
        \int_0^\infty x^j\,\mu_*(dx)=m_j,
        \qquad j\ge0 .
\]

It remains to use compact moment determinacy.  If two finite measures on
\([0,C]\) have the same moments, then they have the same integrals against all
polynomials.  By the Weierstrass approximation theorem, polynomials are
uniformly dense in \(C([0,C])\), so the two measures agree on all continuous
test functions and hence are equal.  Thus there is at most one finite measure
with moments \((m_j)_{j\ge0}\) and support in \([0,C]\).

Every subsequential weak limit is therefore equal almost surely to the same
deterministic measure; denote it by \(\mu\).  The subsequence criterion yields
\[
        \mu_N\Rightarrow\mu
\]
in probability.  This completes the proof.
\end{proof}

\begin{proof}[Proof of Theorem~\ref{l61}]
The Laplace-transform convergence in \eqref{lph} is exactly
Proposition~\ref{prop:height-laplace-lln}.  Indeed, for
\(\gamma=k\beta\),
\[
        T_\chi(w)^k
        =
        \left[
        \mathcal G_\chi(w)
        \prod_{r\ge1}\mathcal F_{u,v,r}(w)
        \right]^{k\beta},
\]
where the power is taken with the admissible logarithmic branch fixed by
Assumption~\ref{ass:contour-branch-admissibility}.  Thus
Proposition~\ref{prop:height-laplace-lln} gives precisely \eqref{lph}.

We now prove the weak convergence of the pushed-forward slope measures.
For \(j\ge0\), set
\[
        m_j(\chi)
        :=
        \frac{1}{\alpha(j+1)\pi\mathbf i}
        \oint_{\mathcal C_\chi}
        T_\chi(w)^{j+1}\frac{dw}{w}.
\]
By Lemma~\ref{lem:slope-laplace-moments}, for every \(j\ge0\),
\begin{equation}\label{eq:nu-eps-moment-limit-in-proof}
        \int_0^\infty x^j\,\nu_\chi^{(\epsilon)}(dx)
        \xrightarrow[\epsilon\to0]{\mathbb P}
        m_j(\chi).
\end{equation}
Since \(\mathcal C_\chi\) is fixed and \(T_\chi\) is holomorphic on a
neighborhood of \(\mathcal C_\chi\), there are constants
\(C_0,C_\chi<\infty\) such that
\[
        |m_j(\chi)|\le C_0C_\chi^j,
        \qquad j\ge0.
\]
Moreover,
\[
        m_j(\chi)\ge0,
\]
because \(m_j(\chi)\) is the limit in probability of the nonnegative random
variables on the left-hand side of
\eqref{eq:nu-eps-moment-limit-in-proof}.  Therefore
\[
        0\le m_j(\chi)\le C_0C_\chi^j,
        \qquad j\ge0.
\]

Applying Lemma~\ref{lem:compact-moment-criterion} to
\[
        \mu_N=\nu_\chi^{(\epsilon_N)}
\]
along an arbitrary sequence \(\epsilon_N\downarrow0\), and then using the
subsequence criterion for convergence in probability, we obtain a
deterministic finite Borel measure \(\nu_\chi\) such that
\[
        \nu_\chi^{(\epsilon)}\Rightarrow\nu_\chi
\]
in probability.  The measure \(\nu_\chi\) is uniquely characterized by the
moment formula
\[
        \int_0^\infty x^j\,\nu_\chi(dx)
        =
        m_j(\chi),
        \qquad j\ge0,
\]
which is \eqref{nu-moments}.  In addition, after increasing \(C_\chi\), if
necessary,
\[
        \operatorname{supp}\nu_\chi\subset[0,C_\chi].
\]

It remains to pass the prelimit slope-density bound to the limit.  By
Lemma~\ref{lem:positive-slope-density-bound},
\[
        \nu_\chi^{(\epsilon)}(A)
        \le
        \frac{2}{\alpha}|A|
\]
for every Borel set \(A\subset[0,\infty)\). Here $|A|$ denotes the Lebesgue measure of a Borel set $A$.  Since
\(\nu_\chi^{(\epsilon)}\Rightarrow\nu_\chi\) in probability, every sequence
\(\epsilon_N\downarrow0\) has a subsequence along which the weak convergence
holds almost surely.  Along such a subsequence, the open-set part of the
Portmanteau theorem gives, for every open set \(O\subset[0,\infty)\),
\[
        \nu_\chi(O)
        \le
        \liminf_{\epsilon\to0}\nu_\chi^{(\epsilon)}(O)
        \le
        \frac{2}{\alpha}|O|.
\]
Now let \(A\subset[0,\infty)\) be Borel.  If \(|A|<\infty\), choose open sets
\(O\supset A\) with
\[
        |O|\le |A|+\eta .
\]
Then
\[
        \nu_\chi(A)
        \le
        \nu_\chi(O)
        \le
        \frac{2}{\alpha}|O|
        \le
        \frac{2}{\alpha}(|A|+\eta).
\]
Letting \(\eta\downarrow0\), we get
\[
        \nu_\chi(A)\le \frac{2}{\alpha}|A|.
\]
If \(|A|=\infty\), the same inequality is trivial.  Hence
\[
        \nu_\chi\ll dx.
\]
Writing
\[
        d\nu_\chi(x)=f_\chi(x)\,dx,
\]
we obtain
\[
        0\le f_\chi(x)\le\frac{2}{\alpha}
        \qquad\text{for a.e. }x.
\]

Define the associated macroscopic height profile by
\eqref{def-H-from-nu}:
\[
        \mathcal H(\chi,\kappa)
        :=
        \frac{1}{n\beta}
        \int_{(e^{-n\beta\kappa},\infty)}
        \frac{1}{x}\,\nu_\chi(dx).
\]
Equivalently,
\[
        \mathcal H(\chi,\kappa)
        =
        \frac{1}{n\beta}
        \int_{e^{-n\beta\kappa}}^\infty
        \frac{f_\chi(x)}{x}\,dx .
\]
On every compact interval in the \(\kappa\)-variable, the lower limit
\(e^{-n\beta\kappa}\) stays in a compact subinterval of \((0,\infty)\).
Since \(f_\chi\in L^\infty\), this formula defines a locally absolutely
continuous function of \(\kappa\).

At every point for which
\[
        x=e^{-n\beta\kappa}
\]
is a Lebesgue point of \(f_\chi\), differentiating the last display gives
\begin{equation}\label{slope-from-nu-density}
\begin{aligned}
        \partial_\kappa\mathcal H(\chi,\kappa)
        &=
        -\frac{1}{n\beta}
        \frac{f_\chi(x)}{x}
        \frac{dx}{d\kappa}
        \bigg|_{x=e^{-n\beta\kappa}}  \\
        &=
        f_\chi(e^{-n\beta\kappa}).
\end{aligned}
\end{equation}
Therefore
\[
        0\le
        \partial_\kappa\mathcal H(\chi,\kappa)
        \le
        \frac{2}{\alpha}
        \qquad\text{for a.e. }\kappa .
\]

Finally, the moment formula \eqref{nu-moments} depends on the marked
sequence only through the macroscopic coordinate \(\chi\).  Hence changing
the fixed residue class of the marked \(L\)-type column, while keeping
\(\epsilon i^{(\epsilon)}\to\chi\), gives the same deterministic moment
sequence.  By the uniqueness in
Lemma~\ref{lem:compact-moment-criterion}, the limiting pushed-forward slope
measure is independent of the chosen fixed residue class.

This proves all assertions of Theorem~\ref{l61}.
\end{proof}

Recall that $S_\chi(w)$ and $T_\chi(w)$ are defined in (\ref{dsc}) and (\ref{dtc}), respectively,
with the compatible branch fixed in Assumption~\ref{ass:contour-branch-admissibility}.  By Theorem~\ref{thm:limit-shape-beta},
the Stieltjes transform of $\nu_\chi$ is
\begin{align}
\mathrm{St}_{\nu_\chi}(z)
:=
\int_0^\infty\frac{\nu_\chi(dx)}{z-x},
\qquad z\in\CC\setminus\operatorname{supp}(\nu_\chi).\label{dst}
\end{align}
For $|z|$ sufficiently large, expanding the Cauchy kernel and using \eqref{nu-moments} gives
\begin{align}
\mathrm{St}_{\nu_\chi}(z)
&=
\sum_{k=1}^{\infty}\frac{1}{z^k}
\int_0^\infty x^{k-1}\nu_\chi(dx)\notag\\
&=
\frac{1}{\alpha\pi\mathbf i}
\oint_{\mathcal C}
\sum_{k=1}^{\infty}\frac{T_\chi(w)^k}{kz^k}\frac{dw}{w}\notag\\
&=
-\frac{1}{\alpha\pi\mathbf i}
\oint_{\mathcal C}
\log\left(1-\frac{T_\chi(w)}{z}\right)\frac{dw}{w}.
\label{stc-natural}
\end{align}
The identity then extends by analytic continuation to each connected component of
$\CC\setminus\operatorname{supp}(\nu_\chi)$.

At points where $\nu_\chi$ has a density, the Stieltjes inversion formula and
\eqref{slope-from-nu-density} give
\begin{equation}\label{dsm2-natural}
\partial_\kappa\mathcal H(\chi,\kappa)
=
-\frac{1}{\pi}\lim_{\eta\downarrow0}
\Im\mathrm{St}_{\nu_\chi}(e^{-n\beta\kappa}+\mathbf i\eta).
\end{equation}

\subsection{Truncated characteristic equations and interlacing assumptions}

For $K\in\ZZ_{>0}$, define the truncated pole set
\begin{align}
\label{ddc}\mathcal D_{\chi,K}
&:=
(\mathcal R_{\chi,1,1}\setminus\mathcal R_{\chi,1,2})
\cup
(\mathcal R_{\chi,3,1}\setminus\mathcal R_{\chi,3,2})\\
&\quad\cup
\bigcup_{k=1}^{K}
\Bigl[
(\mathcal R_{5,1,k,k}\setminus\mathcal R_{5,2,k,k})
\cup
(\mathcal R_{7,1,k,k}\setminus\mathcal R_{7,2,k,k})\\
&\hspace{25mm}\cup
(\mathcal R_{5,1,k-1,k}\setminus\mathcal R_{5,2,k-1,k})
\cup
(\mathcal R_{7,1,k-1,k}\setminus\mathcal R_{7,2,k-1,k})
\Bigr].
\end{align}
Set
\[
S_{\chi,K}(w):=\mathcal G_\chi(w)\prod_{k=1}^{K}\mathcal F_{u,v,k}(w),
\qquad
T_{\chi,K}(w):=S_{\chi,K}(w)^\beta.
\]
The finite-$K$ characteristic equation in the natural Stieltjes variable is
\begin{equation}\label{ceqnK}
T_{\chi,K}(w)=x.
\end{equation}
The corresponding infinite equation is
\begin{equation}\label{ceqn}
T_\chi(w)=x,
\qquad
T_\chi(w):=\left[\mathcal G_\chi(w)\prod_{k\ge1}\mathcal F_{u,v,k}(w)\right]^\beta.
\end{equation}
At the point $(\chi,\kappa)$ we take $x=e^{-n\beta\kappa}$.  If $s=e^{-\kappa}$, then
$x=s^{n\beta}$, and on the compatible real branch the equation $T_{\chi,K}(w)=x$ is
equivalent to
\begin{equation}\label{ceqnK-S-equiv}
S_{\chi,K}(w)=s^n.
\end{equation}
This equivalent rational form is the one used for the signed interlacing and root-count
assumptions below.

\begin{assumption}[Finite signed zero--pole order]
\label{ap62}

Assume that
\[
        a_i=L,
        \qquad i\in[l..r].
\]
For every
\[
        \chi\in
        (V_0,V_m)\setminus\{V_1,\ldots,V_{m-1}\}
\]
and every \(K\in\mathbb Z_{>0}\), write
\[
        S_{\chi,K}(w)
        =
        \mathcal G_\chi(w)
        \prod_{k=1}^{K}\mathcal F_{u,v,k}(w)
        =
        \frac{P_K(w)}{Q_K(w)}
\]
in reduced form.

Assume that \(P_K\) and \(Q_K\) have real coefficients, have no
common zero, and have no zero at the origin.  Assume moreover that all
zeros of \(P_K\) and \(Q_K\) are positive and simple, and that
\[
        \deg P_K=\deg Q_K=:M_K .
\]

If \(M_K\geq1\), enumerate the zeros of \(P_K\) and \(Q_K\) as
\[
        0<z_1<\cdots<z_{M_K},
        \qquad
        0<\xi_1<\cdots<\xi_{M_K},
\]
respectively.  We assume that there exists a unique index
\[
        r_0=r_0(\chi,K)\in\{0,1,\ldots,M_K\}
\]
such that one of the following three signed orders holds.

\begin{enumerate}[label=\textup{(\roman*)}]

\item If \(r_0=0\), then
\[
        0
        <
        z_1<\xi_1<z_2<\xi_2<\cdots
        <z_{M_K}<\xi_{M_K}.
\]

\item If \(1\leq r_0\leq M_K-1\), then
\[
\begin{aligned}
0
&<
\xi_1<z_1<\xi_2<z_2<\cdots
<\xi_{r_0}<z_{r_0} \\
&<
z_{r_0+1}<\xi_{r_0+1}
<z_{r_0+2}<\xi_{r_0+2}
<\cdots
<z_{M_K}<\xi_{M_K}.
\end{aligned}
\]

\item If \(r_0=M_K\), then
\[
        0
        <
        \xi_1<z_1<\xi_2<z_2<\cdots
        <\xi_{M_K}<z_{M_K}.
\]

\end{enumerate}

For \(s>0\), set
\[
        F_{K,s}(w)
        :=
        P_K(w)-s^nQ_K(w).
\]
If
\[
        1\leq r_0\leq M_K-1,
\]
then \(F_{K,s}\) has at least \(M_K-2\) distinct zeros in
\((0,\infty)\).  If
\[
        r_0\in\{0,M_K\}
        \qquad\text{and}\qquad
        M_K\geq1,
\]
then \(F_{K,s}\) has at least \(M_K-1\) distinct zeros in
\((0,\infty)\), and consequently has no nonreal zeros.

When \(M_K=0\), set \(r_0=0\) and regard the ordering and root-count
conditions above as void.

In particular, whenever \(M_K\geq2\), the polynomial \(F_{K,s}\)
has at least \(M_K-2\) distinct positive zeros and therefore has at
most one nonreal conjugate pair.
\end{assumption}

We also record a convenient explicit sufficient condition for the zero--pole ordering component of Assumption~\ref{ap62}.  The next two assumptions give an explicit all-$L$ parameter regime in which the basic zero and pole sets, including their scaled free-boundary copies, remain simple and strictly ordered. 

In Assumptions~\ref{ap64}--\ref{ap65}, the constants
\(\tau_1,\ldots,\tau_n\) are the residue weights from
Assumption~\ref{ap5}\textup{(2)}.  Thus, for \(k=i_{\equiv_n}\),
\[
x_i^{(\epsilon)}=
\begin{cases}
e^{-\epsilon(i-i_{\equiv_n})}\tau_k,& b_i^{(\epsilon)}=+,\\
e^{\epsilon(i-i_{\equiv_n})}\tau_k^{-1},& b_i^{(\epsilon)}=-.
\end{cases}
\]

\begin{assumption}\label{ap64}
Assume that $a_i^{(\epsilon)}=L$ for all $i\in[l^{(\epsilon)}..r^{(\epsilon)}]$. Let
$i,j\in\ZZ$ satisfy
\[
\epsilon i\in(V_{p_1-1},V_{p_1}),
\qquad
\epsilon j\in(V_{p_2-1},V_{p_2}),
\qquad
(i)_{\equiv_n}=i_*,
\qquad
(j)_{\equiv_n}=j_*.
\]
Assume the following.
\begin{enumerate}
    \item If $b_i^{(\epsilon)}=b_j^{(\epsilon)}$, $i_*\ne j_*$, and $p_1=p_2$, then
    $\tau_{i_*}\ne\tau_{j_*}$.
    \item If $b_i^{(\epsilon)}=-$, $b_j^{(\epsilon)}=+$, and $p_1\ge p_2$, then
    \[
    \tau_{i_*}^{-1}\tau_{j_*}<e^{V_{p_2-1}-V_{p_1}}.
    \]
    \item If $b_i^{(\epsilon)}=b_j^{(\epsilon)}$ and $\tau_{i_*}>\tau_{j_*}$, then
    \[
    \tau_{i_*}^{-1}\tau_{j_*}<e^{V_{p_2-1}-V_{p_1}}.
    \]
\end{enumerate}
\end{assumption}

Recall from Lemma~\ref{aps5} that
\(\mathbf 1_{\mathcal E_{j,p,\theta,L}}\), \(\theta\in\{0,1\}\), denotes the
eventual nonemptiness indicator of the family
\(I_{j,p,\theta,L}^{(\epsilon)}\).  Thus it is a deterministic \(0\)-\(1\)
constant depending only on the residue \(j\) and the macroscopic block \(p\).

\begin{assumption}\label{ap65}
Assume that $a_i^{(\epsilon)}=L$ for all $i\in[l^{(\epsilon)}..r^{(\epsilon)}]$, and let
$u,v\in(0,1)$. Suppose that
\begin{enumerate}
    \item for any $(j_1,p_1)$ and $(j_2,p_2)$ with
    $\mathbf 1_{\mathcal E_{j_1,p_1,1,L}}=
     \mathbf 1_{\mathcal E_{j_2,p_2,1,L}}=1$,
    \[
    \max\{u,v\}<e^{V_{p_2-1}-V_{p_1}}\tau_{j_2}^{-1}\tau_{j_1};
    \]
    \item for any $(j_1,p_1)$ and $(j_2,p_2)$ with
    $\mathbf 1_{\mathcal E_{j_1,p_1,0,L}}=
     \mathbf 1_{\mathcal E_{j_2,p_2,0,L}}=1$,
    \[
    \max\{u,v\}<e^{V_{p_2-1}-V_{p_1}}\tau_{j_2}^{-1}\tau_{j_1}.
    \]
\end{enumerate}
\end{assumption}

\begin{lemma}[Finite signed root count from the separated \(L\)-lists]
\label{lem:separation-implies-interlacing}

Assume Assumptions~\ref{ap64}--\ref{ap65}.  Fix
\[
        \chi\in
        (V_0,V_m)\setminus\{V_1,\ldots,V_{m-1}\}
\]
and \(K\in\mathbb Z_{>0}\).
Then Assumption~\ref{ap62} holds.

Moreover, if
\[
        0<\xi_1<\cdots<\xi_{M_K}
\]
are the positive poles of the reduced rational function
\(S_{\chi,K}\), and \(r_0=r_0(\chi,K)\) is the index in
Assumption~\ref{ap62}, then
\[
        \mathcal D_{\chi,K}
        =
        \{\xi_1,\ldots,\xi_{r_0}\},
\]
with the convention that the right-hand side is empty when
\(r_0=0\).
\end{lemma}

\begin{proof}
All zero and pole sets below are understood as multisets.  We first work with
the unreduced finite product
\[
        \mathcal G_\chi(w)\prod_{k=1}^{K}\mathcal F_{u,v,k}(w),
\]
and then pass to the reduced form \(P_K/Q_K\).  A pole below means a zero of
the denominator.

By the explicit \(L\)-type product formulae \eqref{dgs1}--\eqref{dgs2},
together with \eqref{dgc} and the reflected product formula \eqref{dfuvk},
every positive numerator zero or pole of the finite truncation is either an
unreflected \(L\)-type zero--pole location or one of its positive
\(u,v\)-scaled reflected copies.  Since all columns are of type \(L\), the
\(R\)-type factors \eqref{dgs3}--\eqref{dgs4} are absent.

We divide the positive zero--pole families of the \(K\)-truncated
product into two blocks.  The \emph{inward}, or
\emph{pre-defect}, block consists of the unreflected
\(\mathcal R_{\chi,1,\bullet}\)-families and the reflected
\(\mathcal R_5\)-families.  In each such family, a pole precedes
the corresponding numerator zero:
\[
        \text{pole}<\text{zero}.
\]
The \emph{outward}, or \emph{post-defect}, block consists of the
unreflected \(\mathcal R_{\chi,2,\bullet}\)-families and the
reflected \(\mathcal R_6\)-families.  In each such family, a
numerator zero precedes the corresponding pole:
\[
        \text{zero}<\text{pole}.
\]
The \(\mathcal R_5\)-locations are obtained from the corresponding
unreflected locations by multiplication by nonnegative powers of
\(u\) and \(v\), and therefore move toward \(0\).  The
\(\mathcal R_6\)-locations involve reciprocal powers of \(u\) or
\(v\), and therefore move toward \(+\infty\). 

Let
\[
        \mathcal Z_K^{\rm raw}
\]
be the multiset of positive numerator zeros in the unreduced product, and let
\[
        \mathcal P_K^{\rm raw}
\]
be the corresponding multiset of positive poles.  Every active one-body
factor and every reflected copy contributes one point to
\(\mathcal Z_K^{\rm raw}\) and one point to \(\mathcal P_K^{\rm raw}\).  Hence
\[
        |\mathcal Z_K^{\rm raw}|=|\mathcal P_K^{\rm raw}|.
\]

We next prove simplicity of these raw lists.  Assumption~\ref{ap64} gives
strict separation among the unreflected \(L\)-type locations.  Assumption
\ref{ap65} gives strict separation after all \(u,v\)-scalings appearing in
\eqref{dfuvk}.  Therefore no two distinct points of
\(\mathcal Z_K^{\rm raw}\) coincide, and no two distinct points of
\(\mathcal P_K^{\rm raw}\) coincide.  The only possible coincidences are
common numerator--denominator points.  Passing to the reduced form
\[
        S_{\chi,K}(w)=\frac{P_K(w)}{Q_K(w)}
\]
removes precisely these common points, with the same multiplicity from the
two raw lists.  Thus the positive zeros of \(P_K\) and \(Q_K\) are simple and
distinct.  Moreover, cancellation removes the same number of points from the
two raw lists, so the two cardinalities remain equal after reduction.  Hence
\[
        \deg P_K=\deg Q_K=:M_K,
        \qquad
        d_K=M_K .
\]

We now identify the signed order.  For \(i=5,6\) and
\(k_1,k_2\ge1\), choose
\[
        y_{i,k_1,k_2}
        \in
        \mathcal R_{i,1,k_1,k_2}\setminus
        \mathcal R_{i,2,k_1,k_2},
\]
and for \(j=1,2\), choose
\begin{equation}\label{dyj-new}
        y_j\in
        \mathcal R_{\chi,j,1}\setminus
        \mathcal R_{\chi,j,2}.
\end{equation}
These are pole-side locations.  Assumptions~\ref{ap64} and~\ref{ap65} give
\begin{equation}\label{eq:y-basic-order}
        0<
        y_{5,k_1,k_2}
        <
        y_1<y_2
        <
        y_{6,\ell_1,\ell_2}
\end{equation}
for all reflection indices which occur in the \(K\)-truncation.  More
precisely, \(y_1<y_2\) is the unreflected order from
Assumption~\ref{ap64}(2), while the two inequalities involving the reflected
families \(5\) and \(6\) are exactly the smallness inequalities in
Assumption~\ref{ap65}.  Moreover, for \(k\ge1\), Assumption~\ref{ap65} gives
the monotone reflected chains
\begin{equation}\label{eq:y-reflected-chains}
        y_{5,k,k+1}<y_{5,k,k}<y_{5,k-1,k},
        \qquad
        y_{6,k,k+1}>y_{6,k,k}>y_{6,k-1,k}.
\end{equation}
Thus the \(\mathcal R_5\)-strings are ordered in the inward block, and the
\(\mathcal R_6\)-strings are ordered in the outward block.

Let
\[
        z_j\in
        \mathcal R_{\chi,j,2}\setminus
        \mathcal R_{\chi,j,1},
        \qquad j=1,2,
\]
be the corresponding unreflected numerator-zero locations.  For instance,
a point \(y_1\) and a point \(z_1\) have the form
\[
        y_1=e^{\max\{V_{p_2-1},\chi\}}\tau_{j_*}^{-1},
        \qquad
        z_1=e^{V_{p_1}}\tau_{i_*}^{-1}
\]
for suitable \(p_1,p_2\in[m]\) and \(i_*,j_*\in[n]\).  Assumption
\ref{ap64}(3) says that whenever
\[
        \tau_{i_*}^{-1}<\tau_{j_*}^{-1},
\]
one has
\[
        z_1<y_1.
\]
Together with Assumption~\ref{ap64}(1)--(2), this gives the signed
interlacing of the unreflected pole-side and zero-side lists inside each
fixed residue class.  The same comparison applies to the second unreflected
family, using the quotient in \eqref{dgs2}.

We now write this order in a form usable for the global list.  Fix a residue
\(j_*\in[n]\).  For each active signed family \(\mathfrak f\) in the inward
block, let
\[
        T^{\rm pole}_{\mathfrak f,j_*}
\]
be the set of remaining pole-side points with residue \(j_*\), and let
\[
        T^{\rm zero}_{\mathfrak f,j_*}
\]
be the corresponding numerator-zero-side set.  After cancellation, these two
sets have the same cardinality: before cancellation they have the same
cardinality because every factor contributes one zero and one pole, and
cancellation removes common zero--pole points from both lists.

Write
\[
        \alpha^{(\mathfrak f,j_*)}_1<\cdots<
        \alpha^{(\mathfrak f,j_*)}_{q(\mathfrak f,j_*)}
\]
for the pole-side list and
\[
        \beta^{(\mathfrak f,j_*)}_1<\cdots<
        \beta^{(\mathfrak f,j_*)}_{q(\mathfrak f,j_*)}
\]
for the numerator-zero-side list.  The inequalities in
Assumption~\ref{ap64}, together with the scaled inequalities in
Assumption~\ref{ap65}, give
\begin{equation}\label{eq:family-signed-order-inward}
        \alpha^{(\mathfrak f,j_*)}_1
        <
        \beta^{(\mathfrak f,j_*)}_1
        <
        \alpha^{(\mathfrak f,j_*)}_2
        <
        \beta^{(\mathfrak f,j_*)}_2
        <
        \cdots
        <
        \alpha^{(\mathfrak f,j_*)}_{q(\mathfrak f,j_*)}
        <
        \beta^{(\mathfrak f,j_*)}_{q(\mathfrak f,j_*)}.
\end{equation}
Thus each active family in the inward block starts with a pole and ends with
a numerator zero.

For each active signed family \(\mathfrak g\) in the outward block, define
the pole-side and numerator-zero-side lists analogously.  The quotient
orientation in the outward reflected factors in \eqref{dfuvk} gives the
opposite endpoint convention:
\begin{equation}\label{eq:family-signed-order-outward}
        \beta^{(\mathfrak g,j_*)}_1
        <
        \alpha^{(\mathfrak g,j_*)}_1
        <
        \beta^{(\mathfrak g,j_*)}_2
        <
        \alpha^{(\mathfrak g,j_*)}_2
        <
        \cdots
        <
        \beta^{(\mathfrak g,j_*)}_{q(\mathfrak g,j_*)}
        <
        \alpha^{(\mathfrak g,j_*)}_{q(\mathfrak g,j_*)}.
\end{equation}
Thus each active family in the outward block starts with a numerator zero
and ends with a pole.

The family intervals themselves are disjoint and ordered from left to right.
Indeed, \eqref{eq:y-basic-order} separates the inward reflected
\(\mathcal R_5\)-strings, the unreflected strings, and the outward reflected
\(\mathcal R_6\)-strings.  The chain relations
\eqref{eq:y-reflected-chains} order the copies inside the two reflected
blocks, while Assumption~\ref{ap64}(3) orders distinct residue classes inside
each unreflected block.  Hence the family-wise signed orders glue without
crossing.

Consequently, after cancellation, the inward and outward families
remain internally interlaced and the inward families precede the
outward families in the global positive order.

Let
\[
        N_K^{\mathrm{in}}
\]
be the number of poles in the inward block, and set
\[
        N_K^{\mathrm{out}}
        :=
        M_K-N_K^{\mathrm{in}}.
\]
Define
\[
        r_0:=N_K^{\mathrm{in}}
        \in\{0,\ldots,M_K\}.
\]

If \(r_0=0\), there is no inward family, and the global order is
\[
        0<
        z_1<\xi_1<z_2<\xi_2<\cdots<z_{M_K}<\xi_{M_K}.
\]
If \(r_0=M_K\), there is no outward family, and the global order is
\[
        0<
        \xi_1<z_1<\xi_2<z_2<\cdots<\xi_{M_K}<z_{M_K}.
\]
Finally, if \(1\le r_0\le M_K-1\), the unique junction between the
two blocks gives
\[
\begin{aligned}
0
&<
\xi_1<z_1<\xi_2<z_2<\cdots
<\xi_{r_0}<z_{r_0}\\
&<
z_{r_0+1}<\xi_{r_0+1}
<z_{r_0+2}<\xi_{r_0+2}
<\cdots<z_{M_K}<\xi_{M_K}.
\end{aligned}
\]
Thus the three ordering alternatives in
Assumption~\ref{ap62} hold.

In the all-\(L\) specialization, the definition \eqref{ddc} shows
that \(\mathcal D_{\chi,K}\) consists of the poles arising from the
unreflected \(\mathcal R_{\chi,1}\)-families and the reflected
\(\mathcal R_5\)-families.  These are exactly the inward poles.
Since the inward block occupies the first \(r_0\) pole--zero pairs
in the global order,
\[
        \mathcal D_{\chi,K}
        =
        \{\xi_1,\ldots,\xi_{r_0}\},
\]
with the empty-set convention when \(r_0=0\).

It remains to verify the root-count statements.  Set
\[
        F_{K,s}(w):=P_K(w)-s^nQ_K(w),
        \qquad s>0.
\]
For every
\[
        i\in\{1,\ldots,M_K-1\},
\]
except for \(i=r_0\) when
\(1\le r_0\le M_K-1\), the interval
\[
        (z_i,z_{i+1})
\]
contains exactly one simple zero of \(Q_K\).  Hence
\[
        Q_K(z_i)Q_K(z_{i+1})<0.
\]
Since
\[
        F_{K,s}(z_i)=-s^nQ_K(z_i),
        \qquad
        F_{K,s}(z_{i+1})=-s^nQ_K(z_{i+1}),
\]
the intermediate value theorem gives a zero of \(F_{K,s}\) in each
such interval.

If \(1\le r_0\le M_K-1\), this gives
\[
        M_K-2
\]
distinct positive roots.  If
\(r_0\in\{0,M_K\}\), it gives
\[
        M_K-1
\]
distinct positive roots.  In the latter case at most one root of
\(F_{K,s}\) remains unaccounted for.  Since \(F_{K,s}\) has real
coefficients, nonreal roots occur in conjugate pairs, and therefore
no nonreal root is possible.

This proves all parts of Assumption~\ref{ap62}.
\end{proof}

\begin{lemma}[Local radial no-escape for truncated spectral roots]
\label{lem:local-radial-noescape}
Assume Assumptions~\ref{ap64}--\ref{ap65}. Fix
\[
\chi_0\in (V_0,V_m)\setminus\{V_1,\ldots,V_{m-1}\},
\qquad
x_0>0.
\]
For \(K\geq1\), write
\[
C_{\chi,K}:=\lim_{w\to\infty}S_{\chi,K}(w),
\]
where the limit is taken in the reduced rational expression. Suppose that
there exists \(K_*\geq1\) such that
\begin{equation}
S_{\chi_0,K}(0)\neq x_0,
\qquad
C_{\chi_0,K}\neq x_0,
\qquad
K\geq K_* .
\label{eq:radial-endpoint-nonexceptional}
\end{equation}
Then there exist open intervals
\[
I\ni\chi_0,
\qquad
X\ni x_0,
\qquad
X\Subset(0,\infty),
\]
constants
\[
0<r<R<\infty,
\]
and \(K_0\geq K_*\), such that for every
\[
\chi\in I,\qquad x\in X,\qquad K\geq K_0,
\]
every solution \(w\in\mathbb C_+\) of
\[
S_{\chi,K}(w)=x
\]
satisfies
\begin{equation}
r\leq |w|\leq R.
\label{eq:radial-uniform-bound}
\end{equation}
Here
\[
\mathbb C_+:=\{w\in\mathbb C:\operatorname{Im}w>0\}.
\]
\end{lemma}

\begin{proof}
Set
\[
\mathfrak q:=(uv)^2\in(0,1).
\]
Choose compact intervals \(I_0\) and \(X_0\) such that, for some
\(p\in[m]\),
\[
\chi_0\in\operatorname{int}(I_0)\Subset(V_{p-1},V_p),
\qquad
x_0\in\operatorname{int}(X_0)\Subset(0,\infty).
\]
We shall shrink \(I_0\) and \(X_0\) finitely many times, and at the end set
\[
I:=\operatorname{int}(I_0),
\qquad
X:=\operatorname{int}(X_0).
\]

By Lemma~\ref{lem:separation-implies-interlacing}, Assumptions~\ref{ap64}--\ref{ap65}
give the reduced one-defect zero--pole order for every finite truncation.
Moreover, the proof of Lemma~\ref{lem:separation-implies-interlacing}
gives the corresponding decomposition into inward and outward reflected
zero--pole families. Since the inequalities in Assumptions~\ref{ap64}--\ref{ap65}
are strict, after shrinking \(I_0\), the active zero--pole labels, the
cancellation pattern in the reduced rational function \(S_{\chi,K}\), and the
relative order of the reduced zeros and poles are independent of
\(\chi\in I_0\). The unreflected zero--pole locations depend continuously on
\(\chi\), while the reflected zero--pole locations are obtained from finitely
many base locations by the fixed dilations
\[
\mathfrak q^k,\qquad v^2\mathfrak q^{k-1},
\qquad
\mathfrak q^{-k},\qquad u^{-2}\mathfrak q^{-(k-1)}.
\]
All constants below may therefore be chosen uniformly for \(\chi\in I_0\).

\medskip
\noindent
\emph{Endpoint separation.}
Evaluation at \(0\) and at infinity is unchanged under positive dilation of
the argument. Hence there exist constants \(a_0,a_\infty>0\), independent of
\(K\), and positive continuous functions \(c_0,c_\infty\) on \(I_0\), such that
\begin{equation}
S_{\chi,K}(0)=c_0(\chi)a_0^K,
\qquad
C_{\chi,K}=c_\infty(\chi)a_\infty^K.
\label{eq:radial-endpoint-geometric-form}
\end{equation}
Condition~\eqref{eq:radial-endpoint-nonexceptional}, continuity, and the
elementary alternatives for a positive geometric sequence imply, after
shrinking \(I_0\) and \(X_0\), that there exist \(\delta>0\) and
\(K_0\geq K_*\) such that
\begin{equation}
\left|
\frac{x}{S_{\chi,K}(0)}-1
\right|
\geq\delta,
\qquad
\left|
\frac{x}{C_{\chi,K}}-1
\right|
\geq\delta
\label{eq:radial-endpoint-uniform-separation}
\end{equation}
for every
\[
\chi\in I_0,\qquad x\in X_0,\qquad K\geq K_0.
\]
Indeed, if \(a_0\neq1\), then \(c_0(\chi)a_0^K\) converges uniformly on
\(I_0\) either to \(0\) or to infinity; if \(a_0=1\), the first inequality
follows from \(c_0(\chi_0)\neq x_0\) and continuity. The argument at
infinity is identical.

\medskip
\noindent
\emph{Interval families and endpoint scales.}
For fixed \(\chi\in I_0\) and \(K\geq K_0\), cancel all common numerator and
denominator factors of \(S_{\chi,K}\). The inward families are those reduced
zero--pole intervals on which a pole precedes a numerator zero, and the
outward families are those on which a numerator zero precedes a pole.
Thus we may write
\[
\mathcal I_{\chi,K}
:=
\bigl\{[p_{j,K},z_{j,K}]\bigr\}_{j},
\qquad
\mathcal O_{\chi,K}
:=
\bigl\{[z_{j,K},p_{j,K}]\bigr\}_{j},
\]
and the reduced factorization has the form
\begin{equation}
S_{\chi,K}(w)
=
C_{\chi,K}
\prod_{[a,b]\in\mathcal I_{\chi,K}}
\frac{w-b}{w-a}
\prod_{[a,b]\in\mathcal O_{\chi,K}}
\frac{w-a}{w-b}.
\label{eq:radial-interval-factorization}
\end{equation}
If both interval families are empty, then
\(S_{\chi,K}\equiv C_{\chi,K}\), and
\eqref{eq:radial-endpoint-uniform-separation} excludes a solution. We
therefore consider the nonconstant case.

The explicit reflected zero--pole locations imply that, after shrinking
\(I_0\) if necessary, there exist finite interval families
\[
\mathcal I_{\chi}^{(0)},
\qquad
\mathcal O_{\chi}^{(0)},
\]
depending continuously on \(\chi\in I_0\), and finite interval families
\[
\mathcal I^{(1)},\quad
\mathcal I^{(2)},
\qquad
\mathcal O^{(1)},\quad
\mathcal O^{(2)},
\]
independent of \(\chi\), such that
\begin{equation}
\mathcal I_{\chi,K}
=
\mathcal I_{\chi}^{(0)}
\sqcup
\bigcup_{k=1}^{K}
\left(
\mathfrak q^k\mathcal I^{(1)}
\sqcup
v^2\mathfrak q^{k-1}\mathcal I^{(2)}
\right),
\label{eq:radial-inward-decomposition}
\end{equation}
and
\begin{equation}
\mathcal O_{\chi,K}
=
\mathcal O_{\chi}^{(0)}
\sqcup
\bigcup_{k=1}^{K}
\left(
\mathfrak q^{-k}\mathcal O^{(1)}
\sqcup
u^{-2}\mathfrak q^{-(k-1)}\mathcal O^{(2)}
\right).
\label{eq:radial-outward-decomposition}
\end{equation}
Here \(a\mathcal A:=\{aL:L\in\mathcal A\}\). Each of the six families above
is finite and consists of nondegenerate real intervals; empty families simply
make no contribution.

Define
\[
\widehat{\mathcal I}^{\rm ref}
:=
\mathcal I^{(1)}
\cup
\frac{v^2}{\mathfrak q}\mathcal I^{(2)},
\qquad
\widehat{\mathcal O}^{\rm ref}
:=
\mathcal O^{(1)}
\cup
u^{-2}\mathfrak q\,\mathcal O^{(2)}.
\]
Then the reflected inward intervals are precisely
\[
\bigcup_{k=1}^{K}\mathfrak q^k\widehat{\mathcal I}^{\rm ref},
\]
and the reflected outward intervals are precisely
\[
\bigcup_{k=1}^{K}\mathfrak q^{-k}\widehat{\mathcal O}^{\rm ref}.
\]

Let \(\mathcal E_{\chi,K}\) be the set of reduced zeros and poles of
\(S_{\chi,K}\), and put
\[
\underline\rho_{\chi,K}
:=
\min_{\lambda\in\mathcal E_{\chi,K}}|\lambda|,
\qquad
\overline\rho_{\chi,K}
:=
\max_{\lambda\in\mathcal E_{\chi,K}}|\lambda|.
\]

If \(\widehat{\mathcal I}^{\rm ref}\neq\varnothing\), then, after increasing
\(K_0\),
\[
\underline\rho_{\chi,K}
=
\mathfrak q^K
\min_{\lambda\in\partial\widehat{\mathcal I}^{\rm ref}}|\lambda|,
\qquad
\chi\in I_0,\quad K\geq K_0.
\]
If \(\widehat{\mathcal I}^{\rm ref}=\varnothing\), then there exists
\(\underline\rho_0>0\) such that
\[
\underline\rho_{\chi,K}\geq \underline\rho_0,
\qquad
\chi\in I_0,\quad K\geq K_0.
\]
Similarly, if \(\widehat{\mathcal O}^{\rm ref}\neq\varnothing\), then, after
increasing \(K_0\),
\[
\overline\rho_{\chi,K}
=
\mathfrak q^{-K}
\max_{\lambda\in\partial\widehat{\mathcal O}^{\rm ref}}|\lambda|,
\qquad
\chi\in I_0,\quad K\geq K_0.
\]
If \(\widehat{\mathcal O}^{\rm ref}=\varnothing\), then there exists
\(\overline\rho_0<\infty\) such that
\[
\overline\rho_{\chi,K}\leq \overline\rho_0,
\qquad
\chi\in I_0,\quad K\geq K_0.
\]

The decompositions
\eqref{eq:radial-inward-decomposition}--\eqref{eq:radial-outward-decomposition}
give constants \(A,B<\infty\), independent of \(\chi\in I_0\) and \(K\), such
that
\begin{equation}
\sum_{L\in\mathcal I_{\chi,K}}|L|\leq A,
\qquad
\sum_{L\in\mathcal O_{\chi,K}}
\int_L\frac{|dt|}{t^2}\leq A,
\label{eq:radial-geometric-tail-bounds}
\end{equation}
and
\begin{equation}
\sum_{\lambda\in\mathcal E_{\chi,K}}
\frac1{|\lambda|}
\leq
\frac{B}{\underline\rho_{\chi,K}},
\qquad
\sum_{\lambda\in\mathcal E_{\chi,K}}
|\lambda|
\leq
B\overline\rho_{\chi,K}.
\label{eq:radial-endpoint-sums}
\end{equation}
Indeed, under a dilation \(L\mapsto aL\),
\[
|aL|=a|L|,
\qquad
\int_{aL}\frac{|dt|}{t^2}
=
a^{-1}\int_L\frac{|dt|}{t^2},
\]
and the reflected tails are geometric.

For all sufficiently small \(\varepsilon>0\),
\eqref{eq:radial-endpoint-sums} gives
\begin{align}
\sup_{\substack{\chi\in I_0\\
|w|\leq\varepsilon\underline\rho_{\chi,K}}}
\left|
\log
\frac{S_{\chi,K}(w)}{S_{\chi,K}(0)}
\right|
&\leq C\varepsilon,
\label{eq:radial-origin-approximation}\\
\sup_{\substack{\chi\in I_0\\
|w|\geq\varepsilon^{-1}\overline\rho_{\chi,K}}}
\left|
\log
\frac{S_{\chi,K}(w)}{C_{\chi,K}}
\right|
&\leq C\varepsilon.
\label{eq:radial-infinity-approximation}
\end{align}
Near \(0\), this follows by summing
\[
\log
\left[
\frac{(w-z)/(w-p)}{z/p}
\right]
=
\log(1-w/z)-\log(1-w/p),
\]
and near infinity it follows by summing
\[
\log\frac{w-z}{w-p}
=
\log(1-z/w)-\log(1-p/w).
\]

Choose \(\varepsilon>0\) so small that
\[
e^{C\varepsilon}-1<\delta.
\]
Then
\eqref{eq:radial-endpoint-uniform-separation}--\eqref{eq:radial-infinity-approximation}
imply
\begin{equation}
S_{\chi,K}(w)\neq x
\quad\text{if}\quad
|w|\leq
\varepsilon\underline\rho_{\chi,K}
\quad\text{or}\quad
|w|\geq
\varepsilon^{-1}\overline\rho_{\chi,K},
\label{eq:radial-extreme-regions}
\end{equation}
for every
\[
\chi\in I_0,\qquad x\in X_0,\qquad K\geq K_0.
\]

\medskip
\noindent
\emph{Phase balance.}
Let \(w\in\mathbb C_+\), and write
\[
\rho:=|w|,
\qquad
y:=\operatorname{Im}w>0.
\]
For a real interval \(L=[a,b]\), \(a<b\), set
\[
\omega_L(w)
:=
\int_L\frac{y}{|w-t|^2}\,dt.
\]
Equivalently,
\[
\omega_L(w)
=
\arg\frac{w-b}{w-a},
\]
with \(\arg\in(-\pi,\pi)\). Define
\[
\Phi_{\chi,K}(w)
:=
\sum_{L\in\mathcal I_{\chi,K}}\omega_L(w),
\qquad
\Psi_{\chi,K}(w)
:=
\sum_{L\in\mathcal O_{\chi,K}}\omega_L(w).
\]
Since the intervals in each family are pairwise disjoint,
\[
0\leq \Phi_{\chi,K}(w)<\pi,
\qquad
0\leq \Psi_{\chi,K}(w)<\pi.
\]
The factorization \eqref{eq:radial-interval-factorization} gives
\[
\arg S_{\chi,K}(w)
\equiv
\Phi_{\chi,K}(w)-\Psi_{\chi,K}(w)
\pmod{2\pi}.
\]
Therefore every upper-half-plane solution of \(S_{\chi,K}(w)=x>0\)
satisfies the exact identity
\begin{equation}
\Phi_{\chi,K}(w)=\Psi_{\chi,K}(w).
\label{eq:radial-phase-balance}
\end{equation}

\medskip
\noindent
\emph{Lower radial bound.}
The outward intervals are uniformly separated from \(0\). Hence, for
\(\rho\) sufficiently small, \eqref{eq:radial-geometric-tail-bounds} gives
\begin{equation}
\Psi_{\chi,K}(w)
\leq
4y
\sum_{L\in\mathcal O_{\chi,K}}
\int_L\frac{|dt|}{t^2}
\leq A_0y,
\label{eq:radial-small-outward-bound}
\end{equation}
where \(A_0\) is independent of \(\chi\) and \(K\).

Suppose first that
\[
\widehat{\mathcal I}^{\rm ref}\neq\varnothing.
\]
Since the reflected inward intervals are the dilates
\[
\mathfrak q^kL,
\qquad
L\in\widehat{\mathcal I}^{\rm ref},
\qquad
1\leq k\leq K,
\]
there exist constants \(\rho_0>0\) and \(c_1,c_2,c_3>0\), independent of
\(\chi\in I_0\) and \(K\), such that whenever
\[
\varepsilon\underline\rho_{\chi,K}\leq \rho\leq \rho_0,
\]
one can find \(L\in\mathcal I_{\chi,K}\) satisfying
\begin{equation}
c_1\rho
\leq
\min_{t\in L}|t|
\leq
\max_{t\in L}|t|
\leq
c_2\rho,
\qquad
|L|\geq c_3\rho.
\label{eq:radial-comparable-inward-interval}
\end{equation}
Indeed, choose the reflected level whose scale is immediately below
\(\rho\); finiteness, nondegeneracy, and separation from \(0\) of
\(\widehat{\mathcal I}^{\rm ref}\) give the constants uniformly.

It follows that
\begin{equation}
\Phi_{\chi,K}(w)
\geq
\int_L\frac{y}{|w-t|^2}\,dt
\geq
\frac{c_3}{(1+c_2)^2}\frac{y}{\rho}
=:b_0\frac{y}{\rho}.
\label{eq:radial-small-inward-lower-bound}
\end{equation}
For a root, combining
\eqref{eq:radial-phase-balance},
\eqref{eq:radial-small-outward-bound}, and
\eqref{eq:radial-small-inward-lower-bound} yields
\[
b_0\frac{y}{\rho}\leq A_0y.
\]
Since \(y>0\),
\[
\rho\geq \frac{b_0}{A_0}.
\]
Together with \eqref{eq:radial-extreme-regions}, this excludes a fixed
neighborhood of \(0\).

If instead
\[
\widehat{\mathcal I}^{\rm ref}=\varnothing,
\]
then the lower bound on \(\underline\rho_{\chi,K}\) and
\eqref{eq:radial-extreme-regions} directly exclude a fixed neighborhood of
\(0\). Thus there exists \(r_*>0\) such that every upper-half-plane solution
satisfies
\begin{equation}
|w|\geq r_*.
\label{eq:radial-root-lower-bound}
\end{equation}

\medskip
\noindent
\emph{Upper radial bound.}
The inward intervals are uniformly bounded, and their total length is
uniformly bounded by \eqref{eq:radial-geometric-tail-bounds}. Hence, for
\(\rho\) sufficiently large,
\begin{equation}
\Phi_{\chi,K}(w)
\leq
A_\infty\frac{y}{\rho^2},
\label{eq:radial-large-inward-bound}
\end{equation}
where \(A_\infty\) is independent of \(\chi\) and \(K\).

Suppose next that
\[
\widehat{\mathcal O}^{\rm ref}\neq\varnothing.
\]
Since the reflected outward intervals are the dilates
\[
\mathfrak q^{-k}L,
\qquad
L\in\widehat{\mathcal O}^{\rm ref},
\qquad
1\leq k\leq K,
\]
there exist constants \(R_0>0\) and \(c_1',c_2',c_3'>0\), independent of
\(\chi\in I_0\) and \(K\), such that whenever
\[
R_0\leq \rho\leq \varepsilon^{-1}\overline\rho_{\chi,K},
\]
one can find \(L\in\mathcal O_{\chi,K}\) satisfying
\[
c_1'\rho
\leq
\min_{t\in L}|t|
\leq
\max_{t\in L}|t|
\leq
c_2'\rho,
\qquad
|L|\geq c_3'\rho.
\]
Choosing the reflected level whose scale is immediately above \(\rho\) gives
the constants uniformly. Consequently,
\begin{equation}
\Psi_{\chi,K}(w)
\geq
\frac{c_3'}{(1+c_2')^2}\frac{y}{\rho}
=:b_\infty\frac{y}{\rho}.
\label{eq:radial-large-outward-lower-bound}
\end{equation}
For a root,
\eqref{eq:radial-phase-balance},
\eqref{eq:radial-large-inward-bound}, and
\eqref{eq:radial-large-outward-lower-bound} imply
\[
b_\infty\frac{y}{\rho}
\leq
A_\infty\frac{y}{\rho^2}.
\]
Cancelling \(y>0\), we get
\[
\rho\leq \frac{A_\infty}{b_\infty}.
\]
Together with \eqref{eq:radial-extreme-regions}, this excludes the
complement of a fixed disk.

If instead
\[
\widehat{\mathcal O}^{\rm ref}=\varnothing,
\]
then the upper bound on \(\overline\rho_{\chi,K}\) and
\eqref{eq:radial-extreme-regions} directly exclude the complement of a fixed
disk. Thus there exists \(R_*<\infty\) such that every upper-half-plane
solution satisfies
\begin{equation}
|w|\leq R_*.
\label{eq:radial-root-upper-bound}
\end{equation}

Taking
\[
r:=r_*,
\qquad
R:=R_*,
\qquad
I:=\operatorname{int}(I_0),
\qquad
X:=\operatorname{int}(X_0),
\]
proves \eqref{eq:radial-uniform-bound}.
\end{proof}

\begin{lemma}[Local pole-free no-escape for truncated and limiting roots]
\label{lem:local-polefree-noescape}
Assume Assumptions~\ref{ap64}--\ref{ap65}. Fix
\[
\chi_0\in (V_0,V_m)\setminus\{V_1,\ldots,V_{m-1}\},
\qquad
x_0>0.
\]
For \(K\geq1\), write
\[
C_{\chi,K}:=\lim_{w\to\infty}S_{\chi,K}(w),
\]
whenever this limit is taken in the reduced rational expression. Suppose that
\(x_0\) is not an endpoint-exceptional level for the finite truncations, namely
that there exists \(K_*\geq1\) such that
\begin{equation}
S_{\chi_0,K}(0)\neq x_0,
\qquad
C_{\chi_0,K}\neq x_0,
\qquad K\geq K_* .
\label{eq:local-noescape-endpoint-nonexceptional}
\end{equation}
Then there exist open intervals
\[
I\ni \chi_0,\qquad X\ni x_0,\qquad X\Subset(0,\infty),
\]
a number \(K_0\geq K_*\), an open set \(U\subset\mathbb C^*=\mathbb{C}\setminus\{0\}\), and a
compact set
\[
\Omega\Subset U
\]
such that the following hold.

\begin{enumerate}
\item For every \(\chi\in I\) and every \(K\geq K_0\), the reduced rational
function \(S_{\chi,K}\) has no pole in \(U\).

\item For every \(\chi\in I\), the infinite product
\[
S_\chi(w)=G_\chi(w)\prod_{k\geq1}F_{u,v,k}(w)
\]
is meromorphic on \(U\), has no pole in \(U\), and
\[
S_{\chi,K}\longrightarrow S_\chi
\]
locally uniformly on \(U\).

\item For every
\[
\chi\in I,\qquad x\in X,\qquad K\geq K_0,
\]
every nonreal solution of
\[
S_{\chi,K}(w)=x
\]
lies in \(\Omega\).

\item For every
\[
\chi\in I,\qquad x\in X,
\]
every nonreal solution of the limiting equation
\[
S_\chi(w)=x
\]
also lies in \(\Omega\).
\end{enumerate}

Consequently, on any compact \(s\)-interval \(J\Subset(0,\infty)\) for which
\(J^n\) is covered by finitely many neighborhoods \(X\) satisfying
\eqref{eq:local-noescape-endpoint-nonexceptional}, the finite truncations and
the infinite product satisfy the pole-free localization/no-escape conclusion
needed in the Hurwitz and Rouché arguments below.
\end{lemma}

\begin{proof}
It is enough to work in the upper half-plane, since the reduced functions
have real coefficients and nonreal roots occur in conjugate pairs.
By Lemma~\ref{lem:local-radial-noescape}, after shrinking \(I\) and \(X\),
and after increasing \(K_0\), there exist constants \(0<r<R<\infty\) such
that every upper-half-plane solution of
\[
S_{\chi,K}(w)=x,\qquad \chi\in I,\quad x\in X,\quad K\geq K_0,
\]
satisfies
\[
r\leq |w|\leq R.
\]
Set
\[
A:=\{w\in\mathbb C:r/2\leq |w|\leq 2R\}.
\]

We next separate the roots from poles in this fixed annulus. By the reflected
zero--pole decomposition used in the proof of Lemma~\ref{lem:separation-implies-interlacing},
all positive poles of \(S_{\chi,K}\) are obtained from finitely many unreflected
pole locations by the dilations
\[
q^k,\qquad v^2q^{k-1},\qquad q^{-k},\qquad u^{-2}q^{-(k-1)},
\qquad q=(uv)^2\in(0,1).
\]
Since \(A\) is a fixed compact annulus, only finitely many reflected levels can
meet \(A\). Increasing \(K_0\), if necessary, we may assume that 
\begin{align*}
\operatorname{Pole}(S_{\chi,K})\cap A
   =
   \operatorname{Pole}(S_{\chi,K_0})\cap A.
\end{align*}

Thus there is a finite collection of real pole branches
\[
\mathscr P=\{\pi_\nu(\chi):\nu\in\mathcal N\},
\qquad \chi\in I,
\]
such that every pole of every reduced \(S_{\chi,K}\) lying in \(A\) belongs to
\(\{\pi_\nu(\chi):\nu\in\mathcal N\}\).After shrinking \(I\), every surviving pole in \(A\) is simple.  Moreover, there is
\(\delta_{\rm sep}>0\) such that, for all \(\chi\in I\) and all \(K\ge K_0\),
any two distinct reduced zero or pole locations of \(S_{\chi,K}\) lying in \(A\)
are at distance at least \(\delta_{\rm sep}\).

We claim that there exist \(\eta_0>0\) and \(M<\infty\) such that
\begin{equation}
|S_{\chi,K}(w)|>M
\label{eq:pole-neighborhood-large}
\end{equation}
whenever
\[
\chi\in I,\qquad K\geq K_0,\qquad
0<\operatorname{dist}\bigl(w,\{\pi_\nu(\chi):\nu\in\mathcal N\}\bigr)<\eta_0,
\qquad w\in A .
\]
Indeed, near a reduced simple pole \(\pi_\nu(\chi)\), the function has the
form
\[
S_{\chi,K}(w)=\frac{H_{\nu,\chi,K}(w)}{w-\pi_\nu(\chi)} ,
\]
where \(H_{\nu,\chi,K}\) is holomorphic and nonvanishing. The factors whose
zeros and poles lie in \(A\) form a finite, uniformly separated collection.
The factors outside \(A\) contribute uniformly bounded and uniformly
nonzero products, because the reflected tails are geometric. Hence there are
constants \(0<c<C<\infty\), independent of \(\chi\), \(K\), and \(\nu\), such
that
\[
c\leq |H_{\nu,\chi,K}(w)|\leq C
\]
on a fixed neighborhood of the pole branch. Taking \(\eta_0\) small gives
\eqref{eq:pole-neighborhood-large}. Choosing
\[
M>\sup X
\]
then implies that
\[
S_{\chi,K}(w)\neq x,\qquad x\in X,
\]
throughout these pole neighborhoods.

Shrink \(I\) once more so that every pole branch has image diameter less
than \(\eta_0/4\) on \(I\). Define
\[
\Gamma_\nu:=\{\pi_\nu(\chi):\chi\in I\},
\qquad
\Gamma:=\bigcup_{\nu\in\mathcal N}\Gamma_\nu,
\]
and put
\[
\eta:=\eta_0/4.
\]
Then a point within distance \(\eta\) of \(\Gamma\) is within distance
\(\eta_0/2\) of the actual pole set for the corresponding value of \(\chi\).
Therefore no solution of \(S_{\chi,K}(w)=x\), with
\[
\chi\in I,\qquad x\in X,\qquad K\geq K_0,
\]
can lie in
\[
\{w\in A:\operatorname{dist}(w,\Gamma)<\eta\}.
\]

Now define
\[
U:=
\left\{
w\in\mathbb C^*:
\frac r2<|w|<2R,\ 
\operatorname{dist}(w,\Gamma)>\frac{\eta}{2}
\right\}
\]
and
\[
\Omega:=
\left\{
w\in\mathbb C:
r\leq |w|\leq R,\ 
\operatorname{dist}(w,\Gamma)\geq \eta
\right\}.
\]
Then \(\Omega\Subset U\). By construction, \(U\) contains no pole of
\(S_{\chi,K}\) for any \(\chi\in I\) and \(K\geq K_0\). Combining the radial
bound \eqref{eq:radial-uniform-bound} with the pole-neighborhood
exclusion proves that every nonreal solution of the truncated equation lies
in \(\Omega\).

It remains to pass to the infinite product. On \(U\), only finitely many
reflected pole branches can meet the annulus \(A\), and all of them have
been removed. The remaining reflected factors converge normally on compact
subsets of \(U\), by the same geometric estimates used in Lemma~\ref{le36}.
Hence
\[
S_{\chi,K}\to S_\chi
\]
locally uniformly on \(U\), and \(S_\chi\) has no pole in \(U\).

Suppose, for contradiction, that for some \(\chi\in I\) and \(x\in X\) the
limiting equation
\[
S_\chi(w)=x
\]
has a nonreal solution \(w_0\notin\Omega\). If \(w_0\) lies outside the annulus
\(\{r\leq |w|\leq R\}\), choose a small disk around \(w_0\) avoiding poles and
still outside that annulus. By local uniform convergence and Hurwitz's
theorem, \(S_{\chi,K}(w)-x\) has a zero in this disk for all large \(K\), which
contradicts the radial bound for the truncated roots.

If instead \(w_0\) lies in the annulus but
\[
\operatorname{dist}(w_0,\Gamma)<\eta,
\]
then \(w_0\) belongs to a pole neighborhood on which
\(|S_{\chi,K}|\) is uniformly larger than \(\sup X\). Passing to the locally
uniform limit gives \(|S_\chi(w_0)|\geq \sup X\), contradicting
\(S_\chi(w_0)=x\in X\). Therefore every nonreal solution of the limiting
equation lies in \(\Omega\).

Finally, the compact \(s\)-interval statement follows by applying the local
result to finitely many points of \(J^n\) and taking the union of the resulting
compact sets and pole-free neighborhoods. This proves the lemma.
\end{proof}

\begin{lemma}[Finite signed root count for the truncated equation]
\label{lem:signed-root-count}
\label{l65}
Assume
Assumption~\ref{ap62} holds.  Then, for every
\(K\in\mathbb Z_{>0}\) and every \(x>0\), the equation
\[
        S_{\chi,K}(w)=x
\]
has at most one nonreal conjugate pair of roots.
\end{lemma}

\begin{proof}
Write
\[
        S_{\chi,K}(w)=\frac{P_K(w)}{Q_K(w)}
\]
in reduced form and set \(d_K=\max\{\deg P_K,\deg Q_K\}\).  The equation is
therefore
\[
        F_{K,x}(w):=P_K(w)-xQ_K(w)=0 .
\]
By Assumption~\ref{ap62}, \(F_{K,x}\) has at least
\(d_K-2\) distinct zeros in \((0,\infty)\).  Since
\(\deg F_{K,x}\le d_K\), at most two zeros, counted with multiplicity, can
lie outside \((0,\infty)\).  As \(F_{K,x}\in\mathbb R[w]\), its nonreal zeros
come in conjugate pairs.  Hence there is at most one nonreal conjugate pair.
\end{proof}

\begin{lemma}[Origin sign from the one-defect zero--pole order]
\label{l66}
Let \(K\in\mathbb Z_{>0}\), set
\[
        x=e^{-n\beta\kappa},
        \qquad
        s=e^{-\kappa},
\]
so that \(x=s^{n\beta}\). Suppose Assumption \ref{ap62} holds.
If the truncated equation
\[
        S_{\chi,K}(w)=e^{-n\kappa}
\]
has exactly one nonreal conjugate pair, counted with multiplicity two, then
\begin{equation}\label{gnk}
        T_{\chi,K}(0)>e^{-n\beta\kappa}.
\end{equation}
Moreover, all remaining real roots are positive.
\end{lemma}

\begin{proof}
Set
\[
y:=s^n=e^{-n\kappa}.
\]
Since the equation has a nonreal conjugate pair, the endpoint cases
\(r_0=0\) and \(r_0=M_K\) in Assumption~\ref{ap62} are impossible.
Hence
\[
        1\leq r_0\leq M_K-1.
\]
The equation
\[
S_{\chi,K}(w)=s^n
\]
is therefore equivalent to
\[
F_{K,y}(w):=P_K(w)-yQ_K(w)=0.
\]

We first count the positive roots of \(F_{K,y}\). For every
\[
i\in\{1,\ldots,M_K-1\}\setminus\{r_0\},
\]
the interval \((z_i,z_{i+1})\) contains exactly one simple zero of
\(Q_K\). Hence \(Q_K(z_i)\) and \(Q_K(z_{i+1})\) have opposite signs.
Since
\[
F_{K,y}(z_i)=-yQ_K(z_i),
\qquad
F_{K,y}(z_{i+1})=-yQ_K(z_{i+1}),
\]
the values \(F_{K,y}(z_i)\) and \(F_{K,y}(z_{i+1})\) also have
opposite signs. Thus \(F_{K,y}\) has one zero in each such interval.
These intervals are disjoint, so we obtain \(M_K-2\) distinct positive
real roots.

Since
\[
\deg F_{K,y}\leq M_K,
\]
and since by assumption the equation has one nonreal conjugate pair
counted with multiplicity two, these \(M_K-2\) positive roots together
with the two nonreal roots exhaust all roots of \(F_{K,y}\).
Consequently, \(F_{K,y}\) has no other real roots. In particular, all
of its real roots are positive, and it has no zero in \((0,z_1)\).

It remains to determine the sign at the origin. At \(z_1\),
\[
F_{K,y}(z_1)=-yQ_K(z_1).
\]
Because \(Q_K\) has the simple zero
\[
\xi_1\in(0,z_1),
\]
the signs of \(Q_K(0)\) and \(Q_K(z_1)\) are opposite. Since
\(F_{K,y}\) has no zero in \((0,z_1)\), the signs of
\(F_{K,y}(0)\) and \(F_{K,y}(z_1)\) are the same. Consequently,
\[
\frac{F_{K,y}(0)}{Q_K(0)}>0.
\]
But
\[
\frac{F_{K,y}(0)}{Q_K(0)}
=
\frac{P_K(0)-yQ_K(0)}{Q_K(0)}
=
S_{\chi,K}(0)-y.
\]
Hence
\[
S_{\chi,K}(0)>y=s^n.
\]
Since \(T_{\chi,K}=S_{\chi,K}^{\beta}\) on the chosen positive real
branch and \(\beta>0\), it follows that
\[
T_{\chi,K}(0)
>
y^\beta
=
s^{n\beta}
=
x
=
e^{-n\beta\kappa}.
\]
\end{proof}

\subsection{Slope formula and frozen boundary}

\begin{lemma}[Root-labelled Stieltjes representation]
\label{lem:root-labelled-stieltjes-jack}
Assume Assumptions~\ref{ap5}, \ref{ass:contour-branch-admissibility}, and
\ref{ap62}.  Assume also that \(\beta=1\).  Let \(C_\chi\) be the one-point
contour fixed in Assumption~\ref{ass:contour-branch-admissibility}, and let
\(\mathcal{D}_{\chi,K}\) be the truncated pole set defined in (\ref{ddc}).  For
\(|\zeta|\) sufficiently large and each \(\xi\in \mathcal{D}_{\chi,K}\), let
\(w_{\xi,\chi,K}(\zeta)\) be the unique root of
\[
   S_{\chi,K}(w)=\zeta
\]
near \(\xi\), labelled by
\[
   w_{\xi,\chi,K}(\zeta)\longrightarrow \xi,
   \qquad |\zeta|\to\infty .
\]
Choose logarithms so that
\[
   \log\left(1-\frac{S_{\chi,K}(0)}{\zeta}\right)\longrightarrow0,
   \qquad
   \log w_{\xi,\chi,K}(\zeta)-\log\xi\longrightarrow0,
   \qquad |\zeta|\to\infty .
\]
Then
\begin{align}
-\frac{1}{\alpha\pi i}
\oint_{C_\chi}
\log\left(1-\frac{S_{\chi,K}(w)}{\zeta}\right)\frac{dw}{w}
&=
-\frac{2}{\alpha}
\log\left(1-\frac{S_{\chi,K}(0)}{\zeta}\right)
\notag\\
&\quad
+\frac{2}{\alpha}
\sum_{\xi\in \mathcal{D}_{\chi,K}}
\left(
   \log w_{\xi,\chi,K}(\zeta)-\log\xi
\right).
\label{eq:finite-root-labelled-stieltjes-jack}
\end{align}
Consequently, with \(\operatorname{St}_{\nu_\chi}\) defined by
\eqref{dst}
\begin{align}
\operatorname{St}_{\nu_\chi}(\zeta)
&=
-\frac{2}{\alpha}
\log\left(1-\frac{S_\chi(0)}{\zeta}\right)
\notag\\
&\quad
+\frac{2}{\alpha}
\lim_{K\to\infty}
\sum_{\xi\in \mathcal{D}_{\chi,K}}
\left(
   \log w_{\xi,\chi,K}(\zeta)-\log\xi
\right),
\label{eq:st-root-formula}
\end{align}
first for \(|\zeta|\) sufficiently large and then by simultaneous analytic
continuation of the root labels and logarithms.

In particular, let \(x>0\) be a nonexceptional level at which the
Stieltjes boundary value formula \eqref{dsm2-natural} holds, and suppose the
above analytic continuation is taken to \(\zeta=x+i0\).  Then
\begin{equation}
\partial_\kappa H(\chi,\kappa)
=
\frac{2}{\alpha}\mathbf 1_{\{T_\chi(0)>x\}}
-
\frac{2}{\pi\alpha}
\lim_{K\to\infty}
\sum_{\xi\in \mathcal{D}_{\chi,K}}
\left(
   \arg w_{\xi,\chi,K}(x)-\arg\xi
\right).
\label{phk-natural}
\end{equation}
\end{lemma}

\begin{proof}
Since \(\beta=1\), $T_{\chi}=S_{\chi}$, and \(S_{\chi,K}\) is a reduced rational function in
the all-\(L\) setting of Assumption~\ref{ap62}.  Fix \(K\), and put
\[
   F_{K,\zeta}(w):=
   1-\frac{S_{\chi,K}(w)}{\zeta}.
\]
For \(|\zeta|\) large, \(F_{K,\zeta}\) is uniformly close to \(1\) on
\(C_\chi\).  Hence its winding number around \(0\) along \(C_\chi\) is zero.
The poles of \(F_{K,\zeta}\) inside \(C_\chi\) are precisely the points
\(\xi\in \mathcal{D}_{\chi,K}\).  The argument principle therefore implies that
\(F_{K,\zeta}\) has exactly \(|\mathcal{D}_{\chi,K}|\) zeros inside \(C_\chi\), counted
with multiplicity.

Each \(\xi\in \mathcal{D}_{\chi,K}\) is a simple pole of \(S_{\chi,K}\).  Thus, for
\(|\zeta|\) large, the equation \(S_{\chi,K}(w)=\zeta\) has a unique simple
root \(w_{\xi,\chi,K}(\zeta)\) near \(\xi\), and
\[
   w_{\xi,\chi,K}(\zeta)\to\xi .
\]
These roots account for all zeros of \(F_{K,\zeta}\) inside \(C_\chi\).

Choose a small positively oriented circle \(C_0\) around \(0\).  For each
\(\xi\in \mathcal{D}_{\chi,K}\), choose a positively oriented contour \(C_\xi\) enclosing
precisely the pole \(\xi\) and the corresponding root
\(w_{\xi,\chi,K}(\zeta)\), but not \(0\) or any other zero or pole.  Joining
\(w_{\xi,\chi,K}(\zeta)\) to \(\xi\) by a cut inside \(C_\xi\), deform
\(C_\chi\) to
\[
   C_0\cup \bigcup_{\xi\in \mathcal{D}_{\chi,K}}C_\xi
\]
without crossing any zero, pole, or logarithmic cut of \(F_{K,\zeta}\).
Let
\[
   G_{K,\zeta}(w):=\log F_{K,\zeta}(w)
\]
with the branch chosen by continuation from \(\zeta=\infty\).

On \(C_0\), \(G_{K,\zeta}\) is holomorphic near \(0\), and Cauchy's formula
gives
\[
-\frac{1}{\alpha\pi i}
\oint_{C_0}G_{K,\zeta}(w)\frac{dw}{w}
=
-\frac{2}{\alpha}G_{K,\zeta}(0)
=
-\frac{2}{\alpha}
\log\left(1-\frac{S_{\chi,K}(0)}{\zeta}\right).
\]

Now fix \(\xi\in \mathcal{D}_{\chi,K}\).  Since \(0\) lies outside \(C_\xi\), choose a
single-valued holomorphic branch of \(\log w\) near \(C_\xi\).  Also,
\(F_{K,\zeta}\) has one zero and one pole inside \(C_\xi\), so its winding
number along \(C_\xi\) is zero and \(G_{K,\zeta}\) is single-valued near
\(C_\xi\).  Integration by parts gives
\[
-\frac{1}{\alpha\pi i}
\oint_{C_\xi}G_{K,\zeta}(w)\frac{dw}{w}
=
\frac{1}{\alpha\pi i}
\oint_{C_\xi}
\log w\, G_{K,\zeta}'(w)\,dw .
\]
Moreover,
\[
   G_{K,\zeta}'(w)
   =
   \frac{S_{\chi,K}'(w)}{S_{\chi,K}(w)-\zeta}.
\]
This derivative has simple poles at
\(w_{\xi,\chi,K}(\zeta)\) and at \(\xi\), with residues \(1\) and \(-1\),
respectively.  The residue theorem therefore yields
\[
-\frac{1}{\alpha\pi i}
\oint_{C_\xi}G_{K,\zeta}(w)\frac{dw}{w}
=
\frac{2}{\alpha}
\left(
   \log w_{\xi,\chi,K}(\zeta)-\log\xi
\right).
\]
Summing over \(\xi\in \mathcal{D}_{\chi,K}\) proves
\eqref{eq:finite-root-labelled-stieltjes-jack}.

It remains to pass \(K\to\infty\).  By the normal convergence of the reflected
products on the admissible contour neighbourhoods, supplied by Lemma~\ref{le36},
we have
\[
   S_{\chi,K}\longrightarrow S_\chi
\]
uniformly on \(C_\chi\), and the same estimates give
\[
   S_{\chi,K}(0)\longrightarrow S_\chi(0).
\]
Therefore, for \(|\zeta|\) sufficiently large,
\[
-\frac{1}{\alpha\pi i}
\oint_{C_\chi}
\log\left(1-\frac{S_{\chi,K}(w)}{\zeta}\right)\frac{dw}{w}
\longrightarrow
-\frac{1}{\alpha\pi i}
\oint_{C_\chi}
\log\left(1-\frac{S_\chi(w)}{\zeta}\right)\frac{dw}{w}.
\]
By the Stieltjes representation \eqref{stc-natural}, the limiting integral is
\(\operatorname{St}_{\nu_\chi}(\zeta)\).  This proves
\eqref{eq:st-root-formula}.  The continuation statement follows
from uniqueness of analytic continuation, with the same root labels and
logarithm branches continued from \(\zeta=\infty\).

Finally take the upper boundary value \(\zeta=x+i0\).  By
\eqref{dsm2-natural},
\[
   \partial_\kappa H(\chi,\kappa)
   =
   -\frac1\pi
   \Im \operatorname{St}_{\nu_\chi}(x+i0).
\]
The first logarithm in \eqref{eq:st-root-formula} contributes
\[
   -\frac1\pi
   \Im\left[
      -\frac{2}{\alpha}
      \log\left(1-\frac{S_\chi(0)}{x+i0}\right)
   \right]
   =
   \frac{2}{\alpha}\mathbf 1_{\{T_\chi(0)>x\}},
\]
because the boundary value of the logarithm has imaginary part \(\pi\) when
\(T_\chi(0)>x\), and \(0\) when \(T_\chi(0)<x\).  The second term gives
\[
   -\frac{2}{\pi\alpha}
   \lim_{K\to\infty}
   \sum_{\xi\in D_{\chi,K}}
   \left(
      \arg w_{\xi,\chi,K}(x)-\arg\xi
   \right).
\]
This proves \eqref{phk-natural}.
\end{proof}

\begin{lemma}[Upper-boundary roots of the finite truncations]
\label{lem:finite-upper-boundary-roots}

Assume Assumptions~\ref{ap64}--\ref{ap65} and let \(\beta=1\).
Fix
\[
        \chi\in
        (V_0,V_m)\setminus\{V_1,\ldots,V_{m-1}\},
        \qquad
        K\geq1,
        \qquad
        x>0.
\]
Assume that \(\mathcal D_{\chi,K}\neq\varnothing\), and enumerate
its elements increasingly as
\[
        \mathcal D_{\chi,K}
        =
        \{p_{1,K}<\cdots<p_{N_K,K}\}.
\]
Assume further that
\[
        C_{\chi,K}
        :=
        \lim_{w\to\infty}S_{\chi,K}(w)
        \neq x.
\]

Suppose that
\[
        S_{\chi,K}(w)=x
\]
has a unique nonreal conjugate pair
\[
        w_{+,K},\overline{w_{+,K}},
        \qquad
        \operatorname{Im}w_{+,K}>0,
\]
of total multiplicity two.

For \(\delta>0\), let
\[
        W_K(\delta)
        :=
        \left\{
        w\in\mathbb C_+:
        S_{\chi,K}(w)=x+\mathbf i\delta
        \right\},
\]
counted as an unordered multiset with multiplicity.

Then, as \(\delta\downarrow0\),
\[
        W_K(\delta)
        \longrightarrow
        \{w_{+,K},u_{1,K},\ldots,u_{N_K-1,K}\}
\]
as multisets, where \(u_{j,K}\) is the unique simple root of
\(S_{\chi,K}(w)=x\) in
\[
        (p_{j,K},p_{j+1,K}).
\]
In particular, all \(p_{j,K}\) and \(u_{j,K}\) are positive, and
\[
\operatorname{Im}
\sum_{\xi\in\mathcal D_{\chi,K}}
\left(
\log w_{\xi,\chi,K}(x+\mathbf i0)-\log\xi
\right)
=
\arg w_{+,K},
\qquad
\arg w_{+,K}\in(0,\pi).
\]
Here the root and logarithmic boundary values on the left side in the last display
are obtained by continuation of the spectral parameter \(\zeta\)
from large imaginary values to \(x+\mathbf i0\) through
\(\mathbb C_+\), as in
Lemma~\ref{lem:root-labelled-stieltjes-jack}.  Possible
permutations of the individual root labels do not affect the
symmetric sum.
\end{lemma}

\begin{proof}
All roots and poles below are counted with multiplicity.

By Lemma~\ref{lem:separation-implies-interlacing},
Assumption~\ref{ap62} holds.  If
\[
        0<\xi_1<\cdots<\xi_{M_K}
\]
are the positive poles of the reduced rational function
\(S_{\chi,K}\), the same lemma gives
\[
        \mathcal D_{\chi,K}
        =
        \{\xi_1,\ldots,\xi_{r_0}\},
\]
with the empty-set convention when \(r_0=0\).  Consequently,
\[
        N_K=r_0.
\]

Since \(\mathcal D_{\chi,K}\neq\varnothing\), we have \(r_0\geq1\).
The endpoint case \(r_0=M_K\) is also impossible.  Indeed, the
endpoint clause of Assumption~\ref{ap62}, applied with
\[
        s=x^{1/n}>0,
\]
states that
\[
        P_K(w)-xQ_K(w)
\]
has no nonreal zero when \(r_0=M_K\), contrary to the hypothesis of
the present lemma.  Hence
\[
        1\leq N_K=r_0\leq M_K-1.
\]
In particular, the outward block contains
\[
        M_K-N_K\geq1
\]
poles.  Enumerate them increasingly as
\[
        q_{1,K}<\cdots<q_{M_K-N_K,K}.
\]

By the reduced interval factorization
\eqref{eq:radial-interval-factorization}, we may write
\begin{equation}
\label{eq:upper-boundary-factorization}
S_{\chi,K}(w)
=
C_{\chi,K}
\prod_{j=1}^{N_K}
\frac{w-a_{j,K}}{w-p_{j,K}}
\prod_{\ell=1}^{M_K-N_K}
\frac{w-b_{\ell,K}}{w-q_{\ell,K}},
\end{equation}
where
\[
        p_{j,K}<a_{j,K},
        \qquad
        b_{\ell,K}<q_{\ell,K},
\]
the intervals
\[
        [p_{j,K},a_{j,K}],
        \qquad
        [b_{\ell,K},q_{\ell,K}]
\]
are pairwise disjoint and ordered, and
\[
        C_{\chi,K}>0.
\]
Indeed, in the all-\(L\) specialization every elementary factor in the
unreduced product defining \(S_{\chi,K}\) has a positive limit as
\(w\to+\infty\).  Cancellation of common
numerator--denominator factors does not change this limit.
The first product in
\eqref{eq:upper-boundary-factorization} consists of the inward
intervals, while the second consists of the outward intervals.

We first determine the residue signs.  At an inward pole
\(p_{j,K}\), every factor in
\eqref{eq:upper-boundary-factorization} other than the factor
associated with \([p_{j,K},a_{j,K}]\) is positive, because
\(p_{j,K}\) lies outside every other interval.  Therefore
\[
\begin{aligned}
\operatorname*{Res}_{w=p_{j,K}}S_{\chi,K}(w)
&=
C_{\chi,K}(p_{j,K}-a_{j,K})\\
&\quad\times
\prod_{\substack{m=1\\m\neq j}}^{N_K}
\frac{p_{j,K}-a_{m,K}}{p_{j,K}-p_{m,K}}
\prod_{\ell=1}^{M_K-N_K}
\frac{p_{j,K}-b_{\ell,K}}{p_{j,K}-q_{\ell,K}}
<0.
\end{aligned}
\]
Similarly,
\[
        \operatorname*{Res}_{w=q_{\ell,K}}
        S_{\chi,K}(w)>0
\]
at every outward pole.  For any pole \(r\), write
\[
        \rho(r)
        :=
        \operatorname*{Res}_{w=r}S_{\chi,K}(w).
\]

Near a simple pole \(r\),
\[
        S_{\chi,K}(w)
        =
        \frac{\rho(r)}{w-r}+O(1).
\]
Hence the solution of
\[
        S_{\chi,K}(w)=\zeta
\]
which converges to \(r\) as \(\lvert\zeta\rvert\to\infty\) satisfies
\[
        w_r(\zeta)
        =
        r+\frac{\rho(r)}{\zeta}
        +O\!\left(\frac{1}{\lvert\zeta\rvert^2}\right).
\]
Taking \(\zeta=\mathbf iR\), \(R\to\infty\), gives
\[
        w_r(\mathbf iR)
        =
        r-\mathbf i\frac{\rho(r)}{R}
        +O(R^{-2}).
\]
It follows that, for all sufficiently large \(R\),
\[
\begin{cases}
w_{p_{j,K}}(\mathbf iR)\in\mathbb C_+,
& \rho(p_{j,K})<0,\\[2mm]
w_{q_{\ell,K}}(\mathbf iR)\in\mathbb C_-,
& \rho(q_{\ell,K})>0.
\end{cases}
\]

Write the reduced rational function as
\[
        S_{\chi,K}(w)=\frac{P_K(w)}{Q_K(w)}
\]
and set
\[
        F_{K,\zeta}(w)
        :=
        P_K(w)-\zeta Q_K(w).
\]
Since
\[
        \deg P_K=\deg Q_K=M_K
\]
and
\[
        C_{\chi,K}
        =
        \lim_{w\to\infty}S_{\chi,K}(w)\in\mathbb R,
\]
the leading coefficient of \(F_{K,\zeta}\) is nonzero for every
\(\zeta\in\mathbb C_+\).  Thus \(F_{K,\zeta}\) has exactly \(M_K\)
finite roots, counted with multiplicity.

Moreover, \(F_{K,\zeta}\) has no real root for
\(\zeta\in\mathbb C_+\).  Indeed, away from its poles,
\(S_{\chi,K}\) is real-valued on the real axis, while at a pole one
has
\[
        Q_K(w)=0,
        \qquad
        P_K(w)\neq0,
\]
because \(P_K/Q_K\) is reduced.  Consequently, as \(\zeta\) varies
in the connected domain \(\mathbb C_+\), the number of roots of
\(F_{K,\zeta}\) in \(\mathbb C_+\), counted with multiplicity, is
constant: this number can change only if a root reaches the real
axis or escapes to infinity, neither of which is possible.

For \(\zeta=\mathbf iR\) with \(R\) sufficiently large, the pole
expansions above and the degree count show that exactly the \(N_K\)
roots issuing from
\[
        \mathcal D_{\chi,K}
        =
        \{p_{1,K},\ldots,p_{N_K,K}\}
\]
lie in \(\mathbb C_+\).  Hence, for every \(\zeta\in\mathbb C_+\),
the equation
\[
        S_{\chi,K}(w)=\zeta
\]
has exactly \(N_K\) upper-half-plane roots, counted with
multiplicity.

Suppose now that \(\zeta\in\mathbb C_+\) is such that the equation
\(S_{\chi,K}(w)=\zeta\) has no multiple root.  Continue the
pole-labelled branches
\[
        w_{\xi,\chi,K}(\zeta),
        \qquad \xi\in\mathcal D_{\chi,K},
\]
from \(\zeta=\mathbf iR\), \(R\gg1\), along a path in \(\mathbb C_+\) which avoids the finitely many
values of \(\zeta\) for which
\[
        S_{\chi,K}(w)=\zeta
\]
has a multiple root. Every such branch starts in
\(\mathbb C_+\) and cannot cross the real axis while the spectral
parameter remains in \(\mathbb C_+\).  Since there are exactly
\(N_K\) upper-half-plane roots, we obtain
\begin{equation}
\label{eq:labelled-roots-are-upper-roots}
\left\{
w_{\xi,\chi,K}(\zeta):
\xi\in\mathcal D_{\chi,K}
\right\}
=
\left\{
w\in\mathbb C_+:
S_{\chi,K}(w)=\zeta
\right\}
\end{equation}
as unordered multisets.  

We next take the upper boundary value
\[
        \zeta=x+\mathbf i\delta,
        \qquad
        \delta\downarrow0.
\]
Since \(C_{\chi,K}\neq x\), the polynomial
\[
        F_{K,x}(w)=P_K(w)-xQ_K(w)
\]
has degree \(M_K\).

At two consecutive inward poles, the negative residue signs give
\[
\lim_{w\to p_{j,K}^{+}}
\bigl(S_{\chi,K}(w)-x\bigr)
=
-\infty,
\qquad
\lim_{w\to p_{j+1,K}^{-}}
\bigl(S_{\chi,K}(w)-x\bigr)
=
+\infty.
\]
Therefore, for every \(j=1,\ldots,N_K-1\), the equation
\[
        S_{\chi,K}(w)=x
\]
has at least one root in
\[
        (p_{j,K},p_{j+1,K}).
\]

Similarly, the positive residue signs at the outward poles give
\[
\lim_{w\to q_{\ell,K}^{+}}
\bigl(S_{\chi,K}(w)-x\bigr)
=
+\infty,
\qquad
\lim_{w\to q_{\ell+1,K}^{-}}
\bigl(S_{\chi,K}(w)-x\bigr)
=
-\infty.
\]
Hence there is at least one root in every interval
\[
        (q_{\ell,K},q_{\ell+1,K}),
        \qquad
        \ell=1,\ldots,M_K-N_K-1.
\]

The number of these pairwise disjoint intervals is
\[
        (N_K-1)+(M_K-N_K-1)=M_K-2.
\]
By hypothesis, the equation \(S_{\chi,K}(w)=x\) also has one
nonreal conjugate pair of total multiplicity two.  Since
\(F_{K,x}\) has degree \(M_K\), these \(M_K-2\) real roots and the
two nonreal roots exhaust all roots, counted with multiplicity.
It follows that every interval above contains exactly one root and
that each such root is simple.

Write
\[
        u_{j,K}\in(p_{j,K},p_{j+1,K}),
        \qquad
        j=1,\ldots,N_K-1,
\]
for the inward roots, and
\[
        v_{\ell,K}\in(q_{\ell,K},q_{\ell+1,K}),
        \qquad
        \ell=1,\ldots,M_K-N_K-1,
\]
for the outward roots.  The endpoint signs and uniqueness imply
\[
        S_{\chi,K}'(u_{j,K})>0,
        \qquad
        S_{\chi,K}'(v_{\ell,K})<0.
\]

All roots of \(F_{K,x}\) are simple.  Therefore, for all sufficiently
small \(\delta>0\), the equation
\[
        S_{\chi,K}(w)=x+\mathbf i\delta
\]
also has only simple roots, and
\eqref{eq:labelled-roots-are-upper-roots} applies.

For a simple real root \(r\) of \(S_{\chi,K}(w)=x\), the implicit
function theorem gives a local solution of
\[
        S_{\chi,K}(w)=x+\mathbf i\delta
\]
of the form
\[
        w_r(\delta)
        =
        r+\frac{\mathbf i\delta}{S_{\chi,K}'(r)}
        +O(\delta^2).
\]
Consequently, for all sufficiently small \(\delta>0\),
\[
        w_{u_{j,K}}(\delta)\in\mathbb C_+,
        \qquad
        w_{v_{\ell,K}}(\delta)\in\mathbb C_-.
\]
The simple nonreal root \(w_{+,K}\) also has a unique nearby
continuation which remains in \(\mathbb C_+\), while its conjugate
remains in \(\mathbb C_-\).

Thus, for sufficiently small \(\delta>0\), the complete multiset
\(W_K(\delta)\) consists of one root converging to \(w_{+,K}\) and
the \(N_K-1\) roots converging respectively to
\[
        u_{1,K},\ldots,u_{N_K-1,K}.
\]
Therefore
\[
        W_K(\delta)
        \longrightarrow
        \{w_{+,K},u_{1,K},\ldots,u_{N_K-1,K}\}
\]
as multisets.

Finally, by
\eqref{eq:labelled-roots-are-upper-roots}, for all sufficiently
small \(\delta>0\),
\[
\left\{
w_{\xi,\chi,K}(x+\mathbf i\delta):
\xi\in\mathcal D_{\chi,K}
\right\}
=
W_K(\delta)
\]
as unordered multisets.  All these roots lie in \(\mathbb C_+\).
Hence their logarithms, obtained by continuation from large
imaginary \(\zeta\), agree with the principal logarithm on
\(\mathbb C_+\).  Since the poles \(p_{j,K}\) and roots
\(u_{j,K}\) are positive,
\[
        \arg p_{j,K}=0,
        \qquad
        \lim_{\delta\downarrow0}
        \arg w_{u_{j,K}}(\delta)=0.
\]
The remaining upper-half-plane root converges to \(w_{+,K}\), with
\[
        \arg w_{+,K}\in(0,\pi).
\]
Possible permutations of the pole labels do not affect the sum.
Therefore
\[
\operatorname{Im}
\sum_{\xi\in\mathcal D_{\chi,K}}
\left(
\log w_{\xi,\chi,K}(x+\mathbf i0)-\log\xi
\right)
=
\arg w_{+,K}.
\]
\end{proof}

\begin{proof}[Proof of Theorem~\ref{p412}]

Fix a regular slope point satisfying the stated conditions and put
\[
        x=e^{-n\kappa}.
\]

By Lemma~\ref{lem:separation-implies-interlacing},
Assumptions~\ref{ap64}--\ref{ap65} imply
Assumption~\ref{ap62}.  In addition,
if
\[
        0<\xi_1<\cdots<\xi_{M_K}
\]
are the positive poles of \(S_{\chi,K}\), then
\[
        \mathcal D_{\chi,K}
        =
        \{\xi_1,\ldots,\xi_{r_0}\}.
\]

By condition~\textup{(iii)} and
Lemma~\ref{lem:local-polefree-noescape}, applied with
\[
        (\chi_0,x_0)=(\chi,x),
\]
there exist a
pole-free open set \(U\subset\mathbb C^*\) and a compact set
\[
        \Omega\Subset U
\]
such that
\[
        S_{\chi,K}\longrightarrow S_\chi
\]
locally uniformly on \(U\), and all nonreal roots of the finite and
limiting equations in a neighborhood of \((\chi,x)\) lie in
\(\Omega\).

By Lemma~\ref{lem:signed-root-count}, each finite equation
\[
        S_{\chi,K}(w)=x
\]
has at most one nonreal conjugate pair.  If the limiting equation had
at least two upper-half-plane roots, counted with multiplicity,
Rouché's theorem on disjoint disks around these roots would imply
that the finite equations have at least two upper-half-plane roots
for all sufficiently large \(K\), a contradiction.  Thus the
limiting equation has at most one nonreal conjugate pair.

Suppose first that \((\chi,\kappa)\) is liquid.  If along an infinite
subsequence the finite equation had only real roots, every argument
in \eqref{phk-natural} would belong to \(\pi\mathbb Z\).  Since the
argument sums converge, their limit would also belong to
\(\pi\mathbb Z\), and \eqref{phk-natural}, together with
\[
        0\le
        \partial_\kappa\mathcal H
        \le\frac{2}{\alpha},
\]
would force
\[
        \partial_\kappa\mathcal H
        \in
        \left\{0,\frac{2}{\alpha}\right\},
\]
contradicting liquidness.

Therefore, for all sufficiently large \(K\), the finite equation has
a unique nonreal conjugate pair
\[
        w_{+,K},\overline{w_{+,K}},
        \qquad
        \operatorname{Im}w_{+,K}>0.
\]
The endpoint cases \(r_0=0\) and \(r_0=M_K\) in
Assumption~\ref{ap62} are then impossible, because in those cases
the finite equation has no nonreal root.

Lemma~\ref{l66} gives
\[
        S_{\chi,K}(0)>x,
\]
and Lemma~\ref{lem:finite-upper-boundary-roots} gives
\[
\operatorname{Im}
\sum_{\xi\in\mathcal D_{\chi,K}}
\left(
\log w_{\xi,\chi,K}(x+\mathbf i0)-\log\xi
\right)
=
\arg w_{+,K}.
\]
Since
\[
        S_{\chi,K}(0)\longrightarrow S_\chi(0)
\]
and \(S_\chi(0)\neq x\), it follows that
\[
        S_\chi(0)>x.
\]
Consequently \eqref{phk-natural} becomes
\[
        \partial_\kappa\mathcal H(\chi,\kappa)
        =
        \frac{2}{\alpha}
        -
        \lim_{K\to\infty}
        \frac{2\arg w_{+,K}}{\pi\alpha}.
\]
Because the point is liquid, the limiting argument belongs to
\((0,\pi)\).

By Lemma~\ref{lem:local-polefree-noescape}, there is a fixed compact
set
\[
        \Omega\Subset U,
\]
independent of \(K\), such that
\[
        w_{+,K}\in\Omega
\]
for all sufficiently large \(K\).  Hence every subsequence of
\(\{w_{+,K}\}\) has a further subsequence converging to some
\[
        w_*\in\Omega.
\]
Since
\[
        \arg w_{+,K}\longrightarrow\theta
        \qquad\text{for some }\theta\in(0,\pi),
\]
and since \(\Omega\Subset\mathbb C^*\), every such limit satisfies
\[
        w_*\in\mathbb C_+.
\]

Along the convergent further subsequence, the local uniform
convergence
\[
        S_{\chi,K}\longrightarrow S_\chi
        \qquad\text{on }U
\]
and the identities
\[
        S_{\chi,K}(w_{+,K})=x
\]
give
\[
        S_\chi(w_*)=x.
\]
Thus every subsequential limit is an upper-half-plane root of the
limiting equation.  Since that equation has at most one
upper-half-plane root, every subsequential limit is equal to the
same point \(w_+\).  Consequently,
\[
        w_{+,K}\longrightarrow w_+.
\]
This proves
\[
        \partial_\kappa\mathcal H(\chi,\kappa)
        =
        \frac{2}{\alpha}
        -
        \frac{2\arg w_+}{\pi\alpha},
        \qquad
        \arg w_+\in(0,\pi).
\]

Conversely, suppose that
\[
        S_\chi(w)=x
\]
has a nonreal conjugate pair
\[
        w_+,\overline{w_+},
        \qquad
        \operatorname{Im}w_+>0.
\]
Since
\[
        w_+,\overline{w_+}\in\Omega\Subset U,
\]
choose disjoint open disks \(D_+\) and \(D_-\), centered at
\(w_+\) and \(\overline{w_+}\), respectively, such that
\[
        \overline{D_+}\subset U\cap\mathbb C_+,
        \qquad
        \overline{D_-}\subset U\cap\mathbb C_-.
\]  Rouché's theorem implies that, for all
sufficiently large \(K\), the finite equation has a nonreal conjugate
pair
\[
        w_{+,K},\overline{w_{+,K}},
        \qquad
        w_{+,K}\longrightarrow w_+.
\]
By Lemma~\ref{lem:signed-root-count}, this pair is unique.
Lemmas~\ref{l66} and
\ref{lem:finite-upper-boundary-roots} therefore give
\[
        S_{\chi,K}(0)>x
\]
and
\[
\operatorname{Im}
\sum_{\xi\in\mathcal D_{\chi,K}}
\left(
\log w_{\xi,\chi,K}(x+\mathbf i0)-\log\xi
\right)
=
\arg w_{+,K}.
\]
Passing to the limit in \eqref{phk-natural} yields
\[
        \partial_\kappa\mathcal H(\chi,\kappa)
        =
        \frac{2}{\alpha}
        -
        \frac{2\arg w_+}{\pi\alpha}.
\]
Since \(0<\arg w_+<\pi\), this slope belongs to
\[
        \left(0,\frac{2}{\alpha}\right).
\]
Hence \((\chi,\kappa)\) is liquid.

We have proved the liquid equivalence.  At a regular slope point the
slope lies in \([0,2/\alpha]\).  Therefore a point is frozen exactly
when it is not liquid, which, by the preceding equivalence and the
upper bound of one nonreal pair, is equivalent to the all-real-root
condition.

Finally, fix a point
\[
        (\chi_0,\kappa_0)
\]
in the above nonexceptional region, set
\[
        x_0:=e^{-n\kappa_0},
        \qquad
        F_0(w):=S_{\chi_0}(w)-x_0,
\]
and suppose that \(F_0\) has no multiple real zero.

Apply Lemma~\ref{lem:local-polefree-noescape} at
\((\chi_0,x_0)\).  Then there exist open intervals
\[
        I\ni\chi_0,
        \qquad
        X\ni x_0,
\]
a pole-free open set
\[
        U\subset\mathbb C^*,
\]
and a compact set
\[
        \Omega\Subset U
\]
such that, whenever
\[
        \chi\in I,
        \qquad
        e^{-n\kappa}\in X,
\]
every nonreal zero of
\[
        F_{\chi,\kappa}(w)
        :=
        S_\chi(w)-e^{-n\kappa}
\]
lies in \(\Omega\).

Since
\[
        \chi_0\in(V_{p-1},V_p)
\]
for some \(p\in[m]\), shrink \(I\), if necessary, so that
\[
        \overline I\Subset(V_{p-1},V_p).
\]
Then, for every \(\chi\in I\), the same factors occur in the
finite products defining \(\mathcal G_\chi\); only their
\(\chi\)-dependent zero and pole locations vary continuously with
\(\chi\). After
shrinking \(I\) once more, the same numerator--denominator factors
are cancelled in the reduced functions for every \(\chi\in I\).
Each remaining unreflected zero or pole is therefore either
independent of \(\chi\) or varies continuously with \(\chi\).  Since the reflected
factors are independent of \(\chi\) and their product converges
normally on compact subsets of \(U\), it follows that
\[
        S_\chi\longrightarrow S_{\chi_0}
\]
locally uniformly on \(U\) as \(\chi\to\chi_0\).
Consequently,
\[
        S_\chi(w)-e^{-n\kappa}
        \longrightarrow
        S_{\chi_0}(w)-e^{-n\kappa_0}
\]
locally uniformly on \(U\) as
\[
        (\chi,\kappa)\longrightarrow(\chi_0,\kappa_0).
\]
Together with the normal convergence of the reflected product on
compact subsets of \(U\), this gives
\[
        F_{\chi,\kappa}
        \longrightarrow
        F_0
\]
locally uniformly on \(U\) as
\[
        (\chi,\kappa)
        \longrightarrow
        (\chi_0,\kappa_0).
\]

Choose an open set \(V\) such that
\[
        \Omega\subset V\Subset U
\]
and
\[
        F_0(w)\neq0,
        \qquad
        w\in\partial V.
\]
The function \(F_0\) has only finitely many zeros in
\(\overline V\).  Around each of its distinct nonreal zeros choose
a small open disk whose closure is contained in \(V\), lies entirely
in the corresponding half-plane, and contains no other zero of
\(F_0\).  Around each real zero choose a small open disk centered on
the real axis, symmetric under complex conjugation, whose closure is
contained in \(V\) and which contains no other zero of \(F_0\).
Choose all these disks with pairwise disjoint closures.

By hypothesis, every real zero of \(F_0\) is simple.  Hence each
real-centered disk contains exactly one zero of \(F_0\), counted with
multiplicity.  Let \(D_1,\ldots,D_N\) be the chosen open disks and set
\[
        K
        :=
        \overline V
        \setminus
        \bigcup_{j=1}^N D_j.
\]
By construction, \(F_0\) has no zero on \(K\).  Since \(K\) is
compact,
\[
        m
        :=
        \min_{w\in K}|F_0(w)|
        >0.
\]
After shrinking the parameter neighborhood, the locally uniform
convergence gives
\[
        \sup_{w\in K}
        |F_{\chi,\kappa}(w)-F_0(w)|
        <
        \frac m2.
\]
Hence
\[
        |F_{\chi,\kappa}(w)|
        \geq\frac m2,
        \qquad w\in K,
\]
so \(F_{\chi,\kappa}\) has no zero outside the chosen disks in
\(\overline V\).  The locally uniform convergence above therefore implies
that, for all \((\chi,\kappa)\) sufficiently close to
\((\chi_0,\kappa_0)\),

\begin{enumerate}
\item \(F_{\chi,\kappa}\) has no zero on that compact complement;

\item by Rouch\'e's theorem, every chosen disk contains the same
number of zeros of \(F_{\chi,\kappa}\), counted with multiplicity,
as it contains zeros of \(F_0\).
\end{enumerate}

A disk centered at a nonreal zero is contained in one of the two
open half-planes, so the number of upper-half-plane zeros contributed
by such a disk is unchanged.  A disk centered at a real zero contains
exactly one nearby zero.  Since 
\[
        F_{\chi,\kappa}(\overline w)
        =
        \overline{F_{\chi,\kappa}(w)},
\]
this unique zero must remain real: otherwise its conjugate would be
a second distinct zero in the same conjugation-invariant disk.

Thus the number of zeros of \(F_{\chi,\kappa}\) in
\(\mathbb C_+\), counted with multiplicity, is unchanged inside
\(V\).  Since Lemma~\ref{lem:local-polefree-noescape} places every
nearby nonreal zero in
\[
        \Omega\subset V,
\]
there are no additional upper-half-plane zeros outside \(V\).
Consequently, the number of upper-half-plane roots of the
characteristic equation is locally constant at
\((\chi_0,\kappa_0)\).

At every point of \(\operatorname{int}\mathscr N\), the number of
upper-half-plane roots is locally constant whenever the
characteristic equation has no multiple real root.  Consequently,
the portion of the regular frozen boundary lying in
\(\operatorname{int}\mathscr N\) is contained in the real
multiple-root locus \eqref{fb}.  At such a boundary point there
exists \(w\in\mathbb R\) such that
\[
        S_\chi(w)=e^{-n\kappa},
        \qquad
        S_\chi'(w)=0.
\]
Since
\[
        S_\chi(w)=e^{-n\kappa}>0,
\]
the second equation is equivalent to
\[
        \frac{d}{dw}
        \mathrm{Log}_{\chi,\mathcal C_\chi}S_\chi(w)=0.
\]
At a nondegenerate fold point this multiple root has multiplicity
two, so \eqref{fb} is the usual double-root system.  This proves the
multiple-root assertion.
\end{proof}

\section{Height Laplace-Test Fluctuations and the Annular Image Kernel}\label{sect:gff}

In this section we prove Theorem~\ref{t77}.  The two-point asymptotics yield a finite-dimensional Gaussian theorem for height Laplace-test observables.  The charge-centered coordinate of Definition~\ref{def:charge-centered-qt-coordinate} removes the random charge factor in Lemma~\ref{le25} exactly, so the covariance is obtained directly from the Negu\c t-observable covariance and is the annular image kernel of Proposition~\ref{prop:annular-gff-kernel}.  Thus the limiting Laplace-test covariance is not the pullback of the standard upper-half-plane covariance; it is the annular free-boundary image covariance generated by the two reflected boundary images.

\begin{proof}[Proof of Theorem~\ref{t77}]
Fix points
\[
\chi_1,\dots,\chi_s\in(V_0,V_m)\setminus\{V_1,\ldots,V_{m-1}\}
\]
and positive integers $k_1,\dots,k_s$.  For each $d\in[s]$, choose a sequence
$i_d^{(\epsilon)}\in[l^{(\epsilon)}..r^{(\epsilon)}]$ such that
\[
\epsilon i_d^{(\epsilon)}\to \chi_d
\]
and such that the residue class of $i_d^{(\epsilon)}$ modulo $n$ is independent of $\epsilon$.  

Define
\[
X_{k_d}^{(\epsilon)}(\chi_d)
:=
\int_{-\infty}^{\infty}
\left(
 h^{(q,t),\circ}_{M^{(\epsilon)}}\!\left(i_d^{(\epsilon)},\frac{\eta}{\epsilon}\right)
 -
 \EE\left[h^{(q,t),\circ}_{M^{(\epsilon)}}\!\left(i_d^{(\epsilon)},\frac{\eta}{\epsilon}\right)\right]
\right)e^{-n\beta k_d\eta}\,d\eta.
\]
By Corollary~\ref{cor:charge-centered-height-negut}, with the change of variables
$\eta=\epsilon y$, we have the exact centered identity
\[
X_{k_d}^{(\epsilon)}(\chi_d)
=
\frac{2\epsilon^2}{k_d^2\log t\log q}
Q_{k_d}^{(\epsilon)}(\epsilon i_d^{(\epsilon)}).
\]
Indeed, the factor $t^{k_d\tilde c^{(M^{(\epsilon)},i_d^{(\epsilon)},q,t)}}$ in Lemma~\ref{le25} is cancelled by the charge-centered vertical translation.  Moreover,
\[
\log t= -n\beta\epsilon,
\qquad
\log q= -\alpha n\beta\epsilon,
\]
so
\[
\frac{2\epsilon^2}{k_d^2\log t\log q}
\longrightarrow
\frac{2}{\alpha n^2\beta^2 k_d^2}.
\]
Theorem~\ref{t58} therefore implies that the vector
\[
\bigl(X_{k_1}^{(\epsilon)}(\chi_1),\dots,X_{k_s}^{(\epsilon)}(\chi_s)\bigr)
\]
converges to a centered Gaussian vector.  Multiplying the covariance in \eqref{corr-cov-t58} with $g_d=k_d$ and $g_h=k_h$ by the two normalization factors
\[
\frac{2}{\alpha n^2\beta^2 k_d^2}
\qquad\text{and}\qquad
\frac{2}{\alpha n^2\beta^2 k_h^2}
\]
gives exactly \eqref{main-corrected-fluc-cov}.  This proves Theorem~\ref{t77}.
\end{proof}

\section{Example: Half-space Macdonald Process}\label{sect:hsmp}

The degeneration corresponding to the rail-yard realization of the half-space Macdonald-process geometry is obtained as the $v\downarrow0$ limit of the present model with free left boundary and empty right boundary, namely
\begin{align*}
 u>0,
 \qquad
 v\downarrow0,
 \qquad
 c_l=deel.
\end{align*}
All formulas in this section are understood as limits of the doubly free-boundary formulas; this avoids substituting $v=0$ directly into expressions that contain powers $v^{-1}$.  We do not use a separate identification theorem with the notation of~\cite{BBC20}; the statements below should be read as asymptotic consequences for this rail-yard specialization.

\begin{lemma}[Uniform half-space degeneration]\label{lem:uniform-halfspace}
For the contour families fixed in the preceding sections, the limits below are taken in domains where the one-point and two-point integrands are normally dominated, uniformly for all sufficiently small \(v\).  In particular,
\[
\prod_{k\ge1}\mathcal F_{u,v,k}(w)\longrightarrow \mathcal F_{u,0,1}(w)
\]
normally on the one-point contour neighborhoods, and
\[
\mathscr B_{u,v}(z,w)\longrightarrow \mathscr B_{u,0}(z,w)
\]
normally on the two-point product contour neighborhoods.
\end{lemma}

\begin{proof}
This is a direct consequence of the geometric estimates in Lemma~\ref{le36} and the annular expansion \eqref{annular-Buv}.  The factors with \(k\ge2\) contain positive powers of \(v\) after the cancellations already used in Lemma~\ref{le36}, hence converge normally to one in the one-point product.  For the covariance correction, \eqref{annular-Buv} gives uniform domination by a convergent majorant for all small \(v\), and termwise convergence yields the stated limit.
\end{proof}

In this regime the partition function is given by \eqref{fep}. Moreover, the infinite product in Theorem~\ref{l61} collapses substantially: for every $k\ge 2$ one has
\begin{align*}
 \mathcal F_{u,0,k}(w)=1,
\end{align*}
and for $k=1$,
\begin{align*}
 \mathcal F_{u,0,1}(w)
 =
 \mathcal G_{<V_m,L}(u^2w)
 \bigl[\mathcal G_{<V_m,R}(u^2w)\bigr]^{\alpha}.
\end{align*}
Consequently Theorem~\ref{l61} gives the following half-space degeneration of the Laplace-transform observable.  For $\gamma=k\beta$, $k\in\ZZ_{>0}$, define
\[
L_{u,0}^{(\epsilon)}(i^{(\epsilon)},\gamma)
:=\int_{-\infty}^{\infty} e^{-n\gamma\kappa}
\epsilon h^{(q,t)}_{M^{(\epsilon)}}\!\left(i^{(\epsilon)},\frac{\kappa}{\epsilon}\right)d\kappa .
\]
Then, under the same contour and branch-admissibility hypotheses as Theorem~\ref{l61},
\begin{align}
 L_{u,0}^{(\epsilon)}(i^{(\epsilon)},\gamma)
 \xrightarrow[\epsilon\to0]{\mathbb P}
 \frac{1}{n^2\alpha\gamma^2\pi\mathbf i}
 \oint_{\mathcal C}
 \bigl[\mathcal G_{\chi}(w)\mathcal F_{u,0,1}(w)\bigr]^{\gamma}
 \frac{dw}{w}.
 \label{hs-laplace}
\end{align}
Applying the moment-determinacy argument of Theorem~\ref{thm:limit-shape-beta} in this degeneration gives the corresponding limit shape in the natural variable $x=e^{-n\beta\kappa}$.

When \(\beta=1\), under the \(v\downarrow0\) analogues of
Assumptions~\ref{ap64}--\ref{ap65}, the relevant one-point contour
and branch conditions, and the nonexceptionality conditions of
Theorem~\ref{p412}, the regular multiple-root condition becomes
\begin{equation}
\begin{cases}
\displaystyle
\mathcal G_\chi(w)\mathcal F_{u,0,1}(w)
=
e^{-n\kappa},\\[3mm]
\displaystyle
\frac{d}{dw}
\log\!\left(
\mathcal G_\chi(w)\mathcal F_{u,0,1}(w)
\right)
=
0,
\end{cases}
\qquad
w\in\mathbb R.
\end{equation}
This statement is conditional in the same sense as
Theorem~\ref{p412}: the limit-shape realization follows from the
preceding moment argument, whereas the \(v\downarrow0\) analogues
of the separated zero--pole, boundary-value, and endpoint
nonexceptionality conditions are not verified separately here.
Finally, under the same \(L\)-type marked-column, contour, branch,
and normal-convergence hypotheses as Theorem~\ref{t77}, with height
observables written in the charge-centered \((q,t)\)-column
coordinate, the annular Laplace-test Gaussian fluctuation theorem
specializes to this half-space setting with
\[
\mathscr B_{u,0}(z,w)=\sum_{r\ge1}r u^{2r}z^rw^r.
\]
Thus the limiting Laplace-test covariance is the one-boundary degeneration of the annular free-boundary image covariance.  It is not the standard upper-half-plane kernel in general; the displayed series is the surviving reflected-image contribution from the remaining free boundary.

\appendix

\section{Technical Results}\label{sc:dmp}

In this appendix we collect the background on Macdonald polynomials and the technical lemmas used in Sections~\ref{sect:lthf} and~\ref{sect:as}.

\subsection{Macdonald polynomials and scalar products}

Let $\mathcal A=\bigoplus_{k\ge 0}\mathcal A_k$ be a $\ZZ_{\ge 0}$-graded algebra. Its topological completion is the algebra of formal series
\begin{align*}
 a=\sum_{k=0}^{\infty}a_k,
 \qquad a_k\in\mathcal A_k.
\end{align*}
For $a\in\overline{\mathcal A}$, define the lower degree by
\begin{align}
 \operatorname{ldeg}(a):=\inf\{k\ge 0:a_k\neq 0\},
 \label{dldeg}
\end{align}
with the convention $\operatorname{ldeg}(0)=+\infty$.

\begin{definition}\label{dsp1}
Let $\Lambda_X$ be the algebra of symmetric functions in a countable set of variables $X=(x_1,x_2,\ldots)$. A \emph{specialization} is an algebra homomorphism
\begin{align*}
 \rho:\Lambda_X\to\mathcal B,
\end{align*}
where $\mathcal B$ is a commutative $\CC$-algebra. For $f\in\Lambda_X$, we write $f(\rho):=\rho(f)$.
\end{definition}

Define the power-sum symmetric functions by
\begin{align}
 p_n(X):=
 \begin{cases}
 1,& n=0,\\
 \displaystyle\sum_{x\in X}x^n,& n\ge 1.
 \end{cases}
 \label{dpn}
\end{align}
Then $\Lambda_X$ is generated by $\{p_n(X)\}_{n\ge 1}$, so a specialization is determined by the values $\{p_n(\rho)\}_{n\ge 1}$.

For a partition $\lambda=(\lambda_1,\lambda_2,\ldots)$, let
\begin{align*}
 p_{\lambda}:=\prod_{i\ge 1}p_{\lambda_i},
 \qquad
 z_{\lambda}:=\prod_{r\ge 1}r^{m_r(\lambda)}m_r(\lambda)!,
\end{align*}
where $m_r(\lambda)$ is the multiplicity of the part $r$ in $\lambda$. For fixed $q,t\in(0,1)$, define the Macdonald scalar product on $\Lambda_X$ by
\begin{align}
 \langle p_{\lambda},p_{\mu}\rangle_{q,t}
 :=
 \delta_{\lambda\mu}
 \left[\prod_{i=1}^{l(\lambda)}\frac{1-q^{\lambda_i}}{1-t^{\lambda_i}}\right]
 z_{\lambda}.
 \label{dsp}
\end{align}

For each $\lambda\in\YY$, let $P_{\lambda}(X;q,t)$ denote the normalized Macdonald polynomial characterized by the triangular expansion
\begin{align*}
 P_{\lambda}(X;q,t)=m_{\lambda}(X)+\sum_{\mu<\lambda}u_{\lambda\mu}(q,t)m_{\mu}(X),
\end{align*}
where $m_{\mu}$ is the monomial symmetric function. Define $Q_{\lambda}(X;q,t)$ by the normalization
\begin{align*}
 \langle P_{\lambda}(X;q,t),Q_{\mu}(X;q,t)\rangle_{q,t}=\delta_{\lambda\mu}.
\end{align*}
The skew Macdonald polynomials are defined by
\begin{align}
 P_{\lambda}(X,Y;q,t)
 &=\sum_{\mu\in\YY}P_{\lambda/\mu}(X;q,t)P_{\mu}(Y;q,t),
 \label{pls}\\
 Q_{\lambda}(X,Y;q,t)
 &=\sum_{\mu\in\YY}Q_{\lambda/\mu}(X;q,t)Q_{\mu}(Y;q,t).
 \label{qls}
\end{align}

We also use the $q$-Pochhammer symbols
\begin{align}
 (a;q)_N:=\prod_{r=0}^{N-1}(1-aq^r),
 \qquad
 (a;q)_{\infty}:=\prod_{r=0}^{\infty}(1-aq^r).
 \label{daq}
\end{align}

\begin{lemma}\label{la1}
Let $X$ be a countable set of variables. Then $P_{\lambda/\mu}(X;q,t)=Q_{\lambda/\mu}(X;q,t)=0$ unless $\mu\subseteq\lambda$. If $\mu\subseteq\lambda$, then both $P_{\lambda/\mu}(X;q,t)$ and $Q_{\lambda/\mu}(X;q,t)$ are homogeneous of degree $|\lambda|-|\mu|$.
\end{lemma}

\begin{proof}
See \cite[(7.7), p.~344]{MG95}.
\end{proof}

In particular, when $q=t$,
\begin{align}
 P_{\lambda}(X;t,t)=Q_{\lambda}(X;t,t)=s_{\lambda}(X),
 \qquad
 P_{\lambda/\mu}(X;t,t)=Q_{\lambda/\mu}(X;t,t)=s_{\lambda/\mu}(X).
 \label{csm}
\end{align}
See \cite[(4.14), p.~324]{MG95}.

\begin{lemma}\label{la3}
One has
\begin{align}
 \sum_{\lambda\in\YY}P_{\lambda}(X;q,t)Q_{\lambda}(Y;q,t)
 &=
 \prod_{i,j}\frac{(tx_iy_j;q)_{\infty}}{(x_iy_j;q)_{\infty}}
 =: \Pi(X,Y;q,t),
 \label{ip1}\\
 \sum_{\lambda\in\YY}P_{\lambda}(X;q,t)P_{\lambda'}(Y;t,q)
 &=
 \sum_{\lambda\in\YY}Q_{\lambda}(X;q,t)Q_{\lambda'}(Y;t,q)
 =
 \prod_{i,j}(1+x_iy_j).
 \label{ip2}
\end{align}
\end{lemma}

\begin{proof}
Equation~\eqref{ip1} is the Cauchy identity; see \cite[(4.13), p.~324]{MG95}. Equation~\eqref{ip2} is the dual Cauchy identity; see \cite[(5.4), p.~330]{MG95}.
\end{proof}

\begin{lemma}\label{la5}
Let $\{a_k\}_{k\ge 1}$ and $\{b_k\}_{k\ge 1}$ be sequences in completed graded algebras such that
\begin{align*}
 \operatorname{ldeg}(a_k)\to\infty,
 \qquad
 \operatorname{ldeg}(b_k)\to\infty
 \qquad (k\to\infty).
\end{align*}
Then
\begin{align*}
 \Biggl\langle
 \exp\left(\sum_{k=1}^{\infty}\frac{a_kp_k(Y)}{k}\right),
 \exp\left(\sum_{k=1}^{\infty}\frac{b_kp_k(Y)}{k}\right)
 \Biggr\rangle_{Y;q,t}
 =
 \exp\left(\sum_{k=1}^{\infty}\frac{1-q^k}{1-t^k}\frac{a_kb_k}{k}\right),
\end{align*}
where the coefficients $a_k,b_k$ are independent of the variables in $Y$.
\end{lemma}

\begin{proof}
See \cite[Proposition~2.3]{bcgs13}.
\end{proof}

\begin{lemma}\label{la6}
The kernels in Lemma~\ref{la3} admit the exponential representations
\begin{align}
 \Pi(X,Y;q,t)
 &=
 \exp\left(\sum_{n=1}^{\infty}\frac{1-t^n}{1-q^n}\frac{p_n(X)p_n(Y)}{n}\right),
 \label{ia1}\\
 \prod_{i,j}(1+x_iy_j)
 &=
 \exp\left(\sum_{n=1}^{\infty}\frac{(-1)^{n+1}}{n}p_n(X)p_n(Y)\right).
 \label{ia2}
\end{align}
\end{lemma}

\begin{proof}
The first identity is standard; see \cite[p.~310]{MG95}. The second follows from the expansion of $\log(1+z)$.
\end{proof}

\subsection{Duality}

Let $q,t$ be algebraically independent, and write $F:=\QQ(q,t)$. For $u,v\in F$ with $v\notin\{\pm 1\}$, define the $F$-algebra endomorphism $\omega_{u,v}:\Lambda_X\otimes F\to\Lambda_X\otimes F$ by
\begin{align}
 \omega_{u,v}(p_r)=(-1)^{r+1}\frac{1-u^r}{1-v^r}p_r,
 \qquad r\ge 1.
 \label{dda}
\end{align}
Equivalently,
\begin{align*}
 \omega_{u,v}(p_{\lambda})
 =
 \epsilon_{\lambda}
 p_{\lambda}
 \prod_{i=1}^{l(\lambda)}\frac{1-u^{\lambda_i}}{1-v^{\lambda_i}},
 \qquad
 \epsilon_{\lambda}=(-1)^{|\lambda|-l(\lambda)}.
\end{align*}
In particular, $\omega_{v,u}=\omega_{u,v}^{-1}$.

\begin{lemma}
One has
\begin{align*}
 \omega_{q,t}\Pi(X,Y;q,t)=\Pi_0(X,Y):=\prod_{i,j}(1+x_iy_j),
\end{align*}
and
\begin{align}
 \omega_{q,t}P_{\lambda/\mu}(X;q,t)
 &=Q_{\lambda'/\mu'}(X;t,q),
 \label{omp}\\
 \omega_{q,t}Q_{\lambda/\mu}(X;q,t)
 &=P_{\lambda'/\mu'}(X;t,q).
 \label{omq}
\end{align}
\end{lemma}

\begin{proof}
See \cite[(2.18), p.~313]{MG95} for the first identity and \cite[(7.16)--(7.17), p.~347]{MG95} for \eqref{omp}--\eqref{omq}.
\end{proof}

\subsection{The Negu\c t operator and auxiliary contour lemmas}

\begin{definition}
Let $k\in\ZZ_{>0}$ and $q,t>0$. For an analytic symmetric function
\begin{align*}
 F(X)=\sum_{\lambda\in\YY}c_{\lambda}P_{\lambda}(X;q,t),
\end{align*}
define the Macdonald difference operator $D_{-k,X;q,t}$ by
\begin{align}
 D_{-k,X;q,t}F(X)
 :=
 \sum_{\lambda\in\YY}
 c_{\lambda}
 \left[
 (1-t^{-k})\sum_{i=1}^{l(\lambda)}(q^{\lambda_i}t^{-i+1})^k+t^{-kl(\lambda)}
 \right]P_{\lambda}(X;q,t).
 \label{ngt}
\end{align}
\end{definition}

Let $W=(w_1,\ldots,w_k)$ be an ordered set of variables. Define
\begin{align}
 D(W;q,t)
 :=
 \frac{(-1)^{k-1}}{(2\pi\mathbf i)^k}
 \frac{\displaystyle\sum_{i=1}^{k}\frac{w_kt^{k-i}}{w_iq^{k-i}}}
      {\displaystyle\prod_{i=1}^{k-1}\left(1-\frac{tw_{i+1}}{qw_i}\right)}
 \prod_{1\le i<j\le k}
 \frac{\left(1-\frac{w_i}{w_j}\right)\left(1-\frac{qw_i}{tw_j}\right)}
      {\left(1-\frac{w_i}{tw_j}\right)\left(1-\frac{qw_i}{w_j}\right)}
 \prod_{i=1}^{k}\frac{dw_i}{w_i},
 \label{ddf}
\end{align}
and
\begin{align}
 L(W,X;q,t)
 :=
 \prod_{i=1}^{k}\prod_{x\in X}\frac{w_i-\frac{q}{t}x}{w_i-qx}.
 \label{dh}
\end{align}

\begin{proposition}\label{pa2}
Assume either $q,t\in(0,1)$ or $q,t\in(1,\infty)$. Let $f:\CC\to\CC$ be analytic in a neighborhood of $0$ with $f(0)\neq 0$, and let $g:\CC\to\CC$ be analytic in a neighborhood of $0$ such that
\begin{align*}
 g(z)=\frac{f(z)}{f(q^{-1}z)}
\end{align*}
for $z$ sufficiently close to $0$. Then
\begin{align}
 D_{-k,X;q,t}\left(\prod_{x\in X}f(x)\right)
 =
 \left(\prod_{x\in X}f(x)\right)
 \oint\cdots\oint D(W;q,t)L(W,X;q,t)\prod_{i=1}^{k}g(w_i).
 \label{ss}
\end{align}
Here the integration contours are positively oriented simple closed curves such that:
\begin{enumerate}
 \item all contours lie in a common neighborhood of $0$ where both $f$ and $g$ are analytic;
 \item each contour encloses $0$ and all points $qx$ with $x\in X$, but no other singularities of the integrand;
 \item if $q,t\in(0,1)$, then the contour for $w_i$ lies inside the contour for $tw_{i+1}$ for every $i\in[k-1]$;
 \item if $q,t\in(1,\infty)$, then the contour for $w_i$ lies inside the contour for $q^{-1}w_{i+1}$ for every $i\in[k-1]$.
\end{enumerate}
\end{proposition}

\begin{proof}
This is a slightly more general form of \cite[Proposition~4.10]{GZ16}; see also \cite[Proposition~A.2]{LV21}.
\end{proof}

\begin{lemma}\label{al51}
Let $\theta\in(0,\pi)$ and $\xi>0$. Define
\begin{align*}
 R_{\epsilon,\theta,\xi}
 :=
 \bigl\{w\in\CC:\operatorname{dist}(w,[1,\infty))\le \xi\bigr\}
 \cap
 \bigl\{w\in\CC:|\arg(w-(1-\epsilon))|\le \theta\bigr\}.
\end{align*}
Let $\alpha>0$ and suppose $N(\epsilon)\in\ZZ_{>0}$ satisfies
\begin{align*}
 \limsup_{\epsilon\to 0}\epsilon N(\epsilon)>0.
\end{align*}
Then for every fixed $\theta\in(0,\pi)$ and $\xi>0$ one has
\begin{align*}
 \frac{(z;e^{-\epsilon})_{N(\epsilon)}}{(e^{-\epsilon\alpha}z;e^{-\epsilon})_{N(\epsilon)}}
 =
 \left(\frac{1-z}{1-e^{-\epsilon N(\epsilon)}z}\right)^{\alpha}
 \exp\left(O\left(\frac{\epsilon\min\{|z|,|z|^2\}}{|1-z|}\right)\right),
\end{align*}
uniformly for $z\in\CC\setminus R_{\epsilon,\theta,\xi}$ as $\epsilon\to 0$.
\end{lemma}

\begin{proof}
See \cite[Lemma~5.7]{Ah18}.
\end{proof}

\begin{lemma}\label{lb2}
Let $d,h,k\in\ZZ_{>0}$, and let $f,g_1,\ldots,g_d$ be meromorphic functions with all poles contained in $\{z_1,\ldots,z_h\}$. Then for $k\ge 2$,
\begin{align*}
 &\frac{1}{(2\pi\mathbf i)^k}
 \oint\cdots\oint
 \frac{1}{(v_2-v_1)\cdots(v_k-v_{k-1})}
 \prod_{j=1}^{d}\left(\sum_{i=1}^{k}g_j(v_i)\right)
 \prod_{i=1}^{k}f(v_i)\,dv_i\\
 &\qquad=
 \frac{k^{d-1}}{2\pi\mathbf i}
 \oint f(v)^k\prod_{j=1}^{d}g_j(v)\,dv,
\end{align*}
where all contours enclose $\{z_1,\ldots,z_h\}$, and on the left-hand side the contour for $v_i$ is required to lie inside the contour for $v_j$ whenever $i<j$.
\end{lemma}

\begin{proof}
See \cite[Corollary~A.2]{GZ16}.
\end{proof}

\section*{Statements and Declarations}

\textbf{Corresponding Author.}
Correspondence should be addressed to Zhongyang Li at \texttt{zhongyang.li@uconn.edu}.

\textbf{Competing Interests.}
The authors declare that they have no competing interests.

\textbf{Data Availability.}
Data sharing is not applicable to this article, as no datasets were generated or analyzed.

\textbf{Author Contributions.}
The research was initiated and led by Z.L., who developed
the main results and wrote the manuscript. K.S. contributed to the analytic derivation of
key identities in Section 3 and assisted with reviewing the full draft.

\bibliographystyle{amsplain}
\bibliography{fb}

\end{document}